\documentclass[10pt,reqno]{amsart}

\usepackage{amssymb, amsmath, amsthm, amsfonts, mathtools}
\usepackage[foot]{amsaddr}
\usepackage[all]{xy}
\usepackage{enumerate}
\usepackage{mathrsfs}
\usepackage{todonotes}
\usepackage{hyperref} 
\usepackage{xcolor}
\usepackage{tikz-cd}
\usepackage{physics}
\usepackage{bm}
\usepackage{comment}

\usepackage{hyperref}
  \definecolor{dark-red}{rgb}{0.6,0.15,0.15}
   \definecolor{dark-blue}{rgb}{0.15,0.15,0.6}
   \definecolor{medium-blue}{rgb}{0,0,0.5}

\hypersetup{
    colorlinks, 
    linkcolor=dark-red,
    citecolor=dark-blue, urlcolor=medium-blue
}

\usepackage[nameinlink,capitalise,noabbrev]{cleveref}

\numberwithin{equation}{section}

\theoremstyle{plain} 

\newtheorem{thm}{Theorem}[section]

\newtheorem{cor}[thm]{Corollary}

\newtheorem{prop}[thm]{Proposition}

\newtheorem{lem}[thm]{Lemma}
\newtheorem*{thm*}{Theorem}
\newtheorem*{cor*}{Corollary}
\newtheorem*{prop*}{Proposition}

\theoremstyle{definition}
\newtheorem{defn}[thm]{Definition}

\newtheorem{ex}[thm]{Example} 
 
\newtheorem{rem}[thm]{Remark}

\makeatletter
\let\c@equation=\c@equation
\let\c@lem=\c@thm
\let\c@cor=\c@thm
\let\c@conj=\c@thm
\let\c@prop=\c@thm
\let\c@lem=\c@thm
\let\c@defn=\c@thm
\let\c@notation=\c@thm
\let\c@note=\c@thm
\let\c@exmp=\c@thm
\let\c@ex=\c@thm
\let\c@exmps=\c@thm
\let\c@rem=\c@thm
\let\c@warn=\c@thm
\let\c@claim=\c@thm
\let\c@convention=\c@thm
\let\c@conventions=\c@thm
\let\c@quest=\c@thm
\let\c@facts=\c@thm
\makeatother

\newcommand{\N}{\mathbb{N}}
\newcommand{\Q}{\mathbb{Q}}

\DeclareMathOperator{\id}{id}

\DeclareMathOperator{\Aut}{\mathsf{Aut}}
\DeclareMathOperator{\Iso}{\mathsf{Iso}}
\DeclareMathOperator{\Unitary}{\mathsf{U}}

\DeclareMathOperator{\GL}{\mathsf{GL}}

\renewcommand{\Tr}{\mathrm{Tr}}

\newcommand{\Z}{\mathbb{Z}}

\newcommand{\C}{\mathbb{C}}
\newcommand{\R}{\mathbb{R}}

\newcommand{\fA}{\mathfrak{A}}
\newcommand{\fB}{\mathfrak{B}}

\newcommand{\fN}{\mathfrak{N}}

\newcommand{\fR}{\mathfrak{R}}

\newcommand{\fC}{\mathfrak{C}}

\newcommand{\frL}{\mathfrak{L}}

\newcommand{\cC}{\mathcal{C}}
\newcommand{\cH}{\mathcal{H}}

\newcommand{\cE}{\mathcal{E}}
\newcommand{\cF}{\mathcal{F}}

\newcommand{\x}{\mathbf{x}}

\newcommand{\pset}{\wp}

\renewcommand{\Im}{\operatorname{Im}}

\newcommand{\sA}{\mathscr{A}}
\newcommand{\sC}{\mathscr{C}}
\newcommand{\sD}{\mathscr{D}}
\newcommand{\sF}{\mathscr{F}}
\newcommand{\sH}{\mathscr{H}}

\newcommand{\sL}{\mathscr{L}}
\newcommand{\sS}{\mathscr{S}}
\newcommand{\sP}{\mathscr{P}}
\newcommand{\sT}{\mathscr{T}}

\newcommand{\hilbH}{\mathcal{H}}
\newcommand{\hilbK}{\mathcal{K}}

\newcommand{\agnes}[1]{\todo[color=cyan]{#1}}
\newcommand{\daniel}[1]{\todo[color=red]{#1}}

\DeclareMathAlphabet{\mathboondoxcal}{U}{BOONDOX-cal}{m}{n}
\SetMathAlphabet{\mathboondoxcal}{bold}{U}{BOONDOX-cal}{b}{n}
\DeclareMathAlphabet{\mathbboondoxcal} {U}{BOONDOX-cal}{b}{n}

\newcommand{\scrj}{\mathboondoxcal{j}}
\newcommand{\scrk}{\mathboondoxcal{k}}
\newcommand{\scrl}{\mathboondoxcal{l}}

\makeatletter
\DeclareFontFamily{OMX}{MnSymbolE}{}
\DeclareSymbolFont{MnLargeSymbols}{OMX}{MnSymbolE}{m}{n}
\SetSymbolFont{MnLargeSymbols}{bold}{OMX}{MnSymbolE}{b}{n}
\DeclareFontShape{OMX}{MnSymbolE}{m}{n}{
    <-6>  MnSymbolE5
   <6-7>  MnSymbolE6
   <7-8>  MnSymbolE7
   <8-9>  MnSymbolE8
   <9-10> MnSymbolE9
  <10-12> MnSymbolE10
  <12->   MnSymbolE12
}{}
\DeclareFontShape{OMX}{MnSymbolE}{b}{n}{
    <-6>  MnSymbolE-Bold5
   <6-7>  MnSymbolE-Bold6
   <7-8>  MnSymbolE-Bold7
   <8-9>  MnSymbolE-Bold8
   <9-10> MnSymbolE-Bold9
  <10-12> MnSymbolE-Bold10
  <12->   MnSymbolE-Bold12
}{}

\let\llangle\@undefined
\let\rrangle\@undefined
\DeclareMathDelimiter{\llangle}{\mathopen}%
                     {MnLargeSymbols}{'164}{MnLargeSymbols}{'164}
\DeclareMathDelimiter{\rrangle}{\mathclose}%
                     {MnLargeSymbols}{'171}{MnLargeSymbols}{'171}
\makeatother

\newcommand{\ltrans}{|\!\langle}
\newcommand{\rtrans}{\rangle\!|}

\usepackage{mathtools}

\usepackage[normalem]{ulem}

\newcommand{\ssH}{\mathscr{H}}

\newcommand{\frA}{\mathfrak{A}}
\newcommand{\frB}{\mathfrak{B}}

\newcommand{\fF}{\mathfrak{F}}
\newcommand{\frT}{\mathfrak{T}}
\newcommand{\cA}{\mathcal{A}}
\newcommand{\cB}{\mathcal{B}}
\newcommand{\cG}{\mathcal{G}}
\newcommand{\cK}{\mathcal{K}}

\newcommand{\cN}{\mathfrak{N}}
\newcommand{\cO}{\mathcal{O}}
\newcommand{\cU}{\mathcal{U}}
\newcommand{\cV}{\mathcal{V}}
\newcommand{\bbK}{\mathbb{K}}
\newcommand{\bbB}{\mathbb{B}}
\newcommand{\bbS}{\mathbb{S}}
\newcommand{\bbP}{\mathbb{P}}
\newcommand{\bbR}{\mathbb{R}}
\newcommand{\bbZ}{\mathbb{Z}}
\newcommand{\bbN}{\mathbb{N}}
\newcommand{\tn}[1]{\textup{#1}}

\newcommand{\be}{\mathbf{e}}
\renewcommand{\r}{\mathbf{r}}
\newcommand{\s}{\mathbf{s}}
\renewcommand{\u}{\mathbf{u}}

\renewcommand{\x}{\mathbf{x}}
\newcommand{\y}{\mathbf{y}}
\newcommand{\z}{\mathbf{z}}
\newcommand{\bbC}{\mathbb{C}}
\newcommand{\defeq}{\vcentcolon=}
\newcommand{\vecspan}{\tn{span}}
\DeclareMathOperator{\Log}{Log}
\newcommand{\Un}{\mathsf{U}}

\newcommand{\ssec}{\operatorname{Sec}}

\newcommand{\net}{\mathscr{I}}
\newcommand{\topology}{\mathscr{O}}
\newcommand{\locreg}{\mathscr{L}}

\newcommand{\projAStar}{p\raisebox{1pt}{$_{{}_{\fA^*}}$\!}}
\newcommand{\projAPure}{p\raisebox{1pt}{$_{{}_{\sP(\fA)}}$\!}}
\newcommand{\projGNS}{p_{_{\sH}}}
\newcommand{\projLine}{p_{_{\sL}}}

\newcommand{\Rho}{\mathrm{P}}

\newcommand{\colim}{\operatornamewithlimits{colim}}
\DeclareMathOperator{\diag}{diag}

\newcommand{\hide}[1]{}

\title[Continuity of Kadison Transitivity and of the GNS Construction]{Continuous Dependence on the Initial Data in the Kadison Transitivity Theorem and GNS Construction}
\author[Spiegel]{Daniel Spiegel$^{1,2,3}$}
\author[Moreno]{Juan Moreno$^{1}$}
\author[Qi]{Marvin Qi$^{2,3}$}
\author[Hermele]{Michael Hermele$^{2,3}$}
\author[Beaudry]{Agn\`es Beaudry$^{1}$}
\author[Pflaum]{Markus J. Pflaum$^{1,3}$}
\address{$^{1}$Department of Mathematics, University of Colorado, Boulder, CO 80309, USA}
\address{$^{2}$Department of Physics, University of Colorado, Boulder, CO 80309, USA}
\address{$^{3}$Center for Theory of Quantum Matter, University of Colorado,  \newline
  \mbox{ }\hspace{4.5mm}Boulder, CO 80309, USA}
\begin{document}

\begin{abstract}
We consider how the outputs of the Kadison transitivity theorem and Gelfand-Naimark-Segal construction may be obtained in families when the initial data are varied. More precisely, for the Kadison transitivity theorem, we prove that for any nonzero irreducible representation $(\cH, \pi)$ of a $C^*$-algebra $\fA$ and $n \in \bbN$, there exists a continuous function $A:X \rightarrow \fA$ such that $\pi(A(\x, \y))x_i = y_i$ for all $i \in \qty{1, \ldots, n}$, where $X$ is the set of pairs of $n$-tuples $(\x, \y) \in \cH^n \times \cH^n$ such that the components of $\x$ are linearly independent. Versions of this result where $A$ maps into the self-adjoint or unitary elements of $\fA$ are also presented. Regarding the Gelfand-Naimark-Segal construction, we prove that given a topological $C^*$-algebra fiber bundle $p:\fA \rightarrow Y$, one may construct a topological fiber bundle $\sP(\fA) \rightarrow Y$ whose fiber over $y \in Y$ is the space of pure states of $\fA_y$ (with the norm topology), as well as bundles $\sH \rightarrow \sP(\fA)$ and $\mathscr{N} \rightarrow \sP(\fA)$  whose fibers $\sH_\omega$ and $\mathscr{N}_\omega$ over $\omega \in \sP(\fA)$ are the GNS Hilbert space and closed left ideal, respectively, corresponding to $\omega$. When $p:\fA \rightarrow Y$ is a smooth fiber bundle, we show that $\sP(\fA) \rightarrow Y$ and $\sH\rightarrow  \sP(\fA)$ are also smooth fiber bundles; this involves proving that the group of $*$-automorphisms of a $C^*$-algebra is a Banach-Lie group. In service of these results, we review the topology and geometry of the pure state space. A simple non-interacting quantum spin system is provided as an example illustrating the physical meaning of some of these results.
\end{abstract}

\maketitle
\tableofcontents

\newpage

\section*{Introduction}

The primary goal of this paper is to detail how one can do classical maneuvers in the theory of $C^*$-algebras, namely the Kadison transitivity theorem and the Gelfand-Naimark-Segal (GNS) construction, in a way that depends continuously or smoothly on the input data. In the Kadison transitivity theorem, the initial data is a set of vectors in the Hilbert space of an irreducible representation of a $C^*$-algebra $\fA$; these may be taken to represent pure states on $\fA$. In the GNS construction the initial datum is a state on $\fA$ which, again, we will usually take to be pure. We therefore find it necessary and appropriate to hold a second goal in service of the first: to review and elaborate on the topology and geometry of the pure state space of $\fA$.

We are inspired by the connection of $C^*$-algebras to quantum many-body physics. In quantum systems with infinitely many degrees of freedom, one represents observable quantities as self-adjoint elements of a $C^*$-algebra and quantum states as states on the $C^*$-algebra, i.e., normalized positive linear functionals. These $C^*$-algebras typically have a quasi-local structure. Intuitively, if one continuously deforms such a quantum system in a local region, the state of the system is expected to change continuously with respect to the norm topology, while a global deformation of the system yields merely weak$^*$ continuity of the state. This intuition is developed in \S\ref{subsec:1d_system} where a trivial example of a parametrized quantum system is investigated from the point of view of topology. In this paper we focus on the case of norm-continuity since it is more tractable mathematically, but comparisons are made to the weak$^*$ topology in \S\ref{subsec:connectedness} and \S\ref{subsec:1d_system}.

This paper begins with a review of classical results on the topology of the state space in \S\ref{sec:topstructure}, unraveling the relationship between the theory of superselection sectors and
topological properties of the pure state space. In \S\ref{sec:diffstructure}, we turn to the geometry of the pure state space.  We do not present any new theorems in \S\ref{sec:topstructure} or \S\ref{sec:diffstructure}, but rather provide a self-contained and complete presentation of results that are otherwise scattered in the literature. In \S\ref{sec:background-projective-hilbert-spaces} we detail the complex manifold structure of projective Hilbert space and show how to pass from the pure state space to a projective Hilbert space using an irreducible representation. In \S\ref{sec:diffstructure}, we prove that the pure state space carries the structure of a K\"ahler manifold, a result
which goes back to \cite{CirelliPureStatesQMKaehlerBundles}:

\begin{thm*}
  The pure state space $\sP(\fA)$ of a $C^*$-algebra $\fA$ carries in a natural way
  the structure of a (possibly infinite-dimensional) K\"ahler manifold.
  The underlying topology is given by the norm topology. Each connected component is
  open and given by the set  $\sP_\pi(\fA)$ of vector states of some 
  irreducible representation $(\hilbH, \pi)$ of $\fA$.  
  The set $\sP_\pi(\fA)$ carries a unique complex manifold structure such that the
  canonical map
  $r: \bbP\hilbH \to \sP_\pi(\fA)$ given by $r(\C\Psi) (A) =  \ev{\Psi, \pi(A)\Psi}$ for
  $\Psi \in \bbS\hilbH$ and $A \in \fA$ is biholomorphic.
  The hermitian metric $h$ is uniquely determined by the requirement that, for every irreducible representation $(\hilbH, \pi)$ of $\fA$, the canonical projection
  $p_{\bbS \hilbH}:\bbS \hilbH \rightarrow \sP_\pi(\fA)$, $\Psi \mapsto r(\C\Psi)$ 
  is a riemannian submersion. 
\end{thm*}

A review of infinite dimensional manifolds and fiber bundles is provided in the appendix.

The materials of \S\ref{sec:topstructure} and \S\ref{sec:diffstructure} give us the control we need to address our primary goal. In \S\ref{sec:contkadison}, we establish a generalization of the Kadison transitivity theorem \cite[Thm.~5.2.2]{MurphyCAOT} that admits selections of operators that depend continuously on the initial data. This is a new result which we call the \emph{continuous Kadison transitivity theorem}.  

 \begin{thm*}[Continuous Kadison Transitivity]
Let $\fA$ be a $C^*$-algebra and let $(\cH, \pi)$ be an irreducible representation. Let $n$ be a positive integer and let
\begin{align*}
X &= \qty{(\x, \y) \in \cH^{2n}: x_1,\ldots, x_n \tn{ are linearly independent}}, 
\end{align*}
equipped with the subspace topology inherited from $\cH^{2n}$, where $\x= (x_1,\ldots,x_n)$ and $\y=(y_1, \ldots, y_n)$. There exists a continuous map $A \colon X \rightarrow \fA$ such that 
\[
\pi(A(\x,\y))x_i = y_i
\]
for all $(\x, \y) \in X$ and $i=1,\ldots, n$. 

On  the subspace $X_\tn{sa} \subset X$ of pairs $(\x, \y) $ such that there is a self-adjoint $T \in \fB(\cH)$ with $Tx_i = y_i $ for $i =1,\ldots, n$, there is a continuous map  $A \colon X_\tn{sa} \rightarrow \fA_\tn{sa}$ satisfying the same property. If $\fA$ is unital, then on the subspace $X_\tn{u}$ of pairs $(\x, \y) $ such that there is a unitary $T \in \Unitary(\cH)$ with $Tx_i = y_i $ for $i =1,\ldots, n$, every point $(\x_0, \y_0)$ has a  neighborhood $O \subset X_\tn{u}$ for which there exists a continuous map  $A\colon O \rightarrow \Unitary(\fA)$ which agains satisfies the same property.
\end{thm*}

This theorem states the existence of a continuous selection for the function $X \rightarrow \wp(\fA)$ mapping a point $(\x, \y) \in X$ to the set of all $A \in \fA$ satisfying $\pi(A)x_i = y_i$ for all $i=1,\ldots, n$. Thus, the key ingredient in the proof is the Michael selection theorem \cite[Thm.~3.2$''$]{MichaelSelection}, as it provides conditions under which such a selection may be found. A consequence of the continuous Kadison transitivity theorem is that local trivializations may be found for the action of the unitary group of $\fA$ on a fixed pure state. This is studied in \S\ref{sec:principal_fiber_bundles}.

\begin{cor*}
  Let $\fA$ be a unital $C^*$-algebra and let $\omega \in \sP(\fA)$
  be a pure state. The
  map $p_{\Unitary(\fA)}: \Unitary(\fA) \rightarrow \sP_\omega(\fA)$
  defined by $p_{\Unitary(\fA)}(U) (A) = \omega(U^*AU)$
  then is a locally trivial principal $\Unitary_\omega(\fA)$-bundle, where $\sP_\omega(\fA)$ is the set of states of the form $p_{\Unitary(\fA)}(U)$ and
  $\Unitary_\omega(\fA)$ is the isotropy group
  $\qty{U \in \Unitary(\fA): p_{\Unitary(\fA)}(U) = \omega}$.
\end{cor*}

By showing that $\pi_1(\Unitary(\fA)) \not\cong \pi_1(\Unitary_\omega(\fA) \times \sP_\omega(\fA))$, we show that this bundle is nontrivial in a few examples. The most interesting example considered is when $\fA$ is a UHF algebra. While the homotopy groups of the unitary group of a  UHF algebra are known
\cite{SchroederUHFalgebras},
we present a new method for computing these groups, relying on a theorem of Gl\"ockner \cite[Thm.~1.13]{glockner2010homotopy}. This method also allows for the calculation of the homotopy groups of $\Unitary_\omega(\fA)$ in Theorem \ref{thm:UHF_Uomega}, which to our knowledge have not been previously computed. 

In \S\ref{sec:fiberwiseGNS} we turn our attention to the GNS construction. There, ``continuous dependence'' of the GNS construction on its initial data is construed as the ability to create topological fiber bundles out of the Hilbert spaces and left ideals associated to a norm-continuous family of pure states. Precisely, given a topological $C^*$-algebra bundle $p_\fA:\fA \rightarrow X$, we construct the fiber bundle of pure state spaces $p_{\sP(\fA)}:\sP(\fA) \rightarrow X$ and show that the sets
\[
\sH = \bigsqcup_{\omega \in \sP(\fA)} \hilbH_\omega \qqtext{and} \mathscr{N} = \bigsqcup_{\omega \in \sP(\fA)} \fN_\omega
\]
have natural structures as fiber bundles over $\sP(\fA)$, where $\fN_\omega$ is the left ideal associated to a pure state $\omega$ and $\hilbH_\omega = \fA_\omega /\fN_\omega$ the corresponding Hilbert space. We call this the \textit{fiberwise GNS construction}. If one is given a preferred family of pure states, i.e.\ a section $\omega:X \rightarrow \sP(\fA)$, then $\sH$ and $\mathscr{N}$ may be pulled back to obtain bundles $\omega^*\sH$ and $\omega^*\mathscr{N}$ over $X$. The construction of $\mathscr{N}$ relies on the continuous Kadison transitivity theorem, while the construction of $\sH$ relies on Proposition \ref{prop:continuous-dependancy-unitary-automorphism}, reproduced in a simplified form below.


\begin{prop*}
  Let $\fB$ be a $C^*$-algebra with irreducible representation $(\hilbH, \pi)$. For each
  unit vector $\Omega$ denote by  $\xi_\Omega$ the projection $\fB \to \hilbH$, $B\mapsto \pi (B)\Omega$
  and by $\omega_\Omega$ the state $B \mapsto \langle \Omega,\pi(B)\Omega\rangle$.
  Let $O\subset \Aut (\fB) \times \bbS \hilbH$ be the set of all pairs $(\alpha,\Omega)$ for which there exists a unique
  vector $\Phi = \Phi (\alpha,\Omega) \in \bbS \hilbH$ such that
  $\omega_\Omega\circ \alpha^{-1} = \omega_{\Phi}$
  and $\langle \Omega,\Phi \rangle > 0$.
  Then 
  the map $U:O \rightarrow \Unitary(\hilbH)$ which associates to each pair $(\alpha,\Omega)$
  the unique unitary $U_{\alpha,\Omega}$ making the diagram
  \[
  \begin{tikzcd}
  \fB \arrow[r,"\alpha"] \arrow[d,"\xi_\Omega"']& \fB \arrow[d,"\xi_{\Phi (\alpha,\Omega)}"]\\
  \hilbH \arrow[r,"U_{\alpha,\Omega}"']& \hilbH 
  \end{tikzcd}
  \]
  commute is continuous with respect to the norm topologies on $\Aut (\fB)$ and $\Unitary (\hilbH)$. 
\end{prop*}

In \S\ref{sec:smooth_case} we show that when $p_{\fA}:\fA \rightarrow X$ is a smooth $C^*$-algebra bundle, the bundles $p_{\sP(\fA)}:\sP(\fA) \rightarrow X$ and $\sH$ are smooth as well. The transition functions of the bundle $p_{\fA}:\fA \rightarrow X$ map into the automorphism group of the model fiber, thus our definition of a smooth $C^*$-algebra bundle relies on the fact that the automorphism group of a $C^*$-algebra carries the structure of a Banach-Lie group. We provide this structure in Proposition \ref{prop:Aut_Lie_group}. In Proposition \ref{prop:smooth-dependancy-unitary-automorphism} we prove that the map $U:O \rightarrow \Unitary(\cH)$ in the proposition displayed above is in fact smooth, and this leads to a smooth structure on $\sH$. We do not have a smooth analog of the continuous Kadison transitivity theorem, so we do not endow $\mathscr{N}$ with a smooth structure.

We conclude in \S\ref{sec:examples} with a few simple examples of parametrized quantum systems. In section \S\ref{subsec:0d_example} we consider a finite-dimensional $C^*$-algebra and a family of states over $S^2$ representing the ground states of a single spin-$\frac{1}{2}$ particle in a rotatable magnetic field. We show that the first Chern class of the line bundle of ground states of this system, corresponding to the Berry curvature 2-form on $S^2$ \cite{Berry}, may be recovered from the first Chern class of the Hilbert bundle obtained in the fiberwise GNS construction. In section \S\ref{subsec:1d_system} we consider a non-interacting quantum system in one spatial dimension obtained by copying the above system at each point of $\bbZ$. In this case the $C^*$-algebra is infinite-dimensional, hence the weak$^*$ and norm topology on $\sP(\fA)$ are distinct, and we find it interesting to study which types of continuity arise in the family of ground states and under what conditions even in this physically trivial system. We find that it is possible to obtain norm-continuity of the ground states if one enlarges the natural parameter space and uses a finer topology there.

\subsection*{Acknowledgements}
This material is based upon work supported by the National Science Foundation
under Grant No.~DMS 2055501 awarded to A.~Beaudry, M.~Hermele and M.J.~Pflaum.
The research of M.~Qi is supported by the NDSEG program.
A.~Beaudry, M.~Hermele and M.J.~Pflaum also acknowledge support by
a University of Colorado seed grant. 

%

\section{Topological structures on spaces of states and representations}\label{sec:topstructure}

\subsection{Notational preliminaries}
The following notational conventions will be used throughout this paper. The complex conjugate of a number $z \in \C$ will
be denoted both by $z^*$ or  $\overline{z}$. 
Hilbert spaces are always understood over the field of complex numbers with inner product being linear in the second
variable. Representations of a $C^*$-algebra $\fA$ will be denoted as pairs $(\hilbH,\pi)$ where
$\hilbH$ is the Hilbert space on which $\fA$ acts and $\pi: \fA \to \fB (\hilbH)$ is the representing $C^*$-homomorphism.
A cyclic representation with a distinguished cyclic vector $\Omega$ will be written as a triple
$(\hilbH,\pi,\Omega)$. The unitary groups of a Hilbert space $\cH$ and a $C^*$-algebra $\fA$ will be denoted $\Unitary (\hilbH)$ and $\Unitary(\fA)$, respectively.

\subsection{Folia}
\label{sec:folia}
Let $\fA$ be a $C^*$-algebra. We denote by $\sS (\fA)\subset \fA^*$ the
space of states on $\fA$, that is, the space of normalized positive
linear functionals $\fA \to \C$. It is well known that $\sS(\fA)$ is nonempty
and convex when $\fA$ is nonzero. When $\fA$ is unital, $\sS(\fA)$ is weakly$^*$ closed,
hence weakly$^*$ compact by the Banach--Alaoglu theorem.
The norm on $\fA^*$ endows $\sS (\fA)$ in a natural way with a metric
$d_{\|\cdot\|}$ given by  
\[
d_{\|\cdot\|}(\psi, \omega) = \| \psi -\omega \|
\] 
for $\psi,\omega \in \sS (\fA)$. We will refer to this metric as the \emph{canonical metric}.
The state space $\sS (\fA)$ also carries two natural uniform structures: the
\emph{metric} or \emph{norm defined uniformity} induced by the canonical metric and the
\emph{weak$^*$ uniformity}. A basis of the latter is given by the entourages
\[
  U_{\varepsilon, A_1,\ldots,A_k}= \big\{ (\psi,\omega)\in \sS (\fA)^2 \, : \:
  |(\psi -\omega)(A_1)| < \varepsilon , \ldots, |(\psi -\omega) (A_k)| < \varepsilon \big\} ,
\]
where $\varepsilon >0$ and $A_1,\ldots , A_k \in \fA$. The metric uniformity is finer 
than the weak$^*$ uniformity. We will denote $\sS (\fA)$ endowed with the norm defined or the weak$^*$ uniformity by
$\sS (\fA)_{\textup{n}}$ and $\sS (\fA)_{\textup{w}^*}$, respectively. Later we will give a more detailed
account of these uniform spaces.

Recall that by the GNS construction every state $\omega$ gives rise to a nontrivial cyclic 
representation $(\hilbH_\omega, \pi_\omega,\Omega_\omega)$, called the \emph{GNS representation} of $\omega$.
The Hilbert space $\hilbH_\omega$ is the completion of the quotient $\fA/\fN_\omega$ by the left ideal 
\[\fN_\omega = \{A \in \fA : \: \omega (A^*A)=0\}\] 
with respect to the inner product
\[
\langle A + \fN_\omega ,B+\fN_\omega \rangle = \omega(A^*B)
\]
for all $A,B\in \fA$. Following Haag \cite{HaagLocalQP2}, we call $\fN_\omega$ the Gelfand ideal associated to  $\omega$. The representation $\pi_\omega : \fA \to \fB(\hilbH_\omega)$
is the $*$-homomorphism defined by $\pi_\omega(A)(B +\fN_\omega) = AB +\fN_\omega$.  If $\fA$ is unital and $I$
denotes the identity, the cyclic unit vector $\Omega_\omega $ coincides by  definition with $I+\fN_\omega$.
In the non-unital case, $\Omega_\omega$ is the limit of the net $(E_\lambda + \fN_\omega)_{\lambda \in \Lambda}$ for any approximate identity $(E_\lambda)_{\lambda \in \Lambda}$ in $\fA$ (all such limits exist and coincide). 

Given a non-degenerate representation $(\hilbH,\pi)$, the $\pi$-\emph{normal} states of $\fA$ are
those of the form $\omega (A) = \Tr (\varrho \pi(A))$, where $\varrho \in \fB (\hilbH)$ is a
\emph{density matrix}, that is, a positive trace class operator of trace $1$.
We denote the $\pi$-normal states of $\fA$ by $\sS_\pi (\fA)$. In particular, the vector states of $\pi$ are those states of the form $\omega(A) = \ev{\Omega, \pi(A)\Omega}$, where $\Omega \in \hilbH$ is a unit vector; these form a subset of $\sS_\pi(\fA)$.
\hide{Observe that $\Omega$ and $e^{i\theta} \Omega$ define the same state for any $\theta \in \bbR$.}

\begin{lem}\label{lem:vector_state}
Let $(\cH, \pi)$ be a non-degenerate representation of a $C^*$-algebra $\frA$. If $\Psi, \Omega \in \hilbH$ are unit vectors representing the vector states $\psi, \omega \in \sS_\pi(\frA)$, then
\[
d_{\|\cdot\|} (\psi,\omega) = \norm{\psi - \omega} \leq 2 \norm{\Psi - \Omega}.
\]
Thus, the map $\bbS\hilbH \rightarrow \sS_\pi(\frA)_{\textup{n}}$ which assigns to a unit vector
the corresponding vector state is uniformly continuous.
\end{lem}

\begin{proof}
This follows from the triangle inequality and the Cauchy-Schwarz inequality:
\begin{align*}
\norm{\psi - \omega} &= \sup_{\norm{A}\leq 1} \abs{\ev{\Psi, \pi(A)\Psi} - \ev{\Omega, \pi(A)\Omega}}\\
&\leq \sup_{\norm{A}\leq 1} \abs{\ev{\Psi, \pi(A)(\Psi - \Omega)}} + \sup_{\norm{A} \leq 1} \abs{\ev{\Psi - \Omega, \pi(A)\Omega}}\\
&\leq 2 \norm{\Psi - \Omega},
\end{align*}
where, in the last line, we have used
that $\norm{\pi} \leq 1$. 
\end{proof}

As mentioned in the proof above, a representation $\pi$ of a $C^*$-algebra always satisfies $\norm{\pi} \leq 1$. Furthermore, $\pi$ is an isometry if and only if it is faithful. When this is not the case, the following lemma may be of service. The proof of this lemma is explained as part of the proof of Lemma 4.4 in \cite{RobertsRoepstorffSBCAQT}. 

\begin{lem}\label{lem:rep_almost_isometry}
Let $\fA$ and $\fB$ be $C^*$-algebras and let $\pi:\fA \rightarrow \fB$ be $*$-homomorphism. Given $A \in \fA$ and $\varepsilon > 0$, there exists $B \in \fA$ such that $\pi(A) = \pi(B)$ and
\[
\norm{B} < \norm{\pi(A)} + \varepsilon.
\]
\end{lem} 

A strengthened version for unital $C^*$-algebras in which $\norm{B} \leq \norm{\pi(A)}$ may be found in \cite[Cor.~10.1.8]{KadRinFTOAII}, but Lemma \ref{lem:rep_almost_isometry} in the form stated above suffices for our purposes.


\begin{proof}
  Since $\ker \pi$ is a closed two-sided ideal in $\fA$, the quotient  $\fA/\ker \pi$ is a $C^*$-algebra, which
  satisfies the commutative diagram
  \begin{displaymath} 
    \begin{tikzcd}[ampersand replacement=\&]
    \fA \arrow[r,"\pi"] \arrow[d,"q"'] \& \pi(\fA)\\
    \fA/\ker \pi \arrow[ur,"\overline \pi"']
    \end{tikzcd}
  \end{displaymath}
  where $q$ is the quotient map and $\overline \pi:\fA/\ker \pi \rightarrow \pi(\fA)$ is a $*$-isomorphism, hence an
  isometry. By definition of the norm on $\fA/\ker \pi$, we then have
  \[
    \norm{\pi(A)} = \norm{\overline \pi(q(A))} = \norm{q(A)} = \inf_{C \in \ker \pi} \norm{A - C} =
    \inf_{B \in \fA}\qty{\norm{B} : \pi(A) = \pi(B)},
  \]
  and the result follows.
  \end{proof}

The next result shows that this lemma entails that the space of  normal states with respect to a representation $\pi$
can be identified with the space of normal states of the induced von Neumann algebra even if
$\pi$ is not faithful.

\begin{prop}[\emph{cf.} {\cite[Lem.~4.4]{RobertsRoepstorffSBCAQT}}]\label{prop:pullback-representation-isometric-isomorphism-normal-state-spaces}
  Let $(\hilbH,\pi)$ be a non-degenerate representation and
  $\fR = \pi(\fA)'' \subset \fB(\hilbH)$ the induced von Neumann algebra.
  Denote by $\sS_*(\fR)$ the space of normal states on $\fR$.
  Then the pullback map
  $\pi^* : \sS_* (\fR)_{\textup{n}} \to \sS_\pi(\fA)_{\textup{n}}$, $\omega \mapsto \omega \circ \pi$
  is an isometric isomorphism with respect to the canonical metrics,
  so in particular an isomorphism of the corresponding uniform spaces. 
\end{prop}


\begin{proof}
  By definition, $\pi^*$ is surjective. So it suffices to show that $\pi^*$ is an isometry.
  Clearly, for all $\omega \in \fR^*$
  \[
    \| \pi^* \omega \| = \sup\limits_{A \in \fA, \: \| A\|\leq 1}  | \omega (\pi (A)) |
    \leq \sup\limits_{B \in \fR, \: \| B\|\leq 1}  | \omega (B) | = \| \omega \| ,  
  \]
  hence $\pi^*$ is contraction. It remains to show that for all $\psi, \omega \in  \sS_* (\fR)$
  \begin{equation}
    \label{eq:pistar-norm-increasing}
    d_{\|\cdot\|} (\psi, \omega) = \| \psi - \omega \| \leq \| \pi^* \psi - \pi^* \omega \|  \ .
  \end{equation}
  To prove this recall that the state $\psi$ is induced by a density matrix which means that
  there exists a collection of vectors $(x_i)_{i \in I}$ such that
  $\sum_{i\in I} \norm{x_i}^2 = 1$ and
  $\psi (B) = \sum_{i\in I} \ev{x_i, Bx_i}$ for all $B \in \fR$.
  Hereby, each of the sums is the limit of the net of finite partial sums. 
  Likewise, there exists for $\omega$ a family
  $(y_i)_{i \in I} \subset \hilbH$ such that  $\sum_{i\in I} \norm{y_i}^2 = 1$ and
  $\omega (B) = \sum_{i\in I} \ev{y_i, By_i}$ for all $B \in \fR$.
  Now let $\varepsilon >0$ and choose $B\in \fR$ with $\norm{B} \leq 1$ such that 
  \begin{equation}
    \label{eq:canonical-metric-estimate-value-unit-ball-element}
     \| \psi - \omega  \| \leq |\psi(B) -\omega(B)| + \varepsilon \ .
  \end{equation}
  Next choose a finite subset $J\subset I$ so that
  $\sum_{i \in I\setminus J} \norm{x_i}^2 < \varepsilon $ and
  $\sum_{i\in I\setminus J} \norm{y_i}^2 < \varepsilon $. By the Kaplansky density theorem
  there exists $A^\prime \in \fA$ with $\norm{\pi(A^\prime)} \leq 1$ such that
  \[
    \sum_{j\in J} \abs{\ev{x_j, (B - \pi(A^\prime))x_j}} < \varepsilon
    \quad\text{and}\quad
    \sum_{j\in J} \abs{\ev{y_j, (B - \pi(A^\prime))y_j}} < \varepsilon \ .
  \]
  Hence
  \begin{equation}
    \label{eq:estimate-kaplansky}
    \abs{\psi (B) - \psi (\pi(A^\prime))} < 3 \varepsilon
    \quad\text{and}\quad
    \abs{\omega (B) - \omega (\pi(A^\prime))} < 3 \varepsilon \ .
  \end{equation}
  By Lemma \ref{lem:rep_almost_isometry}, there exists $A \in \fA$ such that
  $ \pi (A) = \pi (A^\prime)$ and
  \[ \| A \| < \| \pi (A^\prime)\| + \varepsilon \leq 1 +\varepsilon \ . \]
  Together with the estimates \eqref{eq:canonical-metric-estimate-value-unit-ball-element}
  and \eqref{eq:estimate-kaplansky} this finally entails
  \begin{equation*}
  \begin{split}  
    \| \psi - \omega \| & \leq |\psi (B) -\omega  (B)| + \varepsilon  <
    | \psi (\pi(A^\prime)) - \omega (\pi(A^\prime))|  + 7  \varepsilon = \\
    & = | \psi (\pi(A)) - \omega (\pi(A))| + 7  \varepsilon <
    (1 +\varepsilon) \| \pi^*\psi - \pi^* \omega \| + 7  \varepsilon \ .
  \end{split}  
  \end{equation*}
  By passing to the limit $\varepsilon \searrow 0$, the estimate 
  \eqref{eq:pistar-norm-increasing} follows.
\end{proof}

The $C^*$-algebra $\fA$ acts in a natural way  on the dual $\fA^*$ by
associating to a pair $(B,\omega)\in \fA \times \fA^*$ the continuous linear functional
$B\cdot \omega : \fA \to\C$, $A \mapsto  \omega (B^*AB) $.
Note that if $\omega$ is a state and $B$ fulfills $\omega (B^*B)=1$, then $B\cdot \omega$ is again
a state called a \emph{quasi-local perturbation} of $\omega$.  This motivates the notion of a folium, introduced by Haag--Kadison--Kastler
in \cite{HaaKadKasNCSACS} as a tool for the classification of states in local
quantum physics; see also \cite[Sec.~8.6]{LanFQT}.

\begin{defn}
  By a \emph{folium} in the state space $\sS(\fA)$ one understands a non-empty subspace  $\sF\subset \sS(\fA)$ which is
  \begin{enumerate}[(F1)]
  \item\label{ite:folium-norm-closed} norm closed, 
  \item\label{ite:folium-convex} convex, and
  \item\label{ite:folium-invariance} invariant under the action of $\fA$ in the sense that if
    $\omega \in \sF$ and $B \in \fA$ with  $\omega (B^*B)=1$, then the quasi-local perturbation
    $B\cdot\omega$ lies again in $\sF$.
  \end{enumerate}
\end{defn}

An important observation of Haag--Kadison--Kastler in \cite[\S 1]{HaaKadKasNCSACS}, summarized below, is that the $\pi$-normal states of a non-degenerate representation form
a folium. The proof is omitted in \cite{HaaKadKasNCSACS}, so we provide it;
cf.\ also \cite[Lemma 5.6]{StormerPOVSDASC}.

\begin{thm}
  For every nonzero non-degenerate representation $(\hilbH,\pi)$ of a $C^*$-algebra $\fA$, the space of $\pi$-normal states $\sS_\pi(\fA)$ is the smallest folium containing the vector states of $\pi$. Furthermore, any folium coincides with the space $\sS_\pi(\fA)$ for some non-degenerate representation $\pi$ of $\fA$.
\end{thm}


\begin{proof}
  We first show that for every nonzero non-degenerate representation $\pi$ the
  space $\sS_\pi(\fA)$ is a folium. 
  Denote by  $\fR = \pi(\fA)'' \subset \fB(\hilbH)$ the von Neumann algebra induced by
  the representation $\pi$. According to
  \cite[7.1.13]{KadisonRingroseII}, the space $\sS_* (\fR)$ of normal states on $\fR$
  is norm closed in $\fR^*$. In particular  $\sS_* (\fR)$ is then complete
  in the metric uniformity. 
  Since the pullback $\pi^*: \sS_* (\fR) \to \sS_\pi(\fA)$ is an isometric isomorphism by
  Proposition \ref{prop:pullback-representation-isometric-isomorphism-normal-state-spaces},
  its image has to be complete as well. Hence $\sS_\pi(\fA)$ is closed, proving
  (F\ref{ite:folium-norm-closed}).
  
  The set of density matrices on $\fB(\hilbH)$ is convex, hence $\sS_\pi(\fA)$ is so, too,
  and (F\ref{ite:folium-convex}) holds. Now let $\omega \in \sS_\pi(\fA)$ and 
  $\varrho \in \fB (\hilbH)$ a density matrix such that $\omega (A) = \Tr (\varrho \pi(A))$
  for $ A \in \fA$. Let $B \in \fA$ such that $\omega (B^*B)=1$. Then
  the operator $\pi (B) \varrho \pi(B)^*$ is self-adjoint, has trace $1$
  by the equality $\Tr \big(\pi (B) \varrho \pi(B)^*\big) = \omega (B^*B)=1$ and is positive
  since $\langle v , \pi (B)\varrho\pi(B)^*v\rangle =
  \langle\pi (B)^* v , \varrho\pi(B)^*v\rangle\geq 0$ for all $v\in \hilbH$.
  Moreover, 
  \[
    B \cdot \omega (A) =\omega (B^*AB) = \Tr \big( \left(\pi (B) \varrho \pi(B)^*\right) \pi(A) \big)
     \quad \text{for } A \in \fA \ ,
  \]
  which shows that $B\cdot \omega $ is a $\pi$-normal state and  (F\ref{ite:folium-invariance})
  is fulfilled. Hence  $\sS_\pi(\fA)$ is a folium. 

  Now let $\sF$ be a folium containing the vector states of the representation $\pi$. Consider a positive trace class operator $\varrho \in \fB(\hilbH)$ of trace 1, and let $\omega \in \sS_\pi(\fA)$ be the corresponding $\pi$-normal state. There exists an orthonormal set $(\Omega_i)_{i \in I}$ in $\hilbH$ such that $\varrho = \sum_{i \in I} \lambda_iP_i$, where $P_i$ is the projection onto $\mathbb{C} \Omega_i$ and $\lambda_i > 0$ with $\sum_{i \in I} \lambda_i = 1$. In particular, for any $A \in \fA$,
 \[
 \omega(A) = \Tr(\varrho \pi(A)) = \sum_{i \in I} \lambda_i \ev{\Omega_i, \pi(A)\Omega_i} \ .
 \]
 Given $\varepsilon > 0$, choose a finite subset $J \subset I$ such that for any finite subset
 $K \subset I$ with $J \cap K = \varnothing$ the estimate $\sum_{k \in K} \lambda_k < \varepsilon$ holds true.
 If $\omega_i$ is the vector state corresponding to $\Omega_i$, then
 \[
 \norm{\sum_{k \in K} \lambda_k \omega_k} \leq \sum_{k \in K} \lambda_k \norm{\omega_k} < \varepsilon
 \]
 since $\norm{\omega_k} = 1$. Denoting by $\wp_{\textup{f}} (I)$ the set of all finite subsets of $I$ we see that
 the net
 $\left( \sum_{j \in J} \lambda_j \omega_j\right)_{J \in \wp_{\textup{f}}(I)}$ converges in norm to $\omega$.
 Moreover, the net
 \begin{align}\label{eq:net1}
  \left( \frac{\sum_{j \in J} \lambda_j \omega_j}{\sum_{j \in J} \lambda_j}\right)_{J \in \wp_{\textup{f}} (I)}
 \end{align}
 converges to $\omega$ since the denominators converge to one. Since $\omega_i \in \sF$ for all $i$ and each element of the net \eqref{eq:net1} is a convex combination of $\omega_i$, we conclude from (F\ref{ite:folium-convex}) and (F\ref{ite:folium-norm-closed}) that $\omega \in \sF$. This proves that $\sS_\pi(\fA) \subset \sF$, proving the claim that $\sS_\pi(\fA)$ is the smallest folium containing the vector states of $\pi$.
  
 Finally let $\sF$ be a folium and let $\hilbH = \bigoplus_{\omega \in \sF} \hilbH_\omega$ and $\pi = \bigoplus_{\omega \in \sF} \pi_\omega$, where $(\hilbH_\omega, \pi_\omega, \Omega_\omega)$ denotes the GNS representation of $\omega$.  Clearly, every $\omega \in \sF$ is the $\pi$-normal vector state corresponding to $\Omega_\omega \in \hilbH_\omega \subset \hilbH$, so $\sF \subset \sS_\pi(\fA)$. It remains to be shown that $\sS_\pi(\frA) \subset \sF$. To this end it suffices to prove that every vector state is in $\sF$, since $\sS_\pi(\fA)$ is the minimal folium containing the vector states of $\pi$.
 
If $\omega \in \sF$ and $B \in \frA$ with $\omega(B^*B) = 1$, then the $\pi$-normal vector state corresponding to $\pi(B)\Omega_\omega$ is $B \cdot \omega$, which is in $\sF$ by (F\ref{ite:folium-invariance}). Given a unit vector $\Psi \in \hilbH_\omega$ and $\varepsilon > 0$, cyclicity of $\pi_\omega$ yields $C \in \frA$ such that 
 \begin{equation}\label{eq:vector_approximate}
 \norm{\Psi - \pi(C)\Omega_\omega} < \min(\varepsilon/4, 1).
 \end{equation}
 Then $\norm{\pi(C)\Omega_\omega}^2 = \omega(C^*C) > 0$, so we may define $B = C/\sqrt{\omega(C^*C)}$, for which $\norm{\pi(B)\Omega_\omega}^2 = \omega(B^*B) = 1$. Note that by \eqref{eq:vector_approximate} and the
 reverse triangle inequality
 \begin{equation}\label{eq:normalized_vector_approximate}
 \norm{\pi(B)\Omega_\omega - \pi(C)\Omega_\omega} = \abs{1 - \sqrt{\omega(C^*C)}} < \frac{\varepsilon}{4} \ .
 \end{equation}
 If $\psi \in \sS(\frA)$ is the $\pi$-normal vector state corresponding to $\Psi$, then Lemma \ref{lem:vector_state} entails
 \begin{align}\label{eq:state_approximate}
 \norm{\psi - B \cdot \omega} &\leq 2\norm{\Psi - \pi(B)\Omega_\omega}. 
 \end{align}
 Thus, \eqref{eq:vector_approximate}, \eqref{eq:normalized_vector_approximate}, and \eqref{eq:state_approximate} together imply that $\norm{\psi - B \cdot \omega} < \varepsilon$. Since $\sF$
 is norm closed and $B \cdot \omega \in \sF$, one concludes that $\psi \in \sF$.
 
 Next consider a unit vector of the form $\Psi = \sum_{i=1}^n \lambda_i \Psi_{i}$, where $\lambda_i \in \mathbb{C}$ and $\Psi_{i} \in \hilbH_{\omega_i}$ are unit vectors with distinct $\omega_i \in \sF$. If $\psi$ is the $\pi$-normal vector state corresponding to $\Psi$, then
 \[
 \psi(A) = \sum_{i=1}^n \abs{\lambda_i}^2 \ev{\Psi_i, \pi(A)\Psi_i}.
 \]
 Since $\norm{\Psi}^2 = \sum_{i=1}^n \abs{\lambda_i}^2 = 1$, we see that $\psi$ is a convex combination of elements of $\sF$, so $\psi \in \sF$ by (F\ref{ite:folium-convex}). Since any unit vector in $\hilbH$ is the limit of a net of such finite linear combinations, $\sF$ contains all vector states of $\pi$ by  Lemma \ref{lem:vector_state} and (F\ref{ite:folium-norm-closed}). This proves that $\sS_\pi(\fA) \subset \sF$. 
\end{proof}

\subsection{Pure states and superselection sectors}
\label{sec:pure-states}
The extreme points of $\sS(\fA)$ are called pure states, the set of which we denote by $\sP(\fA)$. It is well-known that $\sP(\fA)$ is nonempty when $\fA$ is nonzero and that $\sS (\fA)$ coincides with the weakly$^*$ closed convex hull of $\sP (\fA)$ when $\fA$ is unital (both being consequences of the Krein--Milman theorem). The GNS representation $(\cH_\omega, \pi_\omega)$ of a state $\omega$ is irreducible if and only if $\omega$ is pure. In this case, the quotient $\fA/\fN_\omega$ is already complete, hence $\cH_\omega = \fA/\fN_\omega$; this follows from the Kadison transitivity theorem \cite[Thm.~5.2.4]{MurphyCAOT}.

Given a non-degenerate representation $\pi$, we denote the set of pure $\pi$-normal states on $\fA$ by $\sP_\pi(\fA) = \sP(\fA) \cap \sS_\pi(\fA)$.

\begin{prop}\label{prop:pure_folia_vector_state}
Let $\fA$ be a $C^*$-algebra with a non-degenerate representation $(\hilbH,\pi)$. 
\begin{enumerate}[{\rm (i)}]
\item\label{ite:vector-state} If $\omega \in \sP_\pi(\fA)$, then $\omega$ is a vector state. 
\item\label{ite:irreducible-vector-state-pure} If $\pi$ is irreducible and $\omega$ is a vector state, then $\omega \in \sP_\pi(\fA)$. 
\item\label{ite:vectors-states-same-state-linearly-dependent} If $\pi$ is irreducible and $\Psi, \Omega \in \hilbH$ are unit vectors defining the same vector state, then $\Psi$ and $\Omega$ are linearly dependent.
\end{enumerate}
\end{prop}

\begin{proof}
(\ref{ite:vector-state}) Suppose $\omega \in \sP_\pi(\fA)$ and let $\varrho \in \fB(\hilbH)$ be a density matrix such that $\omega(A) = \Tr(\varrho \pi(A))$. There exists an orthonormal set $\qty{\Omega_i}_{i \in I}$ in $\hilbH$ such that $\varrho = \sum_{i \in I} \lambda_i P_i$, where $P_i$ is the projection onto $\bbC \Omega_i$ and $\lambda_i > 0$ with $\sum_{i \in I} \lambda_i = 1$. Fix $j \in I$ and let $\psi$ be the vector state corresponding to $\Omega_j$. Then for all positive $A \in \fA$,
\[
\lambda_j \psi(A) = \lambda_j \ev{\Omega_j, \pi(A)\Omega_j} \leq \sum_{i \in I} \lambda_i \ev{\Omega_i, \pi(A)\Omega_i} = \omega(A).
\]
Since $\omega$ is pure, there exists $\lambda \in [0,1]$ such that $\lambda_j \psi = \lambda \omega$. Since $\norm{\psi} = \norm{\omega} = 1$, we see that $\lambda_j = \lambda$, which implies that $\omega = \psi$, a vector state.

(\ref{ite:irreducible-vector-state-pure}) This is given by \cite[Thm.~5.1.7]{MurphyCAOT}.


(\ref{ite:vectors-states-same-state-linearly-dependent}) If $\Psi$ and $\Omega$ are linearly independent, then the Kadison transitivity theorem yields $B \in \fA$ such that $\pi(B)\Psi = \Psi$ and $\pi(B) \Omega = 0$. This contradicts the assumption that $\ev{\Psi, \pi(A)\Psi} = \ev{\Omega, \pi(A)\Omega}$ for all $A \in \fA$.
\end{proof}

\begin{defn}\label{def:superselection_sector}
  Given a $C^*$-algebra $\fA$, call two pure states $\psi,\omega \in \sP(\fA)$
  \emph{equivalent} 
  if their GNS representations $\pi_\psi$ and $\pi_\omega$ are unitarily equivalent.
  Let $\sim$ denote the corresponding equivalence relation  on $\sP(\fA)$.
  By a \emph{superselection sector} of $\fA$ we understand an equivalence class of pure states with respect
  to $\sim$. 

  Given $\omega \in \sP(\fA)$, we denote its superselection sector by $\sP_\omega(\fA)$. Conveniently, it follows from Proposition \ref{prop:pure_folia_vector_state} that $\sP_{\pi_\omega}(\fA) = \sP_\omega(\fA)$. Indeed, if $\psi \in \sP_{\pi_\omega(\fA)}$, then $\psi$ is a vector state of $\pi_\omega$, hence $\pi_\psi$ is unitarily equivalent to $\pi_\omega$ by uniqueness of the GNS representation up to unitary equivalence, so $\psi \in \sP_\omega(\fA)$. Conversely, if $\psi \in \sP_\omega(\fA)$, then unitary equivalence of $\pi_\psi$ and $\pi_\omega$ implies that $\psi$ is a vector state of $\pi_\omega$, hence $\psi \in \sP_{\pi_\omega}(\fA)$.
\end{defn}

\begin{rem}
Every nonzero irreducible representation is unitarily equivalent to the GNS representation of some pure state, by uniqueness of the GNS representation (up to unitary equivalence). Accordingly, the collection $\widehat{\fA}$ of unitary
equivalence classes of nonzero irreducible representations of $\fA$ is a set (rather than a class), called the (\emph{representation})  \emph{spectrum} of $\fA$. We provide a brief description of $\widehat{\fA}$ below and refer the reader to Chapter 3 of Dixmier \cite{DixCA} for a complete exposition. 

Usually $\widehat{\fA}$  is  endowed with the \emph{Jacobson} or \emph{hull-kernel topology}, which is the
unique topology having the hull-kernel operation 
\[
  \pset (\widehat{\fA}) \to \pset (\widehat{\fA}) , \quad
   E \mapsto \overline{E} := \operatorname{hull}(\ker (E) )= \Big\{ [\pi] \in\widehat{\fA} :
  {\bigcap}_{[\varrho] \in E}\ker \varrho \subset \ker  \pi \Big\}
\]
as its closure operation, cf.~\cite[Sec.~3.1]{DixCA}. Note that the hull-kernel operation
satisfies Kuratowski's axioms for a closure operation and that, for $\fA$ unital, the space
$\widehat{\fA}$ with the Jacobson topology is quasi-compact by \cite[3.1.8.~Prop.]{DixCA}.
In general, the representation spectrum is locally quasi-compact, see \cite[3.3.8.~Cor.]{DixCA}. 



Letting $\sP(\fA)_{\textup{w}^*}$ denote  $\sP(\fA)$ endowed with the weak$^*$ topology, 
the map $\kappa: \sP(\fA)_{\textup{w}^*}\to \widehat{\fA}$ which associates to each
  pure state the unitary equivalence class of its GNS representation is surjective, continuous
  and open by \cite[Thm.~3.4.11.]{DixCA}. Hence, the induced map on quotients
    \[
    \overline{\kappa} : \sP(\fA)_{\textup{w}^*}/{\sim} 
    \rightarrow \widehat{\fA}, \quad \sP_\omega(\fA) \mapsto [\pi_\omega] 
  \]
is a homeomorphism. 

Note that when considering $\sP(\fA)_{\textup{n}}$ instead (i.e., $\sP(\fA)$ endowed with the norm topology), the  quotient space $\sP(\fA)_{\textup{n}} /\!\! \sim$ is a discrete topological space. 
This will follow from Corollary \ref{cor:sector_is_open} below, where it is shown that superselection sectors are open in  $\sP(\fA)_{\textup{n}}$.
\end{rem}


In the remainder of this section we will give various characterizations of superselection sectors and show that our definition 
coincides with the interpretation  of superselection sectors in physics as maximal regions in the pure state space 
where any two states have coherent superpositions. Most of the mathematical results below can be found in
Roberts--Roepstorff \cite{RobertsRoepstorffSBCAQT}, Glimm--Kadison \cite{GlimmKadisonUOCA}, and Pedersen \cite{PedersenCAlgAutomorphisms}. 

Before we start with the mathematical definition of coherent superpositions let us note that given two states 
$\omega$ and $\psi$ there always exists a non-degenerate representation $(\hilbH,\pi)$ such that both  $\omega$ and $\psi$ become
vector states with respect to that representation. For example, one can take $(\hilbH,\pi)$ as the
direct sum of the GNS representations of  $\omega$ and $\psi$. 

\begin{defn}
  Two distinct pure states  $\psi$ and $\omega$  of a $C^*$-algebra $\fA$  are said to
  \emph{fulfill the superposition principle} or to be \emph{coherently superposable} if there exists
  a non-degenerate representation $(\hilbH,\pi)$  in which $\psi$ and $\omega$ are represented by the unit
  vectors  $\Psi$ and $\Omega$, respectively, and for all $\alpha,\beta \in \C$ the vector state
  $\varphi$ corresponding to the unit vector
  \[\Phi = \frac{\alpha\Psi + \beta\Omega}{\|\alpha\Psi + \beta\Omega\|}\] 
  is a pure state. If this is the case, one calls each of the states $\varphi$ obtained in that way a
  \emph{coherent superposition} of $\omega$ and $\psi$. 
\end{defn}

The following proposition is a rephrasal and clarification of the ``sufficient'' implication of \cite[Thm.~6.1]{ArakiMTQF}. We restrict our attention here to pure states, although \cite[Thm.~6.1]{ArakiMTQF} is stated for general states.

\begin{prop}\label{prop:inequivalent_superposition}
Let $\frA$ be a nonzero $C^*$-algebra and let $\psi$ and $\omega$ be pure states in different superselection sectors. If $(\hilbH, \pi)$ is a non-degenerate representation with unit vectors $\Psi, \Omega \in \hilbH$ representing $\psi$ and $\omega$, respectively, then $\ev{\Psi, \Omega} = 0$. Furthermore, if $\varphi$ is the vector state corresponding to $\Phi = \alpha \Psi + \beta \Omega$ for any nonzero $\alpha, \beta \in \bbC$ with $\abs{\alpha}^2 + \abs{\beta}^2 = 1$, then 
\[
\varphi = \abs{\alpha}^2 \psi + \abs{\beta}^2 \omega.
\]
In particular, $\varphi \notin \sP(\fA)$ and  $\psi$ and $\omega$ are not coherently superposable.
\end{prop}

\begin{proof}
Define $\hilbH_\Psi = \overline{\pi(\frA)\Psi}$ and note that $\hilbH_\Psi$ is a closed invariant subspace. Let $P_\Psi :\hilbH \rightarrow \hilbH_\Psi$ be the orthogonal projection and $\iota_\Psi :\hilbH_\Psi \rightarrow \hilbH$ the inclusion. Lastly, define $\pi_\Psi : \frA \rightarrow \fB(\hilbH_\Psi)$ by
\[
\pi_\Psi(A) = P_\Psi \pi(A)\iota_\Psi
\] 
for all $A \in \frA$. Since $\hilbH_\Psi$ is an invariant subspace, we see that $\iota_\Psi \pi_\Psi(A) = \pi(A)\iota_\Psi$ for all $A \in \fA$. The orthogonal complement of an invariant subspace is invariant, so it further follows that $\pi_\Psi(A)P_\Psi = P_\Psi \pi(A)$ for all $A \in \fA$. Note also that $P_\Psi \iota_\Psi = id_{\hilbH_{\Psi}}$. Thus, for all $A, B \in \fA$,
\begin{align*}
\pi_\Psi(AB) &= P_\Psi \pi(A) \pi(B)\iota_\Psi P_\Psi \iota_\Psi = P_\Psi \pi(A) \iota_\Psi \pi_\Psi(B) P_\Psi \iota_\Psi \\
&= P_\Psi \pi(A)\iota_\Psi P_\Psi \pi(B) \iota_\Psi = \pi_\Psi(A)\pi_\Psi(B).
\end{align*}
Furthermore, since $\iota_\Psi$ is an isometry, for any $x, y \in \hilbH_\Psi$ we have
\begin{align*}
\ev{\pi_\Psi(A^*)x, y} &= \ev{\iota_\Psi \pi_\Psi(A^*)x, \iota_\Psi y} = \ev{\pi(A^*)\iota_\Psi x, \iota_\Psi y} \\
&= \ev{\iota_\Psi x, \pi(A)\iota_\Psi y} = \ev{\iota_\Psi x, \iota_\Psi \pi_\Psi(A)y} = \ev{x, \pi_\Psi(A)y}.
\end{align*}
Therefore $\pi_\Psi(A^*) = \pi_\Psi(A)^*$, and $(\hilbH_\Psi, \pi_\Psi)$ is a representation of $\frA$.
Identical arguments and the corresponding notation apply with $\Psi$ replaced by $\Omega$.

The non-degeneracy of $\pi$ implies that $\Psi \in \hilbH_\Psi$, so $\Psi$ is a cyclic vector for $\pi_\Psi$. Furthermore,
\[
\psi(A) = \ev{\Psi, \pi(A)\Psi} = \ev{\iota_\Psi \Psi, \pi(A)\iota_\Psi \Psi} = \ev{\iota_\Psi \Psi, \iota_\Psi \pi_\Psi(A)\Psi} = \ev{\Psi, \pi_\Psi(A)\Psi}.
\]
It follows that $\pi_\Psi$ is unitarily equivalent to the GNS representation of $\psi$. Thus, $\pi_\Psi$ is not unitarily equivalent to $\pi_\Omega$. Since $\psi$ and $\omega$ are pure, their GNS representations are irreducible, and so are $\pi_\Psi$ and $\pi_\Omega$.

Finally, observe that
\[
P_\Omega \iota_\Psi \pi_\Psi(A) = P_\Omega \pi(A)\iota_\Psi = \pi_\Omega(A)P_\Omega \iota_\Psi \quad\text{for all } A \in \frA \ .
\] 
Since $\pi_\Psi$ and $\pi_\Omega$ are not unitarily equivalent, Schur's lemma
implies that $P_\Omega \iota_\Psi= 0$; see \cite[2.2.2]{DixCA}.
It follows that $\hilbH_\Psi$ and $\hilbH_\Omega$ are mutually orthogonal. In particular, $\ev{\Psi, \Omega} = 0$ and
\begin{align*}
\varphi(A) = \ev{\alpha \Psi + \beta \Omega, \pi(A)\qty(\alpha \Psi + \beta \Omega)}  = \abs{\alpha}^2 \psi(A) + \abs{\beta}^2 \omega(A)
\end{align*}
as desired. Since $\varphi$ is a nontrivial convex combination of pure states, $\varphi \notin \sP(\fA)$.
\end{proof}

The following lemma is a rephrasal of Proposition 3.13.4 in \cite{PedersenCAlgAutomorphisms}. It was originally stated for unital $C^*$-algebras in \cite{GlimmKadisonUOCA}.

\begin{lem}\label{lem:sector_is_open}
Let $\frA$ be a $C^*$-algebra and let $\psi, \omega \in \sP(\fA)$. If $\psi$ and $\omega$ are in different superselection sectors, then $\norm{\psi - \omega} = 2$.
\end{lem}

\begin{proof}
Let $\hilbH = \hilbH_\psi \oplus \hilbH_\omega$ and $\pi = \pi_{\psi} \oplus \pi_{\omega}$ be the direct sum of the two GNS representations. Note that $\pi$ is non-degenerate since $\pi_{\psi}$ and $\pi_{\omega}$ are non-degenerate. Let $P_\psi :\hilbH \rightarrow \hilbH_\psi$ and $\iota_\psi :\hilbH_\psi \rightarrow \hilbH$ be the usual projections and inclusions, and define $P_\omega$ and $\iota_\omega$ similarly. Observe that 
\[
\pi_{\psi}(A) P_\psi = P_\psi \pi(A) \qqtext{and} \iota_\psi \pi_{\psi}(A) = \pi(A) \iota_\psi
\]
for all $A \in \frA$, and similarly with $\psi$ replaced by $\omega$.

Define $U \in \fB(\hilbH)$ by $U(x, y) = (x, -y)$ for all $(x,y) \in \hilbH$. Our goal is to show that $U \in \pi(\fA)''$. Suppose $T \in \pi(\frA)'$. For $i, j \in \qty{\psi, \omega}$, we compute
\[
\pi_{i}(A) P_i T \iota_j  = P_i \pi(A) T \iota_j  = P_i T \pi(A) \iota_j = P_i T \iota_j \pi_{j}(A).
\]
For $i = j$, this implies that $P_i T \iota_i \in \pi_{i}(\frA)'$. Since $\pi_{i}$ is irreducible, we know $P_i T\iota_i = \lambda_i I$ for some $\lambda_i \in \bbC$. For $i \neq j$, Schur's lemma and the assumption that $\pi_{\psi}$ and $\pi_{\omega}$ are not unitarily equivalent imply, as in \cite[2.2.2]{DixCA}, that $P_i T\iota_j = 0$. Thus, $T(x,y) = (\lambda_\psi x, \lambda_\omega y)$ for all $(x,y) \in \hilbH$, so
\[
UT(x, y) = (\lambda_\psi x, - \lambda_\omega y) = TU(x, y).
\]
This implies that $U \in \pi(\frA)''$, as desired. Note that clearly $U \in \pi(\frA)'$ as well.

Fix $\varepsilon > 0$. By the von Neumann bicommutant theorem, we know the closure of $\pi(\frA)$ in the strong operator topology on $\fB(\hilbH)$ is equal to $\pi(\frA)''$. Since $U \in \pi(\fA)''$ and $\norm{U} = 1$, the Kaplansky density theorem implies that there exists $A \in \frA$ with $\norm{\pi(A)} \leq 1$ and 
\[
\norm{\pi(A)(\Psi,\Omega) - U(\Psi,\Omega)} < \varepsilon,
\]
where $\Psi$ and $\Omega$ are the cyclic unit vectors corresponding to $\psi$ and $\omega$ in the GNS construction. By Lemma \ref{lem:rep_almost_isometry}, there exists $B \in \fA$ such that $\pi(A) = \pi(B)$ and $\norm{B} < 1 + \varepsilon$. Thus,
\begin{align*}
\abs{(\psi - \omega)\qty(B)} &=  \abs{\ev{\Psi, \pi_\psi(B)\Psi} - \ev{\Omega, \pi_\omega(B)\Omega}}\\
&=  \abs{\ev{U(\Psi, \Omega), \pi(B)(\Psi, \Omega)}}\\
&\geq \abs{\ev{U(\Psi, \Omega),U(\Psi, \Omega)}} - \abs{\ev{U(\Psi, \Omega), (\pi(A) - U)(\Psi, \Omega)}}\\
&\geq 2 - \sqrt{2} \varepsilon.
\end{align*}
Thus,
\[
\norm{\psi - \omega} \geq \abs{(\psi - \omega)\qty(\frac{B}{1 + \varepsilon})} \geq \frac{2 - \sqrt{2}\varepsilon}{1 + \varepsilon}.
\]
Since $\varepsilon$ was arbitrary, this implies that $\norm{\psi - \omega} \geq 2$. Since $\norm{\psi - \omega} \leq 2$ by the triangle inequality, the result is proven.
\end{proof}


\begin{cor}\label{cor:sector_is_open}
Let $\fA$ be a $C^*$-algebra. Then the superselection sectors of $\fA$ are open in $\sP(\fA)_{\textup{n}}$, i.e., in $\sP(\fA)$ endowed with the norm topology.
\end{cor}

The following theorem is a rephrasal of Proposition 4.6 in \cite{RobertsRoepstorffSBCAQT}. However, we give a new proof for the case when $\psi$ and $\omega$ are in the same superselection sector. 

\begin{thm}\label{thm:transition_probability}
Let $\fA$ be a $C^*$-algebra, let $\psi, \omega \in \sP(\fA)$,  and let $(\hilbH, \pi)$ be a non-degenerate representation with unit vectors $\Psi, \Omega \in \hilbH$ representing $\psi$ and $\omega$, respectively. If $\psi$ and $\omega$ are in different superselection sectors, or if they are in the same superselection sector and $(\hilbH, \pi)$ is irreducible, then
\[
\abs{\ev{\Psi, \Omega}}^2 = 1 - \frac{1}{4}\norm{\psi - \omega}^2.
\]
\end{thm}

\begin{proof}
If $\psi$ and $\omega$ are in different superselection sectors, then $\norm{\psi - \omega} = 2$ by Lemma \ref{lem:sector_is_open}, so the right hand side is zero, and $\ev{\Psi, \Omega} = 0$ by Proposition \ref{prop:inequivalent_superposition}.

We now consider the case where $\psi$ and $\omega$ are in the same superselection sector and $(\hilbH, \pi)$ is irreducible. If $\Psi$ and $\Omega$ are linearly dependent, then $\psi = \omega$ and $\abs{\ev{\Psi, \Omega}}^2 = 1$, so the identity holds. To prove the case where $\Psi$ and $\Omega$ are linearly independent, we will show inequality in both directions. 

Let $\lambda = 1 - \abs{\ev{\Psi, \Omega}}^2$ (which is strictly positive since $\Psi$ and $\Omega$ are linearly independent unit vectors) and define
\[
e_\Psi = \lambda^{-1/2}\qty(\Omega - \ev{\Psi, \Omega}\Psi) \qqtext{and} e_\Omega = \lambda^{-1/2}\qty(\Psi - \ev{\Omega, \Psi} \Omega).
\]
Then $\qty{\Psi, e_\Psi}$ and $\qty{\Omega, e_\Omega}$ are both orthonormal systems for $\vecspan\qty{\Psi, \Omega}$. We may construct a unitary $U \in \fB(\hilbH)$ such that $U\Psi = e_\Omega$ and $Ue_\Psi = -\Omega$ by extending $U$ linearly on $\vecspan\qty{\Psi, \Omega}$ and having it act as the identity on the orthogonal complement. Of course, $\norm{U} = 1$ since $U$ is unitary. Since $\pi$ is irreducible, we know $U \in \pi(\fA)'' = \fB(\hilbH)$. Therefore, given $\varepsilon > 0$, the von Neumann bicommutant theorem and the Kaplansky density theorem yields $A \in \fA$ such that $\norm{\pi(A)} \leq 1$, 
\begin{align*}
\norm{\pi(A)\Psi - U\Psi} < \varepsilon \qqtext{and} \norm{\pi(A)\Omega - U\Omega} < \varepsilon.
\end{align*}
By Lemma \ref{lem:rep_almost_isometry} we may assume $\norm{A} < 1 + \varepsilon$. Note that
\[
U\Omega = U\qty(\lambda^{1/2}e_\Psi + \ev{\Psi, \Omega}\Psi) = -\lambda^{1/2}\Omega + \ev{\Psi, \Omega}e_\Omega.
\]
Then we can compute
\begin{align*}
\abs{(\psi - \omega)(A)} &= \abs{\ev{\Psi, \pi(A)\Psi} - \ev{\Omega, \pi(A)\Omega}}\\
&\geq \abs{\ev{\Psi, U\Psi} - \ev{\Omega, U\Omega}} - \abs{\ev{\Psi, \pi(A)\Psi - U\Psi}} - \abs{\ev{\Omega, \pi(A)\Omega - U\Omega}}\\
& \geq 2\lambda^{1/2} - 2\varepsilon.
\end{align*}
Thus,
\[
\norm{\psi - \omega} \geq \abs{(\psi - \omega)\qty(\frac{A}{1 + \varepsilon})} \geq \frac{2\lambda^{1/2} - 2\varepsilon}{1 + \varepsilon}.
\]
Since $\varepsilon > 0$ was arbitrary, this implies that $\norm{\psi - \omega} \geq 2\lambda^{1/2}$, which rearranges to
\[
\abs{\ev{\Psi, \Omega}}^2 \geq 1 - \frac{1}{4} \norm{\psi - \omega}^2.
\]

For the reverse inequality, let $A \in \frA$ with $\norm{A}\leq 1$. Given $\alpha, \beta, \gamma \in \bbC$ with $\abs{\alpha} = \abs{\beta} = \abs{\gamma} = 1$,  define $A' = \alpha A$, $\Psi' = \beta \Psi$, and $\Omega' = \gamma \Omega$, we note that $\Psi'$ and $\Omega'$ define the same states as $\Psi$ and $\Omega$, respectively, and $\abs{(\psi - \omega)(A)} = \abs{(\psi - \omega)(A')}$. We choose $\alpha$ so that $(\psi - \omega)(A') \geq 0$ and we choose $\beta$ and/or $\gamma$ such that $\ev{\Psi', \Omega'}$ is real. We compute
\begin{align*}
\ev{\Psi' \pm \Omega', \pi(A')(\Psi' \mp \Omega')} = (\psi  - \omega)(A')  \pm \qty[\ev{\Omega', \pi(A')\Psi'} - \ev{\Psi', \pi(A')\Omega'}].
\end{align*}
For one of the sign choices, the term in square brackets will have nonnegative real part, and therefore the magnitude of the left hand side will be greater than $\abs{(\psi - \omega)(A)}$. For either sign choice, the Cauchy-Schwarz inequality gives
\begin{align*}
\abs{\ev{\Psi' \pm \Omega', \pi(A')\qty(\Psi' \mp \Omega')}} &\leq \norm{\Psi' + \Omega'} \norm{\Psi' - \Omega'} \\
&= 2\sqrt{1 - \qty(\Re \ev{\Psi', \Omega'})^2}\\
&= 2 \sqrt{1 - \abs{\ev{\Psi, \Omega}}^2}.
\end{align*}
Thus, we have
\[
\abs{(\psi - \omega)(A')} = \abs{(\psi - \omega)(A)} \leq 2\sqrt{1 - \abs{\ev{\Psi, \Omega}}^2}.
\]
Since $A$ was arbitrary, we have $\norm{\psi - \omega} \leq 2 \sqrt{1 - \abs{\ev{\Psi, \Omega}}^2}$, which rearranges to the desired inequality.
\end{proof}

\begin{thm}\label{thm:superselection_sector_equivalences}
Let $\fA$ be a $C^*$-algebra and let $\psi, \omega \in \sP(\fA)$ be pure states. The following are equivalent:
\begin{enumerate}[{\rm (i)}]
\item\label{ite:superselection-sector} $\psi$ and $\omega$ are in the same superselection sector,
\item\label{ite:irreducible-representation} there exists a nonzero irreducible representation $(\hilbH, \pi)$ such that $\psi, \omega \in \sP_\pi(\fA)$,
\item\label{ite:quasilocal-perturbation} there exists $B \in \fA$ such that $\psi(B^*B)  = 1$ and $\omega = B \cdot \psi$,
\item\label{ite:normal-states} for any non-degenerate representation $(\hilbH, \pi)$, we have $\psi \in \sP_\pi(\fA)$ if and only if $\omega \in \sP_\pi(\fA)$,
\item\label{ite:path-component} $\psi$ and $\omega$ are in the same path component of
  $\sP(\fA)_{\textup{n}}$,
\item\label{ite:superposable} $\psi$ and $\omega$ are  coherently superposable.
\end{enumerate}
If $\fA$ is unital, then the element $B \in \fA$ in {\rm (\ref{ite:quasilocal-perturbation})} may be chosen to be unitary. 
\end{thm}

The equivalence (\ref{ite:superselection-sector}) $\Leftrightarrow$ (\ref{ite:irreducible-representation}) $\Leftrightarrow$ (\ref{ite:quasilocal-perturbation}) $\Leftrightarrow$ (\ref{ite:path-component}) was stated by Roberts--Roepstorff in the case where $\fA$ is unital and the element $B \in \fA$ in (\ref{ite:quasilocal-perturbation}) is unitary \cite[Prop.~4.2 \& Thm.~4.5]{RobertsRoepstorffSBCAQT}. The equivalence (\ref{ite:superselection-sector}) $\Leftrightarrow$ (\ref{ite:superposable}) is implied by \cite[Thm.~6.1]{ArakiMTQF}, which gives a more general equivalence for states which are not necessarily pure. The equivalence of (\ref{ite:normal-states}) with the others is easily proved. We clarify that the result holds for non-unital $C^*$-algebras as well.

\begin{proof}
Denote by $(\hilbH_\psi, \pi_\psi,\Omega_\psi)$ and $(\hilbH_\omega, \pi_\omega,\Omega_\omega)$ the GNS representations of $\psi$ and $\omega$, respectively. 
Now verify the following.

\eqref{ite:superselection-sector} $\Rightarrow$ \eqref{ite:irreducible-representation}. Let $U:\hilbH_\psi \rightarrow \hilbH_\omega$ be a unitary 
intertwining the representations $(\hilbH_\psi, \pi_\psi)$ and $(\hilbH_\omega, \pi_\omega)$. Then 
\begin{align*}
\ev{U\Omega_\psi, \pi_\omega(A)U\Omega_\psi} = \ev{U\Omega_\psi, U\pi_\psi(A)\Omega_\psi} = \ev{\Omega_\psi, \pi_\psi(A)\Omega_\psi} = \psi(A)  
\end{align*}
for all $A \in \fA$, which implies that $\psi \in \sP_{\pi_\omega}(\fA)$ and, of course, $\omega \in \sP_{\pi_{\omega}}(\fA)$ as well. 

\eqref{ite:irreducible-representation} $\Rightarrow$ \eqref{ite:quasilocal-perturbation}. Let $\Psi, \Omega \in \hilbH$ be unit vectors representing $\psi$ and $\omega$. By the Kadison transitivity theorem, there exists $B \in \fA$ such that $\pi(B)\Psi = \Omega$. Then 
\[
\psi(B^*B) = \ev{\Psi, \pi(B^*B)\Psi} = \ev{\Omega, \Omega} = 1
\]
and
\[
(B \cdot \psi)(A) = \psi(B^*AB) = \ev{\Omega, \pi(A)\Omega} = \omega(A),
\]
as desired. Note that if $\fA$ is unital, then Kadison's transitivity theorem allows $B$ to be chosen unitary. 

\eqref{ite:quasilocal-perturbation} $\Rightarrow$ \eqref{ite:superselection-sector}. In the GNS representation of $\psi$ one has
\[
  \omega(A) = \ev{\pi_\psi(B)\Omega_\psi, \pi_\psi(A)\pi_\psi(B)\Omega_\psi} \quad\text{for all } A \in \fA \ .
\] 
Since $\pi_\psi(B)\Omega_\psi$ is a cyclic unit vector, uniqueness of the GNS representation up to unitary equivalence implies that $\omega$ and $\psi$ are in the same superselection sector.

\eqref{ite:quasilocal-perturbation} $\Rightarrow$ \eqref{ite:normal-states}. Given any non-degenerate representation $(\hilbH, \pi)$, $\psi \in \sP_\pi(\fA)$ implies $\omega \in \sP_\pi(\fA)$ by (F\ref{ite:folium-invariance}). Since \eqref{ite:quasilocal-perturbation} $\Rightarrow$ \eqref{ite:superselection-sector} and \eqref{ite:superselection-sector} is symmetric in $\psi$ and $\omega$, we also have \eqref{ite:quasilocal-perturbation} with $\psi$ and $\omega$ switched, hence $\omega \in \sP_\pi(\fA)$ implies $\psi \in \sP_\pi(\fA)$.

\eqref{ite:normal-states} $\Rightarrow$ \eqref{ite:path-component}. Let $(\hilbH_\psi, \pi_\psi, \Psi)$ be the GNS representation of $\psi$. By \eqref{ite:normal-states}, there exists a unit vector $\Omega \in \hilbH_\psi$ representing $\omega$. Since the unit sphere of $\hilbH_\psi$ is path connected, there exists a continuous path in the unit sphere from $\Psi$ to $\Omega$, hence there exists a continuous path in $\sP(\fA)_{\textup{n}}$ from $\psi$ to $\omega$ by Lemma \ref{lem:vector_state}.

\eqref{ite:path-component} $\Rightarrow$ \eqref{ite:superselection-sector}. By Corollary \ref{cor:sector_is_open}, the superselection sector containing $\psi$ is open in $\sP(\fA)$ with respect to the norm topology. If $(\hilbH_\psi, \pi_\psi)$ is the GNS representation of $\psi$, then \eqref{ite:irreducible-representation}  $\Rightarrow$ \eqref{ite:superselection-sector} and \eqref{ite:superselection-sector} 
$\Rightarrow$ \eqref{ite:normal-states} implies that $\sP_{\pi_\psi}(\fA)$ is the superselection sector containing $\psi$. But $\sP_{\pi_\psi}(\fA) = \sS_{\pi_\psi}(\fA) \cap \sP(\fA)$ is norm closed in $\sP(\fA)$ since $\sS_{\pi_\psi}(\fA)$ is norm closed in $\sS(\fA)$. Thus, the superselection sector containing $\psi$ is both an open and closed subset of $\sP(\fA)$ in the norm
topology, so it contains the path component of $\sP(\fA)_{\textup{n}}$ containing $\psi$.

\eqref{ite:superselection-sector} $\Rightarrow$ \eqref{ite:superposable}. 
  Assume that $\psi$ and $\omega$ are in the same superselection sector. Let $(\hilbH,\pi)$
  be an irreducible representation which is unitarily equivalent to $\pi_\omega$ and hence to $\pi_\psi$
  as well. Then there exist unit vectors $\Psi\in\hilbH$ and $\Omega \in \hilbH$ which induce
  the states $\psi$ and $\omega$, respectively. Let $\alpha,\beta \in \C^*$ and put
  $\Phi = (\alpha\Psi + \beta\Omega)\widehat{\phantom{\omega}}$. The vector state $\varphi$
  corresponding to the unit vector $\Phi$ is pure by
  Proposition \ref{prop:pure_folia_vector_state} (\ref{ite:irreducible-vector-state-pure}),
  hence $\psi$ and $\omega$ are coherently superposable. 
  
\eqref{ite:superposable} $\Rightarrow$ \eqref{ite:superselection-sector}. This is contained in Proposition \ref{prop:inequivalent_superposition}.
\end{proof}

\begin{rem}\label{rem:causal_structures}
  The equivalence \eqref{ite:superselection-sector} $\Rightarrow$ \eqref{ite:superposable} ties together the mathematical notion of superselection sector defined in Definition \ref{def:superselection_sector} and the physical concept of superposition.
  However, when $\fA$ is the algebra of observables of a quantum mechanical  system or a quantum field theory, it may not be the case that all the
  superselection sectors in the sense of Definition \ref{def:superselection_sector} above are physically relevant. In their ground breaking work \cite{DHRI,DHRII},
  Doplicher--Haag--Roberts (DHR) introduced a general theory of superselection sectors for algebraic
  quantum field theory. More precisely, they defined   superselection sectors for a local net of $C^*$-algebras over Minkowski space
  fullfilling the axioms by Haag--Kastler \cite{HaagKastlerAAQFT,HaagLocalQP}. Adapted to more
  general situations in algebraic quantum mechanics and algebraic quantum field theory, their approach can be
  described as follows.    
  Assume to be given a quasi-local algebra $\fA$. This means that $\fA$ coincides
  with the the colimit 
  of a \emph{causal net of algebras} $(\fA_{\Lambda})_{\Lambda \in \net}$ 
  which is  a particular kind of inductive system of $C^*$-algebras indexed over 
  a so-called \emph{causal index set} $\net$. 
  Let us explain in more detail what one understands by these notions.
  Assume to be given a topological space $M$.
  In most applications $M$ is either a lorentzian manifold which models the underlying
  spacetime or a discrete lattice like $\Z^d$.   
  A \emph{causal complement} then is an operation ${}^\perp :\topology(M) \to \topology (M)$
  on the topology of $M$ such that the
  following conditions hold true, cf.~\cite{WollenbergCausalNets,KeylCausal,nLabCausalComplement}:
  \begin{enumerate}[{\rm (i)}] 
  \item\label{ite:C1}
    $O \subset O^{\perp\perp}$ for all open $O \subset M$.
  \item
    $ O \cap O^\perp = \emptyset$ for all open $O \subset M$.  
  \item\label{ite:C3}
    $\big( O_1 \cup O_2 \big)^\perp =  O_1^\perp \cap O_2^\perp$
    for all open $O_1,O_2\subset M$.     
  \end{enumerate}
  Sets $O\in \topology (M)$ with the property that $O= (O^\perp )^\perp$ are called
  \emph{causally complete}.
  In addition to axioms (\ref{ite:C1}) to (\ref{ite:C3}) it is assumed that there are enough
  causally complete subsets which means that we require: 
  \begin{enumerate}[{\rm (i)}]
  \setcounter{enumi}{3}
  \item\label{ite:C4}
    There exists a countable family $\net$ of causally complete non-empty open subsets
    $\Lambda \subset M$ which is a basis of the topology, upward directed by inclusion and such that
    $\Lambda^\perp \neq \emptyset$ for all $\Lambda\in \net$. 
  \end{enumerate}
  The last property in particular guarantees that $M = \bigcup_{\Lambda \in \net} \Lambda$.
  A family $\net$ which satisfies the condition in axiom (\ref{ite:C4}) is called
  a \emph{causal index set}. 
  By a \emph{causal net of algebras} one now understands a strict inductive system of $C^*$-algebras
  $(\fA_{\Lambda})_{\Lambda \in \net}$ where the index set $\net$ is a causal index set and where the
  commutation relation $[\fA_{\Lambda} , \fA_{\Omega}] =0$ holds for all causally disjoint
  $\Lambda,\Omega \in \net$ that is for all $\Lambda,\Omega$ in $\net$ such
  $\Omega \subset \Lambda^\perp$ or, equivalently,  $\Lambda \subset \Omega^\perp$.
  Note that the strictness condition implies that $\fA_{\Lambda}$ and $\fA_{\Omega}$ are subalgebras of
  $\fA_{\Gamma}$ where $\Gamma \in \net$ has been chosen so that $\Omega \cup \Lambda \subset \Gamma$.
  Therefore, the commutation relation for $\fA_{\Lambda}$ and $\fA_{\Omega}$ makes sense, indeed. 
  
  Since a causal index set comprises a basis for the topology, one can define for every open
  $O\subset M$ the $C^*$-algebra $\fA_O$ as the $C^*$-algebra colimit of the inductive system 
   $(\fA_{\Lambda})_{\Lambda \in \net, \Lambda \subset O}$.
  The \emph{quasi-local algebra} $\fA$ coincides by definition with the $C^*$-algebra $\fA_M$.
  
  In physically interesting examples, $\net$ might be the  set of finite subsets of a countable
  discrete lattice such as $\Z^ d$ with the set-theoretic complement as causal complement or the set of open double cones in Minkowski space together with
  the relativistic  causal complement as in the work of DHR.
  Usually, the system comes equipped with a symmetry group $G$ which acts simultaneously on the
  space $M$ and on the inductive system $(\fA_\Lambda)_{\Lambda \in \net}$ in a compatible fashion and so that the
  causal complement is preserved. This means that
  $(g O)^\perp = g O^\perp $ for open $O \subset M$ and 
  $g\fA_\Lambda \subset \fA_{g\Lambda}$ for $\Lambda \in \net$. In the original DHR
  setup the symmetry group is the Poincar\'e group with its natural action on Minkowski space.
  In addition to the causal index set $\net$ one sometimes assumes to be given a second family 
  $\locreg \subset \topology (M)$ of \emph{localizable regions}. The elements of $\locreg$ and their
  causal complements are assumed to be non-empty, and the union of all $\Lambda \in \locreg$ is
  assumed to coincide with the space $M$.
  The set of localizable regions may coincide with the causal index set $\net$ as for example in the original work by DHR. 
  The final ingredient in the DHR analysis is a distinguished reference state $\omega$ on $\fA$. That
  state is assumed to be pure and invariant under the action of the symmetry group if there is one.
  Usually, the reference state is a vacuum state or  a ground state of some Hamiltonian of
  the system. The physically allowable sectors are now defined as those equivalence classes of pure
  states  $\varrho$ on $\fA$
  which are invariant under the $G$-action  and satisfy the following condition:
  \begin{enumerate}[{\rm (DHR)}]
  \item
    There exists a $\Lambda \in \locreg$ such that over the causal complement $\Lambda^\perp$
    the restricted GNS representations $\pi_\omega\big|_{\fA_{\Lambda^\perp}}$ and
    $\pi_\varrho\big|_{\fA_{\Lambda^\perp}}$ are  unitarily equivalent.
  \end{enumerate}
  In the situation where the quasi-local algebra is defined by a net of
  algebras over Minkowski space fulfilling the Haag--Kastler axioms and the localizable regions are the
  double cones in Minkowski space, the superselection sectors which satisfy the DHR condition are called
  the DHR sectors. 
  In the $C^*$-algebraic formulation of quantum spin systems the
  described approach to superselection sectors has been advocated in the work by Naaijkens,
  Cha, and Nachtergaele, see
  \cite{NaaijkensLEKTCP,NaaijkensKQDMLQPPV,ChaNaaijkensNachtergaeleSCIQSS}. In a quantum spin system over a countably
  infinite lattice, the localizable regions are infinite cones which in general are not finite subsets of the
  lattice anymore. This means that  in this case  the space of localizable regions differs
  from the causal index set defining the quasi-local algebra. A similar phenomenon appears also in the 
  approach by Buchholz--Fredenhagen \cite{BuchholzFredenhagenLocalityParticleStates}
  to superselection sectors describing relativistic massive particles. There, the localizable regions are
  given by infinite cones as well and the elements of the causal index sets are relatively compact double cones
  in Minkowski space. 
\end{rem}

The last result in this subsection is of a more categorical nature about the naturality of the
GNS construction and the functoriality of the sector space. 

\begin{prop}\label{prop:functorial-GNS}
Let $\alpha: \fB \to \fC$ be a $*$-isomorphism between $C^*$-algebras $\fB$ and $\fC$.
Then the following holds true.
\begin{enumerate}[{\rm (i)}] 
\item\label{ite:naturality-GNS}
The GNS construction is natural in the sense that for every
pure state $\omega$ on $\fB$ with corresponding GNS  representation
$(\hilbH_\omega,\pi_\omega,\Omega_\omega)$ and Gelfand ideal
$\cN_\omega = \{ B\in \fB : \omega (B^* B) = 0 \}$ the state
$\psi := \alpha_* \omega $ is a pure state on $\fC$ with Gelfand ideal given by
$\cN_\psi = \alpha (\cN_\omega)$. Moreover,
if $(\hilbK_\psi,\rho_\psi,\Omega_\psi  )$ is a cyclic representation of $\fC$
so that $\Omega_\psi $ represents the state $\psi$, 
then there is a commutative diagram
\[
  \begin{tikzcd}
  \fB \arrow[r,"\alpha"] \arrow[d,"\xi_\omega"]& \fC \arrow[d,"\xi_\psi"]\\
  \hilbH_\omega \arrow[r,"U_\alpha"]& \hilbK_\psi 
 \end{tikzcd}
\]
where $\xi_\omega :\fB \to \hilbH_\omega$ and $\xi_\psi :\fC \to \hilbK_\psi$
are the maps $B \mapsto \pi_\omega(B) \Omega_\omega$ and
$C \mapsto \varrho_\psi(C) \Omega_\psi$, respectively. 
The bottom arrow $U_\alpha:\hilbH_\omega \rightarrow \hilbK_\psi$ is
unitary and uniquely determined by the equality
$U_\alpha (\Omega_\omega) = \Omega_\psi$.
Finally, the maps $\xi_\omega$ and $\xi_\psi$ are both surjective, and their kernels coincide with
$\fN_\omega$ and $\fN_\psi$, respectively. 
\item\label{ite:functoriality-sector-spaces}
  There is a homeomorphism $\overline{\alpha} :  {\ssec}(\fB) \to {\ssec} (\fC)$
  between spaces of superselection sectors endowed with the quotient
  weak$^*$ topologies which is uniquely determined by the
  requirement that for every pure state $\omega$ on $\fB$ its sector
  $S_\omega$ is mapped under $\overline{\alpha}$ to the sector $S_\psi$ of $\fC$
  containing $\psi = \alpha_*(\omega)$. 
\end{enumerate}
\end{prop}
  \begin{proof}
    First we show (\ref{ite:naturality-GNS}). Since $\alpha$ is a $*$-isomorphism,
    it is clear that the state $\psi = \alpha_* \omega $ is pure whenever $\omega$ is
    and that its  Gelfand ideal $\cN_\psi$ coincides with $\alpha (\cN_\omega)$.
    Now observe that the triple $(\hilbK_\psi,\rho_\psi\circ \alpha,\Omega_\psi  )$
    is a cyclic representation of $\cB$ such that for all $B\in \cB$
    \[
      \ev{\Omega_\psi,  \rho_\psi\circ \alpha (B)\Omega_\psi} =
      \psi (\alpha (B)) = \omega (B) \ .
    \]
    By uniqueness of GNS-representations there exists a unique unitary operator
    $U_\alpha:\hilbH_\omega \to \hilbK_\psi$ making the above diagram commute
    and such that $U_\alpha (\Omega_\omega) = \Omega_\psi$.
    
    As already observed before, the Hilbert space $\hilbH_\omega$ of the GNS representation
    of $\omega$ coincides with $\fA/\fN_\omega$, hence the map
    $\xi_\omega: \fA \to  \hilbH_\omega$, $A \mapsto \pi_\omega (A)\Omega_\omega = A + \fN_\omega$
    is surjective and has kernel $\fN_\omega$. By unitarity of $U_\alpha$,
    the map  $\xi_\psi$ has to be surjective as well and its kernel is given by
    $\alpha (\fN_\omega) =\fN_\psi$. 
   

    To verify (\ref{ite:functoriality-sector-spaces}), observe first that the
    $*$-isomorphism  $\alpha$ induces a homeomorphism
    $\alpha_* : \sP (\fB)\to\sP(\fC)$
    between pure state spaces endowed both with either the norm or the
    weak$^*$ topologies. In particular $\alpha_*$ therefore maps path components
    of $\sP (\fB)_{\textup{n}}$ to those of  $\sP (\fC)_{\textup{n}}$ and vice versa.
    The claim now follows.   
  \end{proof}

\subsection{Connectedness properties}
\label{subsec:connectedness}

\begin{prop}\label{prop:decomposition-pure-state-space-superselection-sectors}
  Given a  non-zero $C^*$-algebra $\fA$, the  space $\sP(\frA)_{\textup{n}}$ of pure states endowed with the norm topology is locally path connected.
  The superselection sectors are the path components of $\sP(\frA)_{\textup{n}}$ and coincide with its connected components.
\end{prop}

\begin{proof}
  We show every open ball $B_r(\omega) \subset \sP(\fA)_{\textup{n}}$ with $r \leq 2$ is path connected. If $\psi \in B_r(\omega)$, then Lemma \ref{lem:sector_is_open} implies $\psi$ and $\omega$ are in the same superselection sector, so by the equivalence of (\ref{ite:superselection-sector}) and (\ref{ite:irreducible-representation}) in  Theorem \ref{thm:superselection_sector_equivalences}, there exists a nonzero irreducible representation of $(\cH, \pi)$ such that $\psi, \omega \in \sP_\pi(\fA)$. Let $\Psi, \Omega \in \bbS \cH$ represent $\psi$ and $\omega$. There exists a path $\gamma :[0,1] \rightarrow \bbS \cH$ such that $\gamma(0) = \Psi$ and $\gamma(1) = \Omega$, and we can compose this with the continuous map from Lemma \ref{lem:vector_state} to get a continuous path in $\sP(\fA)_{\textup{n}}$ between $\psi$ and $\omega$. Thus, $B_r(\omega)$ is path-connected, and this implies that $\sP(\fA)_{\textup{n}}$ is locally path-connected. 

  The equivalence of (\ref{ite:superselection-sector}) and (\ref{ite:path-component}) in Theorem \ref{thm:superselection_sector_equivalences} implies
  that the superselection sectors are just the path components of $\sP(\frA)_{\textup{n}}$. The connected components coincide with the path components since this is always true in a locally path-connected space.
\end{proof}

For infinite-dimensional $C^*$-algebras of interest in the study of quantum lattice spin systems, $\sP(\frA)_{\textup{n}}$ has uncountably many components (we give an example in Section \ref{sec:examples}). In contrast, there is often only one component of $\sP(\frA)_{\textup{w}^*}$, i.e.\ the pure state space with the weak$^*$ topology. This can be seen as a consequence of the theorem below. Following \cite{KadTSOTD} (in the version of \cite{AarFSSCSA}), we call a set $\sF$ of states \emph{full}
if for all $A\in \fA$ the relation $A\geq 0$ is equivalent to $\omega (A)\geq 0$ for all
$\omega \in \sF$. The theorem below gives a crucial characterization of when a set of states is full.

\begin{thm}[{\cite[Thm.~2.2]{KadTSOTD} \& \cite[Thm.~1 \& 2]{AarFSSCSA}}] \label{thm:full_set}
  For a set $\sF$ of states on a unital $C^*$-algebra $\fA$, the following are equivalent:
  \begin{enumerate}[{\rm (i)}]
  \item\label{ite:full} $\sF$ is full,
  \item\label{ite:full_weak*_cch} the weak$^*$ closure of the convex hull of  $\sF$  contains  $\sS(\fA)$,
  \item\label{ite:full_weak*_closure} the weak$^*$ closure of  $\sF$ contains $\sP (\fA)$,
  \item\label{ite:full_norm} $\| A \| =\sup\limits_{\varrho \in \sF} \varrho (A)$ for every positive element $A \in \fA$. 
  \end{enumerate}
  If $\fA$ is non-unital and nonzero, then {\rm (ii) $\Leftrightarrow$ (iii) $\Leftrightarrow$ (iv) $\Rightarrow$ (i)}.
\end{thm}

A nice proof of (\ref{ite:full}) $\Rightarrow$ (\ref{ite:full_weak*_cch}) and (\ref{ite:full}) $\Rightarrow$ (\ref{ite:full_weak*_closure}) can be found in \cite[Thm.~5.1.14]{MurphyCAOT}. We caution that \cite[Prop.~3.2.10]{BratteliRobinsonOAQSMII} falsely states (\ref{ite:full}) $\Rightarrow$ (\ref{ite:full_weak*_cch}) without requiring $\fA$ to be unital; a counterexample is provided in \cite{AarFSSCSA}. 



\begin{cor}
\label{cor:connectedness-cstaralgebra-faithful-irred-rep}  
If $\fA$ is a nonzero $C^*$-algebra with a faithful irreducible representation, then $\sP(\fA)_{\textup{w}^*}$ is connected.
\end{cor}

\begin{proof}
Let $(\cH, \pi)$ be a faithful irreducible representation. Then for every $A \in \fA_+$,
\[
\norm{A} = \norm{\pi(A)} = \sup_{\Psi \in \cH} \ev{\Psi, \pi(A)\Psi} = \sup_{\psi \in \sP_\pi(\fA)} \psi(A),
\]
so $\sP_\pi(\fA)$ is weak$^*$ dense in $\sP(\fA)$ by Theorem \ref{thm:full_set}. Since $\sP_\pi(\fA)$ is connected in the norm topology, it is connected in the weak$^*$ topology, and density of $\sP_\pi(\fA)$ now implies that $\sP(\fA)$ is connected.
\end{proof}

Following \cite[Defs.~1.8 \& 5.7]{EilCCCA}, we call a $C^*$-algebra $\fA$ \emph{connected},
respectively \emph{locally connected}, whenever the corresponding pure state space $\sP(\fA)_{\textup{w}^*}$
has that property. Using this terminology, Cor.~\ref{cor:connectedness-cstaralgebra-faithful-irred-rep} then says that a
$C^*$-algebra with a faithful irreducible representation is connected. So in particular every simple  $C^*$-algebra
is connected. In fact, every simple $C^*$-algebra is locally connected, as implied by Theorem 5.6 in \cite{EilCCCA}.

Finally, in this section we will provide a few criteria which entail that the pure state space is
path-connected or locally path-connected in the weak$^*$ topology. The essential tool is the following.

\begin{thm}[{\cite[Prop.~5.9]{EilCCCA}}]\label{thm:locally_path_connected}
If $\fA$ is separable, connected, and locally connected, then $\sP(\fA)_{\textup{w}^*}$ is path-connected and locally path-connected.
\end{thm}

This theorem is a trivial synthesis of the facts that $\sP(\fA)_{\textup{w}^*}$ is a Polish space when $\fA$ is separable \cite[Prop.~4.3.2]{PedersenCAlgAutomorphisms}, and that a locally connected complete metric space is locally path-connected. It is unfortunately difficult to find a correct proof of the latter fact in the literature. One may find a roadmap of common errors along with corrections in \cite{PersistenceofErrors}, where several references to both correct and incorrect proofs are provided. Nonetheless, Theorem \ref{thm:locally_path_connected} applies in many cases of interest in physics.

\begin{thm}
  Let $\fA$ be a $C^*$-algebra. 
  \begin{enumerate}[{\rm (i)}]
  \item\label{ite:simple-separable}
    If $\fA$ is simple and separable, then $\fA$ is path-connected and locally path-connected.
  \item\label{ite:direct-sum-separables}
    If $\fA = \bigoplus_{n=1}^N \fA_n$, $N \in \mathbb{N} \cup \qty{\infty}$ is the direct sum of separable, simple $C^*$-algebras $\fA_n$, then $\fA$ is locally path-connected.
  \item\label{ite:colimit-injective-direct-system}
    If $\fA$ is the colimit of an injective direct system of countably many separable, simple $C^*$-algebras, then $\fA$ is path-connected and locally path-connected.
  \end{enumerate}  
\end{thm}

By the direct sum $\bigoplus_{n=1}^\infty \fA_n$, we mean the set of sequences $(A_n)_{n \in \bbN} \in \prod_{n =1}^\infty \fA_n$ such that $\norm{A_n} \rightarrow 0$, equipped with the max norm and componentwise algebraic operations. 

\begin{proof}
If $\fA$ is simple and separable, then $\sP(\fA)_{\textup{w}^*}$ is connected, locally connected, and completely metrizable, hence path-connected and locally path-connected by the preceding remarks. The type of $C^*$-algebra described in (\ref{ite:colimit-injective-direct-system}) is itself simple and separable, and therefore path-connected and locally path-connected by (\ref{ite:simple-separable}).

It is stated without proof in \cite{EilCCCA} that the pure state space of $\bigoplus_{n=1}^\infty \fA_n$ is homeomorphic to the disjoint union $\bigsqcup_{n=1}^\infty \sP(\fA_n)$, from which it follows that (\ref{ite:direct-sum-separables}) holds since each $\sP(\fA_n)$ is locally path-connected. We give a proof of this fact. For each $n \in \bbN$ let $\pi_n:\fA \rightarrow \fA_n$ be the canonical projection and let $\iota_n :\fA_n \rightarrow \fA$ be the canonical inclusion $\iota_n(A) = (B_m)_{m \in \bbN}$ where $B_m = 0$ if $m \neq n$ and $B_n = A$. Given $\omega \in \sP(\fA_n)$, the fact that $\pi_n$ is a surjective $*$-homomorphism implies $\omega \circ \pi_n \in \sP(\fA)$. Thus, we have a map $\bigsqcup_{n=1}^\infty \sP(\fA_n) \rightarrow \sP(\fA)$. If $\psi \in \sP(\fA_m)$ and $\omega \in \sP(\fA_n)$ such that $\psi \circ \pi_m = \omega \circ \pi_n$, then $m = n$ must hold, hence surjectivity of the projections implies $\psi = \omega$, so this map is injective. For surjectivity, suppose $\omega \in \sP(\fA)$. It cannot happen that $\omega \circ \iota_n = 0$ for all $n \in \bbN$ because the span of $\bigcup_{n \in \bbN} \iota_n\qty(\fA_n)$ is dense in $\fA$, therefore there exists $n \in \bbN$ such that $\omega \circ \iota_n \neq 0$. By definition of the algebraic operations on $\fA$, we see that $(\iota_n \circ \pi_n)(A) \leq A$ for all $A \in \fA_+$. Therefore $\omega \circ \iota_n \circ \pi_n$ is a positive linear functional on $\fA$ dominated by $\omega$. Since $\omega$ is pure, there exists $t \in [0,1]$ such that $\omega \circ \iota_n \circ \pi_n = t \omega$. Composing with $\iota_n$ on the right and using the fact that $\omega \circ \iota_n \neq 0$ implies $t = 1$. Purity of $\omega \circ \iota_n$ follows easily from purity of $\omega$.

We have established a bijective correspondence $\bigsqcup_{n=1}^\infty \sP(\fA_n)_{\textup{w}^*} \rightarrow \sP(\fA)_{\textup{w}^*}$. Continuity of this map follows from the universal property of the disjoint union and the fact that each map $f_n:\sP(\fA_n)_{\textup{w}^*} \rightarrow \sP(\fA)_{\textup{w}^*}$, $f_n(\omega) = \omega \circ \pi_n$ is weak$^*$ continuous. We show that each $f_n$ is open, from which it follows from the definition of the disjoint union topology that the map $\bigsqcup_{n=1}^\infty \sP(\fA_n)_{\textup{w}^*} \rightarrow \sP(\fA)_{\textup{w}^*}$ is open, and therefore a homeomorphism. Given $\omega \in \sP(\fA_n)$, choose $A \in \fA_n$ such that $\abs{\omega(A)} > 1/2$. Then $U = \qty{\psi \in \sP(\fA): \abs{\psi(\iota_n(A)) - \omega(A)} < 1/2}$ is a neighborhood of $\omega \circ \pi_n$ contained in $f_n(\sP(\fA_n))$ since $\psi \in U$ implies $\psi \circ \iota_n \neq 0$. Thus, $f_n(\sP(\fA_n))$ is open. For any basis neighborhood $U_{\varepsilon, \omega, A_1,\ldots, A_k} \subset \sP(\fA_n)_{\textup{w}^*}$ around $\omega \in \sP(\fA_n)_{\textup{w}^*}$, we have  
\[
f_n(U_{\varepsilon, \omega, A_1,\ldots, A_k}) = f_n(\sP(\fA_n)) \cap U_{\varepsilon, \omega \circ \pi_n, \iota_n(A_1),\ldots, \iota_n(A_k)}.
\]
This exhibits $f_n(U_{\varepsilon, \omega, A_1,\ldots, A_k})$ as an open set in $\sP(\fA)_{\textup{w}^*}$, so $f_n$ is open.
\end{proof}

Given that there may be unphysical superselection sectors for a particular physical system, it is reasonable to ask whether $\sP(\fA)_{\textup{w}^*}$ remains path-connected after some superselection sectors are removed. For unital, separable, simple $C^*$-algebras, the answer is yes. In fact, one may find a path between arbitrary pure states $\psi, \omega \in \sP(\fA)$ that remains in $\sP_\psi(\fA)$ until the path reaches $\omega$ at the endpoint. This is a trivial consequence of Theorem \ref{thm:superselection_sector_equivalences} and the following remarkable theorem of Kishimoto, Ozawa, and Sakai.

\begin{thm}[{\cite[Thm.~1.1]{kishimoto_ozawa_sakai_2003}}]
Let $\fA$ be a separable $C^*$-algebra. If $\psi, \omega \in \sP(\fA)$ and $\ker \pi_\psi = \ker \pi_\omega$, then there exists an automorphism $\alpha \in \Aut(\fA)$ and a continuous family of unitaries $U:[0,\infty) \rightarrow \Unitary(\fA)$, $t \mapsto U_t$ such that $U_0 = I$, $\psi \circ \alpha = \omega$, and
\[
\alpha(A) = \lim_{t \rightarrow \infty} U_t^*A U_t 
\]
for all $A \in \fA$.
\end{thm}

\section{Geometry of the pure state space}\label{sec:diffstructure}

In \S\ref{subsec:kaehlerstruct} below, we will prove that the pure state space $\sP(\fA)$ has the structure of a K\"ahler manifold. However, we start by giving some background on the topology and geometry of infinite dimensional projective Hilbert spaces.

\subsection{Background on projective Hilbert spaces}
\label{sec:background-projective-hilbert-spaces}
Quantum mechanical states are traditionally described as rays in a Hilbert space $\hilbH$ or in other words as elements of 
the \emph{projective Hilbert space} 
\[\bbP \hilbH = \qty{\bbC \Psi: \Psi \in \hilbH \setminus \qty{0}}\] 
consisting  of
one-dimensional subspaces. 
The topology and differential geometry of projective Hilbert space will play an important role in the
fiberwise GNS construction of Section~\ref{sec:hilbert-bundle}. We review some fundamental concepts and constructions here. See \cite{CirelliNormalPureStatesVonNeumannAlgebra,CirelliHamiltonianVectorFieldsQM,CirelliPureStatesQMKaehlerBundles,AshtekarSchilling,FreedWignersTheorem,CiaglaJostSchwachhoefer} for related results. 

We endow $\bbP \hilbH$ with the quotient topology with respect to the canonical projection
$p: \hilbH \setminus \qty{0} \rightarrow \bbP \hilbH$, $p(\Psi) = \bbC \Psi$; this is the same as the quotient topology obtained from
the restricted projection $p: \bbS \hilbH \rightarrow \bbP \hilbH$, where $\bbS \hilbH$ denotes the unit sphere of $\hilbH$.
Given two rays $\bbC \Psi, \bbC \Omega \in \bbP \hilbH$, we define the \emph{ray product} by
\[
\ltrans \bbC \Psi, \bbC \Omega \rtrans = \frac{\abs{\ev{\Psi, \Omega}}}{\norm{\Psi} \norm{\Omega}},
\]
which is clearly independent of the representatives $\Psi, \Omega \in \hilbH \setminus \qty{0}$. The square of the ray product is called the \emph{transition probability} between $\bbC \Psi$ and $\bbC \Omega$.
We now define three metrics on $\bbP \hilbH$ and relate them to each other and to the ray product. The first
one is the \emph{chordial distance} or \emph{chordial metric} which for two  rays $\bbC \Psi, \bbC \Omega$ with
representing  vectors $\Psi, \Omega$ of norm $1$ is given by
\[
  d_{\textup{chd}} (\bbC \Psi, \bbC \Omega) =
  \inf_{\lambda \in \Unitary(1)} \norm{\Psi - \lambda \Omega} \ . 
\]
The second metric is the \emph{Fubini--Study distance}. It is defined by
\[
  d_{\textup{FS}} (\bbC \Psi, \bbC \Omega) = \arccos \ltrans \bbC \Psi, \bbC \Omega \rtrans \ . \quad 
\]
Finally, we define the \emph{gap metric}
\[
d_{\textup{gap}}(\bbC\Psi, \bbC \Omega) = \norm{P(\bbC \Psi) - P(\bbC \Omega)}
\]
where $P(\bbC \Psi) \in \fB(\hilbH)$ is the orthogonal projection onto $\bbC \Psi$.

\begin{prop}\label{prop:PH_gap_metric}
  Let $\hilbH$ be a nonzero complex Hilbert space. Then the following holds true for the metric
  structures on $\bbP\hilbH$. 
  \begin{enumerate}[{\rm (i)}] 
  \item\label{ite:gap}
    The chordial metric $d_{\textup{chd}}$ is complete and
    induces the quotient topology on $\bbP\hilbH$  with respect to the canonical projection
    $p: \bbS \hilbH \rightarrow \bbP \hilbH$. Moreover, the chordial metric  satisfies the formula 
  \begin{equation}
    \label{eq:relation-transition-metric-amplitudes}
    d_{\textup{chd}}(\bbC \Psi, \bbC \Omega)^2 = 2 \big( 1 - \ltrans \bbC \Psi, \bbC \Omega \rtrans \big)
    \quad \text{for all } \Psi, \Omega \in \hilbH \setminus \qty{0} \ .
  \end{equation}
  \item\label{ite:FS}
  The Fubini--Study distance is equivalent to the chordial metric. More precisely,
  \begin{equation}
  \label{eq:equivalence-metric-projective-space-fubini-study-distance}
    d_{\textup{chd}}(\scrk, \scrl) \leq  d_{\textup{FS}}  (\scrk, \scrl)
    \leq  \frac{\sqrt{2} \pi}{4} \, d_{\textup{chd}} (\scrk, \scrl)  \quad \text{for all } \scrk, \scrl \in \bbP \hilbH \ .
  \end{equation}
  \item\label{ite:op}
    The map $P : \bbP \hilbH \to \frB (\hilbH)$ which associates to every ray $\scrl$ the
    orthogonal projection onto it is a bi-Lipschitz embedding when $\bbP \hilbH$ is endowed
    with the chordial metric and the image is endowed with the  metric induced by the operator norm
    on $\frB (\hilbH)$. In other words, the chordial metric and the gap metric $d_{\textup{gap}}$
    on $\bbP \hilbH$ are equivalent. More precisely, the following estimate and equality holds
   for all $\scrk, \scrl \in \bbP \hilbH$:
   \begin{equation}
   \label{eq:equivalence-metric-projective-space-gap-metric}
    \frac{1}{\sqrt{2}} d_{\textup{chd}}(\scrk, \scrl) \leq  d_{\textup{gap}}  (\scrk, \scrl)
    = \sqrt{1 - \ltrans \scrk, \scrl\rtrans^2} 
    \leq  d_{\textup{chd}} (\scrk, \scrl)  \ .
  \end{equation}
  \end{enumerate}
\end{prop}

\begin{proof}
(\ref{ite:gap})
    By definition, $d_{\textup{chd}}$ is non-negative and symmetric.
    To show positive definiteness assume that $d_{\textup{chd}}(\C\Psi,\C\Omega) = 0$ with
    $\Psi,\Omega\in \bbS\hilbH$. By compactness
    of $\Unitary(1)$, there exists $\lambda \in \Unitary(1)$ such that
    $\| \Psi -\lambda \Omega\|=0$. Hence $\C\Psi = \C\Omega$. To verify the triangle inequality
    let $\scrk,\scrl, \scrj\in \bbP\hilbH$ and choose a representative $\Phi\in \scrj\cap\bbS\hilbH$.
    Then 
    \begin{align}
    \nonumber  
      d_{\textup{chd}} (\scrk, \scrl) & = 
      \inf \left\{ \| \Psi -\Omega\| :\Psi\in\scrk\cap \bbS\hilbH , \: \Omega \in \scrl\cap \bbS\hilbH \right\} \\
      \nonumber  & \leq 
      \inf \left\{ \| \Psi -\Phi\| : \Psi \in \scrk\cap \bbS\hilbH \right\}
      + \inf \left\{ \| \Phi -\Omega\| : \: \Omega \in \scrl\cap \bbS\hilbH \right\} \\
      \nonumber  & =   d_{\textup{chd}} (\scrk, \scrj) +  d_{\textup{chd}} (\scrj, \scrl) \ . 
    \end{align}
    Hence $d_{\textup{chd}}$ is a metric. 
    Next we prove formula \eqref{eq:relation-transition-metric-amplitudes}.
    Given $\Psi,\Omega\in  \bbS \cH$, observe that for all $\lambda \in \Unitary(1)$ the estimate
  \[
    \|\Psi - \lambda \Omega\|^2 =
    2 \left( 1-  \Re \left( \lambda \langle {\Psi}, {\Omega} \rangle \right) \right)
    \geq  2 \left( 1 -  | \langle {\Psi}, {\Omega} \rangle | \right) 
  \]
  holds true. Putting
  \[
    \lambda_0 =
    \begin{cases}
      \frac{\ev{\Omega, \Psi}}{\abs{\ev{\Psi, \Omega}}} & \text{for }
      \langle {\Psi}, {\Omega} \rangle \neq 0 \ , \\
      1 & \text{for } \langle {\Psi}, {\Omega} \rangle = 0 \  ,
    \end{cases}
  \]
  then minimizes the functional $\| {\Psi} -\lambda {\Omega} \|$, hence 
  \[
    d_{\textup{chd}}(\bbC \Psi, \bbC \Omega)^2 = \|{\Psi} -\lambda_0 {\Omega}\|^2 =
    2\left( 1 -  |\langle {\Psi}, {\Omega} \rangle | \right) =
    2\big( 1 -  \ltrans \bbC \Psi, \bbC \Omega \rtrans \big) \ . 
  \]

  Given $\Psi \in \bbS \cH$, the function $d_{\textup{chd}}(\bbC \Psi, -):\bbP \cH \rightarrow \qty[0,\sqrt{2}]$ is continuous  because by \eqref{eq:relation-transition-metric-amplitudes}
  composition with $p:\bbS \cH \rightarrow \bbP \cH$ is continuous. Hence the metric topology is coarser than the quotient topology. Conversely, it follows from the definition of $d_{\textup{chd}}$ that 
  \[
  \bbB_{\varepsilon,{\textup{chd}}}(\bbC \Psi) = p\qty(\bbB_{\varepsilon}(\Psi) \cap \bbS \cH)
  \]
  for any $\Psi \in \bbS \cH$ and $\varepsilon > 0$, where $\bbB_{\varepsilon,{\textup{chd}}}(\bbC \Psi)$ is the ball of radius $\varepsilon$ centered on $\bbC \Psi$ with respect to
  the chordial metric and similarly $\bbB_\varepsilon(\Psi)$ is the ball of radius $\varepsilon$ centered on $\Psi$.
  This proves that the quotient topology is coarser than the metric topology, and also that $p$ is an open map.

  To verify completeness observe first that for a given element $\Psi \in  \bbS \hilbH$ and ray
  $\scrl \in \bbP \hilbH$ with $\ltrans \bbC \Psi, \scrl \rtrans \neq 0$ there exists a unique representative
  $\Omega_{\Psi,\scrl} \in \bbS \hilbH\cap \scrl$ 
  such that $\langle \Psi, \Omega_{\Psi,\scrl} \rangle = \ltrans \bbC \Psi, \scrl \rtrans$.
  Now assume that $(\scrl_n)_{n\in \bbN}$   is a Cauchy sequence in $\bbP \hilbH$.
  Then there exists a strictly increasing sequence of natural numbers $(n_k)_{k\in \bbN}$ such that 
    \[
       d_{\textup{chd}} (\scrl_n ,  \scrl_m) < \frac{1}{2^{k+1}} \quad \text{for all } n,m \geq n_k \ .
    \]
  Pick a representative $\Omega_0 \in  \scrl_{n_0} \cap \bbS \hilbH$ and
  define the sequence $(\Omega_k)_{k\in \bbN}$ of unit vectors recursively by
  \[
    \Omega_{k+1} =  \Omega_{\Omega_k,\scrl_{n_{k+1}}} \ .
  \]
  Then, for all $k\in \N$,
  \begin{equation*}
    \begin{split}
      \| \Omega_{k+1}- \Omega_k \| \, & =
      \sqrt{2(1-\Re \langle \Omega_{k+1},\Omega_k \rangle)} =
      \sqrt{2(1- \ltrans \scrl_{n_{k+1}}, \scrl_{n_k} \rtrans )} = \\
      & = d_{\textup{chd}}( \scrl_{n_{k+1}}, \scrl_{n_k} ) < \frac{1}{2^{k+1}} \ . 
    \end{split}
  \end{equation*}
  So $(\Omega_k)_{k\in \bbN}$ is a Cauchy sequence in $\bbS\hilbH$, hence convergent to a vector $\Omega \in \bbS\hilbH$.
  Let $\scrl =\bbC\Omega$. Since $d_{\textup{chd}}(\scrl_{n_k},\scrl)\leq \| \Omega_k - \Omega \|$, the subsequence
  $(\scrl_{n_k})_{k\in \bbN}$  converges to $\scrl$, hence $(\scrl_n)_{n\in \N}$ does so too. Therefore
  $(\bbP\hilbH,d_{\textup{chd}})$ is  a complete metric space.
%
%
  
  (\ref{ite:FS}) Consider the function 
  \[
    f:[0,1] \rightarrow \bbR, \quad f(x)  =
       \frac{\arccos x}{\sqrt{2(1 - x)}} \text{ if } x \in [0, 1) \: \text{ and }  \:
       f (1) = 1 \ .
  \]
  This function is continuous on $[0,1]$ and differentiable on $(0,1)$, with
  \[
    f'(x) = \frac{1}{\sqrt{2}} \, (1 - x)^{- 3/2} \left( \frac{\arccos x}{2} - \sqrt{\frac{1 - x}{1 + x}} \right) \ .
  \]
  Given $x \in (0, 1)$, put $\theta = \arccos x$.
  Using trigonometric power reducing identities yields
  \[
  f'(x) = \frac{1}{\sqrt{2}} \, (1 - x)^{- 3/2} \left( \frac{\theta}{2} - \tan\qty(\frac{\theta}{2}) \right) < 0 \ .
  \]
  Therefore $f$ is monotonically decreasing, so $1 = f(1) \leq f(x) \leq f(0) = \frac{\sqrt{2}\pi}{4}$ for all $x \in [0,1]$.
  Since
  \[
    f(\ltrans \bbC \Psi, \bbC \Omega \rtrans) =
    \frac{d_\tn{FS}( \bbC \Psi, \bbC \Omega )}{d_\tn{chd}( \bbC \Psi, \bbC \Omega )}
  \]
  provided $\ltrans \bbC \Psi, \bbC \Omega \rtrans < 1$, this proves
  \eqref{eq:equivalence-metric-projective-space-fubini-study-distance}.

  (\ref{ite:op})
   The operator norm distance of  $P( \scrk)$ and $P( \scrl)$ is given by 
   \begin{equation}
     \label{eq:definition-operator-norm-difference-projections}
     \left\| P( \scrk) - P( \scrl) \right\| = \sup\limits_{\Phi \in \bbS\hilbH}
     \left\| \big( P( \scrk) - P( \scrl)\big)\Phi \right\| \ . 
   \end{equation}
   Pick normalized representatives $\Psi \in  \scrk$ and  $\Omega \in  \scrl$. After 
   possibly multiplying $\Omega$ by a complex number of modulus $1$ one can assume that 
   $\ev{\Psi,\Omega} = \ltrans \scrk, \scrl \rtrans \geq 0$.
   If  $\Psi$ and $\Omega$ are linearly dependent then  $ \scrk$ and $ \scrl$ coincide 
   and \eqref{eq:equivalence-metric-projective-space-gap-metric} is trivial. So assume that $\Psi$ and $\Omega$ are linearly independent. 
   First we want to show that 
   \begin{equation}
     \label{eq:estimate-operator-norm-difference-projections-normalized-vectors}
     \left\| \big( P( \scrk) - P( \scrl)\big)\Phi \right\|^2 \leq 1 - \ev{\Psi, \Omega}^2
     \quad \text{for all } \Phi \in \bbS\hilbH \ .
   \end{equation}
   To this end expand $\Phi = \Phi^\parallel + \Phi^\perp$, where 
   $\Phi^\parallel$ lies in the plane spanned by $\Psi$ and  $\Omega$ and $\Phi^\perp$ is perpendicular to that
   plane. Then 
   \[
      \big( P( \scrk) - P( \scrl)\big)\Phi =   \ev{\Psi,\Phi} \Psi  - \ev{\Omega,\Phi} \Omega =
      \langle \Psi,\Phi^\parallel \rangle \Psi  - \langle \Omega,\Phi^\parallel\rangle \Omega = \big( P( \scrk) - P( \scrl)\big)\Phi^\parallel \ .
   \]
   Hence it suffices to verify \eqref{eq:estimate-operator-norm-difference-projections-normalized-vectors} for 
   $\Phi \in \bbS\hilbH \cap \vecspan\qty{\Psi, \Omega}$. Observe that there exist unique elements
   $ \theta \in [0, \frac \pi 2]$ and  $\mu\in \bbS^1$ such that 
   $\ev{\Psi,\Phi} = \overline{\mu} \cos \theta$. One can then find a normalized vector
   $\Omega^\perp \in \vecspan\qty{\Psi,\Omega}$ perpendicular to $\Psi$ such that 
   \[
     \mu \Phi = \cos \theta \, \Psi + \sin \theta \, \Omega^\perp \ . 
   \]
   Note that
   \[
     \Omega = \ev{\Psi,\Omega} \Psi + \ev{\Omega^\perp,\Omega} \Omega^\perp \quad\text{and}\quad 
     \abs{\ev{\Omega^\perp ,\Omega}}^2 = 1 - \ev{\Psi,\Omega}^2 \ .
   \]
   Now compute 
   \begin{equation*}
     \begin{split}
       \big\| \big( P( \scrk) & - P( \scrl)\big)\Phi \big\|^2  =  
       \left\| \big( P( \scrk) - P( \scrl)\big)\mu \Phi \right\|^2 = 
       \abs{  \ev{\Psi,\mu \Phi} \Psi  - \ev{\Omega,\mu \Phi} \Omega }^2 = \\
       = \, & \abs{\ev{\Psi, \mu \Phi}}^2 - 2 \ev{\Psi,\Omega} \, \Re \left( \ev{\Psi, \mu \Phi} \ev{\mu \Phi, \Omega} \right) +
       \abs{\ev{\Omega, \mu \Phi}}^2 = \\
        = \, & \cos^2 \! \theta -2  \cos \theta \, \ev{\Psi,\Omega}\left( \cos \theta \,  \ev{\Psi,\Omega}  
         +  \sin \theta \, \Re \ev{\Omega^\perp, \Omega} \right) \, \\ 
         & +  \cos^2 \! \theta \, \ev{\Psi,\Omega}^2 + 2  \cos \theta \, \ev{\Psi,\Omega} \, \sin \theta \, 
         \Re \ev{\Omega^\perp,\Omega}  + \sin^2 \! \theta \, \abs{\ev{\Omega^\perp,\Omega}}^2  = \\
        = \, &  1 -   \ev{\Psi,\Omega}^2   \ . 
     \end{split}
   \end{equation*}
   This proves \eqref{eq:estimate-operator-norm-difference-projections-normalized-vectors}, but also implies
   by \eqref{eq:definition-operator-norm-difference-projections} that  
   \[
    \left\| P( \scrk) - P( \scrl) \right\|^2 =  1 -   \ev{\Psi,\Omega}^2 = 1 - \ltrans \scrk, \scrl \rtrans^2 \ .
   \]
   The claim \eqref{eq:equivalence-metric-projective-space-gap-metric} now follows by \eqref{eq:relation-transition-metric-amplitudes}
   and the inequality
   \[
      1 - x \leq 1 - x^2 \leq 2 (1-x) \quad \text{for all } x \in [0,1] \ . \qedhere
   \]
\end{proof}

\begin{cor}\label{cor:PH_superselection_metric_equivalence}
  Let $\fA$ be a nonzero $C^*$-algebra and $(\hilbH, \pi)$ be a nonzero irreducible representation.
  Given $\Psi \in \bbS \hilbH$ denote by $\psi$ the vector state represented by $\Psi$.
  Then the  map $r: \bbP \hilbH \rightarrow \sP_\pi(\fA)_{\textup{n}}$, $\bbC \Psi \mapsto \psi$ is a bi-Lipschitz
  bijection, hence an isomorphism of uniform spaces, where $\bbP\hilbH$ and $\sP_\pi(\fA)$ are endowed with the gap
  metric and the canonical metric, respectively. 
\end{cor}

  For later purposes, we call   $r: \bbP \hilbH \rightarrow \sP_\pi(\fA)_{\textup{n}}$
  the uniform isomorphism \emph{associated} to the irreducible representation  $(\hilbH, \pi)$.

\begin{proof}
Note that the map $\bbC \Psi \mapsto \psi$ is well-defined, since multiplying $\Psi \in \bbS \hilbH$ by a unimodular complex number does not change the corresponding vector state. By Proposition \ref{prop:pure_folia_vector_state}, the map is a bijection. By Theorem \ref{thm:transition_probability} and Proposition \ref{prop:PH_gap_metric}, we have
\begin{equation}\label{eq:gap_vs_norm}
  d_{\textup{gap}} (\bbC \Psi, \bbC \Omega)^2 = 1 - \ltrans \bbC \Psi, \bbC \Omega \rtrans^2 =
  \frac 14 \norm{\psi - \omega}^2,
\end{equation}
where $\omega$ is the vector state corresponding to $\Omega \in \bbS \hilbH$. This implies that the map
$r$ is bi-Lipschitz.
\end{proof}

\begin{cor}\label{cor:superselection_sector_closed}
Let $\fA$ be a nonzero $C^*$-algebra. The union of an arbitrary collection of superselection sectors of $\fA$ is norm-closed in $\fA^*$.
\end{cor}

\begin{proof}
By Theorem \ref{thm:superselection_sector_equivalences}, each superselection sector coincides with $\sP_\pi(\fA)$ for some nonzero irreducible representation $(\hilbH, \pi)$. It then follows from Proposition \ref{prop:PH_gap_metric} and Corollary \ref{cor:PH_superselection_metric_equivalence} that each superselection sector is complete, hence closed, in $\fA^*$ in the norm topology. By Lemma \ref{lem:sector_is_open}, each Cauchy sequence in $\sP(\fA)$ is eventually in a single superselection sector. It follows that the union of an arbitrary collection of superselection sectors is complete, hence norm-closed.
\end{proof}

Corollary \ref{cor:superselection_sector_closed} is equivalent to \cite[Cor.~4.8]{KadisonLimitsofStates}. Kadison takes a different approach to the proof, but remarks that it may be proven with Theorem \ref{thm:transition_probability} as we have done in the proof of Corollary \ref{cor:PH_superselection_metric_equivalence}.

We now begin to consider the differential geometry of projective Hilbert space. 
We refer the reader to Appendix \ref{sec:banach-hilbert-manifolds} or \cite{LanIDM} for details
on infinite dimensional manifolds used henceforth.  

\begin{lem}\label{lem:PH_trivializations}
  Let $\hilbH$ be a Hilbert space of dimension $\geq 2 $ and fix $\Psi \in \bbS\hilbH$. Let
  $C_\Psi = (\bbC \Psi)^\perp$ and let $\bbB_1 (\C\Psi)$ be the open unit ball around $\bbC \Psi$
  with respect to the gap metric. The maps
  $\sigma_\Psi: \bbS \hilbH \setminus \qty(\bbS \hilbH \cap C_\Psi)\rightarrow C_\Psi\times \Unitary(1)$ and
  $\tau_\Psi:\bbB_1 (\C\Psi) \rightarrow C_\Psi$ given by
\[
\sigma_\Psi\qty(\Omega) = \qty(\frac{\Omega}{\ev{\Psi, \Omega}} - \Psi, \frac{\ev{\Psi, \Omega}}{\abs{\ev{\Psi, \Omega}}})
\qqtext{and} 
\tau_\Psi (\bbC\Omega) = \frac{\Omega}{\ev{\Psi,\Omega}} - \Psi 
\]
for $\Omega \in \bbS \hilbH \setminus \qty(\bbS \hilbH \cap C_\Psi)$
are well-defined homeomorphisms.
\end{lem}

\begin{proof}
  If $\Omega \in \bbS \hilbH \setminus (\bbS \hilbH \cap C_\Psi)$, then
  $\ev{\Psi, \Omega} \neq 0$ and
  \[
    \ev{\Psi, \frac{\Omega}{\ev{\Psi,\Omega}} - \Psi} = 1 - \ev{\Psi,\Psi} = 0 \ .
  \]
  So $\frac{\Omega}{\ev{\Psi,\Omega}} - \Psi \in C_\Psi$
  and $\sigma_\Psi$ is well-defined.
Continuity of $\sigma_\Psi$ is manifest. We now define a continuous map
$\sigma^{-1}_\Psi: C_\Psi \times \Unitary(1) \rightarrow \bbS \hilbH \setminus \qty(\bbS \hilbH \cap C_\Psi)$ by
\[
  \sigma^{-1}_\Psi(\Phi, \lambda) = \frac{\lambda (\Phi + \Psi)}{\sqrt{1 + \norm{\Phi}^2}} \ .
\]
It is straightforward to check that $\sigma_\Psi^{-1}$ is indeed a two-sided inverse
for $\sigma_\Psi$.

Next we consider $\tau_\Psi$. We have already shown that
$\frac{\Omega}{\ev{\Psi,\Omega}} - \Psi$ is perpendicular to $\Psi$, so $\tau_\Psi$
is well defined. Since composition of $\tau_\Psi$ with the canonical projection 
$p: \bbS \hilbH \setminus \qty(\bbS \hilbH \cap C_\Psi)  \to \bbB_1 (\bbC \Psi)$
is continuous, $\tau_\Psi$ is continuous as well.
Now put
\begin{equation}
  \label{eq:inverse-chart-projective-hilbert-space}
   \tau_\Psi^{-1}(\Phi) = \bbC(\Phi + \Psi) \quad\text{for } \Phi \in C_\Psi \ .
\end{equation} 
It is clear that $\Psi + \Phi \neq 0$ for such $\Phi$, so
$\bbC(\Psi + \Phi) \in \bbP \hilbH$ and the ray product
$\ltrans \tau^{-1}_\Psi(\Phi), \bbC \Psi \rtrans = \norm{\Psi + \Phi}^{-1}$
is positive. Hence
$\tau_\Psi^{-1}(\Phi) \in \bbB_1 (\C\Psi)$. Continuity of 
the map $\tau_\Psi^{-1}: C_\Psi \to \bbB_1 (\C\Psi)$ is immediate, likewise that it
is a two-sided inverse for $\tau_\Psi$.
\end{proof}

\begin{thm}\label{thm:holomorphic-atlas-projective-space}
  Let $\hilbH$ be a Hilbert space of dimension $\geq 2$ and let $\bbP \hilbH$ be its
  projective Hilbert space. The set of all maps $ \tau_\Psi : \bbB_1 (\C\Psi) \to C_\Psi$ with
  $\Psi \in \bbS\hilbH$  is a holomorphic atlas for $\bbP \hilbH$. Hereby,
  $\bbB_1 (\C\Psi)$, $C_\Psi$ and $\tau_\Psi$ are as in Lemma \ref{lem:PH_trivializations}.
\end{thm}

\begin{proof}
  Suppose to be given $\Psi, \Omega \in \bbS \hilbH$ such that
  $\bbB_1 (\C\Psi) \cap \bbB_1 (\C\Omega) \neq \varnothing$. Note that
  then $\ev{\Omega, \Phi + \Psi}\neq 0$ for 
  $\Phi \in \tau_\Psi(\bbB_1 (\C\Psi) \cap \bbB_1 (\C\Omega))$ 
  since $\C (\Phi +\Psi) \in \bbB_1 (\C\Omega)$ by Eq.~\eqref{eq:inverse-chart-projective-hilbert-space}
  and assumption. Lemma \ref{lem:PH_trivializations} then entails that for all such $\Phi$
  \[
    (\tau_\Omega \circ \tau_\Psi^{-1})(\Phi) =
    \frac{\Phi + \Psi}{\ev{\Omega, \Phi + \Psi}} - \Omega
    \ ,
  \]
  hence the transition map $\tau_\Omega \circ \tau_\Psi^{-1}: \tau_\Psi(\bbB_1 (\C\Psi) \cap \bbB_1 (\C\Omega)) \to\tau_\Omega(\bbB_1 (\C\Psi) \cap \bbB_1 (\C\Omega))  $ is holomorphic. 
\end{proof}

By the theorem, $\bbP \hilbH$ becomes a complex manifold
modeled on Hilbert spaces. Note that the sphere $\bbS\hilbH$ carries in a natural way
the structure of a real analytic manifold by the Regular Value Theorem for
Banach manifolds \cite[Theorem D]{GloFSIIDM} and since the map
$ \hilbH \setminus \{ 0\} \to \R_{>0}$, $\Omega  \mapsto \langle \Omega,\Omega\rangle$ is
a real analytic submersion. A real analytic atlas for $\bbS\hilbH$ is given by the set
of all $\sigma_\Psi$ with $\Psi \in \bbS\hilbH$ since the maps
\[
  (\sigma_\Psi, \| \cdot \|^2) : \hilbH \setminus C_\Psi
  \to C_\Psi \times \Unitary(1) \times \R_{>0} , \:
  \Omega \mapsto \left(\sigma_\Psi \left(\frac{\Omega}{\|\Omega\|}\right) , \|\Omega\|^2 \right)
\]
are real analytic diffeomorphisms.

\begin{thm}
  Let $\hilbH$ be a Hilbert space of dimension $\geq 2$.
  Then the canonical projection $p_{\bbS\hilbH}:\bbS \hilbH \rightarrow \bbP \hilbH$
  is real analytic and for each $\Psi\in \bbS \hilbH$ the map 
  \begin{align*}
    \rho_\Psi  \colon \bbS\hilbH \setminus (\bbS \hilbH \cap C_\Psi) \rightarrow
    \bbB_1 (\C\Psi) \times \Unitary(1), \:\:
    \Phi  \mapsto \qty(\bbC \Phi , \frac{\ev{\Psi, \Phi}}{\abs{\ev{\Psi, \Phi}}}) 
\end{align*}
is a real analytic local trivialization for $p$ with typical fiber
$\Unitary(1)$. 
Moreover, the set of pairs $(\bbB_1 (\C\Psi), \rho_\Psi )$ with
$\Psi \in \bbS \hilbH$ forms a bundle atlas, and $p_{\bbS\hilbH}: \bbS \hilbH \rightarrow \bbP \hilbH$
becomes a real analytic $\Unitary (1)$-bundle.
\end{thm}

\begin{proof}
We first note that $\rho_\Psi = (\tau_\Psi^{-1} \times \id_{\Unitary(1)}) \circ \sigma_\Psi$.
Indeed, observe that $p_{\bbS\hilbH}^{-1}(\bbB_1 (\C\Psi))$ coincides with 
$\bbS \hilbH \setminus \qty(\bbS \hilbH \cap C_\Psi)$ and that the diagram
\[
\begin{tikzcd}
\bbS \hilbH \setminus \qty(\bbS \hilbH \cap C_\Psi) \arrow[rr," (\tau_\Psi^{-1} \times id_{\Unitary(1)}) \circ \sigma_\Psi"]\arrow[dr,"p_{\bbS\hilbH}"'] && \bbB_1 (\C\Psi) \times \Unitary(1) \arrow[dl]\\
&\bbB_1 (\C\Psi)&
\end{tikzcd}
\]
commutes, as can be seen from the definition of $\sigma_\Psi$ and the formula for $\tau_\Psi$. Thus,
\[
(\tau_\Psi^{-1} \times id_{\Unitary(1)}) \circ \sigma_\Psi(\Phi) = \qty(\bbC \Phi, \frac{\ev{\Psi, \Phi}}{\abs{\ev{\Psi, \Phi}}}) = \rho_\Psi(\Phi).
\]
Commutativity of the diagram also shows that $p$ is real analytic since $\sigma_\Psi$ and
$\tau_\Psi^{-1}$ are so. 
Next we want to show that for given $\Psi, \Omega \in \bbS \hilbH$ the transition map
\[
  (\bbB_1 (\C\Psi) \cap \bbB_1 (\C\Omega)) \times \Unitary(1) \rightarrow \Unitary(1) : \:
  (\bbC \Phi, \lambda) \mapsto
  (\rho_{\Omega} \circ \rho_{\Psi}^{-1})(\bbC \Phi , \lambda)
\]
is real analytic. To this end check that
\begin{align*}
  \qty(\rho_{\Omega} \circ \rho_{\Psi}^{-1})(\C\Phi,\lambda) = \lambda \cdot \frac{\ev{\Omega, \Phi}}{\abs{\ev{\Omega, \Phi}}} \cdot \frac{\abs{\ev{\Psi, \Phi}}}{\ev{\Psi, \Phi}} \ .
\end{align*}
The right hand side is obviously real analytic as a function of
$\Phi \in \bbS\hilbH \setminus (C_\Psi \cup C_\Omega)$ since the 
inner product on a Hilbert space is real analytic because it is a real bilinear map. 
This entails the claim.
\end{proof}

\begin{cor}\label{cor:PH_mapsto_SH}
  Let $\hilbH$ be a Hilbert space of dimension $\geq 2$ and fix $\Psi \in \bbS \hilbH$. The function
  $\bbB_1 (\C\Psi) \rightarrow \bbS \hilbH$ which maps a ray $\scrl\in\bbB_1 (\C\Psi)$ to the unique
  $\Phi \in  \scrl \cap \bbS \hilbH$ such that  $\ev{\Phi, \Psi} > 0$ then is real analytic.
\end{cor}

\begin{proof}
  The stated map is the composition
  $\scrl \mapsto \sigma_\Psi^{-1}( \tau_\Psi( \scrl), 1)$
which is real analytic since $\tau_\Psi$ and $\sigma_\Psi^{-1}$ are.
\end{proof}

\begin{cor}\label{cor:purestate_mapsto_SH}
  Let $\fA$ be a $C^*$-algebra,  $\omega \in \sP(\fA)_{\textup{n}}$, and let 
  $\bbB_2(\omega)  \subset \sP(\fA)_{\textup{n}}$ be the open ball of radius $2$
  centered on $\omega$.  If $(\hilbH, \pi, \Omega)$ is the GNS representation of $\omega$,
  $\bbB_2(\omega) $ is contained in the space of pure $\pi$-normal states
  $\sP_\pi(\fA)_{\textup{n}}$ and the map
  $s: \bbB_2(\omega) \rightarrow \bbS \hilbH$, $\varphi \mapsto \Phi $,
  where $\Phi$ is the unique unit vector representing $\varphi$ and having
  $\ev{\Phi, \Omega} > 0$, is a norm continuous section of
  the canonical  projection $p_{\bbS \hilbH}: \bbS \hilbH \to \sP_\pi (\fA)_{\textup{n}}$.
  When $\sP_\pi (\fA)_{\textup{n}}$ is endowed with the unique complex manifold structure so that
  the canonical isomorphism of uniform spaces $r:\bbP\hilbH \to \sP_\pi (\fA)_{\textup{n}}$
  associated to the representation $(\hilbH, \pi)$  as in
  Corollary \ref{cor:PH_superselection_metric_equivalence} is biholomorphic,
  the thus defined section $s$ is real analytic.
  Moreover, the canonical injection
  $\sP_\pi (\fA)_{\textup{n}} \hookrightarrow \fA^*$
  then is real analytic and its tangent map is injective. 
\end{cor}

\begin{proof}
  We know $\bbB_2(\omega)$ is contained in $\sP_\pi(\fA)$ by Lemma \ref{lem:sector_is_open}, so
  by Corollary \ref{cor:PH_superselection_metric_equivalence} the restriction of $r^{-1}$ provides a
  norm continuous map $\bbB_2(\omega) \rightarrow \bbP \hilbH$.
  Equation \eqref{eq:gap_vs_norm} shows that the image is contained in the open unit ball about $\bbC \Omega$ with respect to the gap metric. Composing with the norm continuous map from Corollary \ref{cor:PH_mapsto_SH} yields the norm continuous section $s$. By construction, $s$ is real analytic when $r^{-1}$ is.
  To verify the last claim consider for $\Psi \in \bbS\hilbH$
  the chart $\tau_\Psi :\bbB_1 (\C\Psi) \to C_\Psi$
  from Theorem \ref{thm:holomorphic-atlas-projective-space}.
  The composition 
  \[
    F :C_\Psi \overset{\tau_\Psi^{-1}}{\longrightarrow} \bbP\hilbH
    \overset{r}{\longrightarrow} \sP_\pi (\fA)_{\textup{n}} \hookrightarrow \fA^*
  \]
  then maps $v\in C_\Psi$ to the linear functional
  \[
    F (v) : \fA\to \C, \: A\mapsto \frac{1}{\|\Psi+v\|^2}
    \langle \Psi+v, \pi(A) (\Psi+v)\rangle \ .
  \] 
  This is clearly a real analytic function in $v$. Let us determine its
  derivative. For $v \in C_\Psi$ consider the path
  $\gamma : \R \to C_\Psi$, $t\mapsto tv$ and compute
  \[
    T_0 F (v) = \left. \frac{d}{dt} \right|_{t=0} F (\gamma (t)) =
    \langle \Psi , \pi( \,\cdot\, ) v\rangle 
    + \langle v , \pi( \,\cdot\, ) \Psi\rangle \ .
  \]
  Assume that $v \neq 0$. Then $\Psi$ and $v$ are orthogonal to each and linearly independent.
  Hence by the Kadison transitivity theorem there exists an operator $A\in \cA$ such that
  $\pi (A)v = \Psi$ and $\pi(A)\Psi =v$. Hence
  $T_0 F (v) (A) = \|\Psi\|^2 + \|v\|^2 > 1$ and $T_0F$ has a trivial kernel.
  Therefore, the tangent map $T\sP_\pi (\fA)_{\textup{n}} \rightarrow T\fA^*$ 
  is injective as claimed.
\end{proof}

\begin{rem}\label{rem:injectivity-tangent-map-not-entailing-submanifold-infinite-dimensional-case}
  In the case where  $ \fA^*$ is finite dimensional, the preceding corollary
  entails that $\iota : \sP_\pi (\fA)_{\textup{n}} \hookrightarrow \fA^*$ is a
  real analytic embedding and that $\sP_\pi (\fA)_{\textup{n}}$  is a 
  submanifold of $ \fA^*$.
  The  corresponding argument can not be extended to
  the infinite dimensional case since for a smooth topological embedding
  $N\hookrightarrow M$ of infinite dimensional manifolds injectivity of the
  tangent map $TN \rightarrow TM$ in general neither entails
  that $N \hookrightarrow M$ is an immersion
  nor that  $N$ is a submanifold of $M$.  
\end{rem}

\subsection{The K\"ahler manifold structure on the pure state space} 
\label{subsec:kaehlerstruct}
By Proposition \ref{prop:decomposition-pure-state-space-superselection-sectors}, the superselection sectors
form the path connected components of the pure state space $\sP(\fA)$ endowed with the norm topology. 
Moreover, every superselection sector is open in $\sP(\fA)_{\textup{n}}$ by Corollary
\ref{cor:sector_is_open}. To define a holomorphic structure on $\sP(\fA)_{\textup{n}}$ it therefore
suffices to define one on each superselection sector separately. So let $\omega$ be a pure state
on $\fA$ and $S_\omega$ the superselection sector it defines. Then  $S_\omega$ coincides with the
space $\sP_{\pi_\omega}(\fA)_{\textup{n}}$ of pure $\pi_\omega$-normal states where as usual 
$(\hilbH_\omega , \pi_\omega, \Omega_\omega)$ denotes the GNS representation of $\omega$. 
According to Corollary \ref{cor:purestate_mapsto_SH}, the uniform isomorphism
$r_{\omega} :\bbP\hilbH_\omega \to \sP_{\pi_\omega} (\fA)_{\textup{n}}$ associated to the representation $\pi_\omega$
endows $S_\omega$ 
with a  holomorphic structure. It remains to show that the holomorphic structure is independent of 
the particular representative of the superselection sector. It suffices to verify
that for every other state $\psi \in S_\omega$ 
the ``transition'' map $r_{\omega,\psi} := r_{\omega}^{-1} \circ r_{\psi} : \bbP\hilbH_\psi \to \bbP\hilbH_\omega$ 
is holomorphic, where $r_{\psi}$ is the uniform isomorphism associated to the GNS representation $(\hilbH_\psi,\pi_\psi)$.
Since the GNS representations for $\omega$ and $\psi$ are unitarily equivalent, 
there exists a unique  unitary intertwiner
$U: \hilbH_\psi \to \hilbH_\omega$ between $\pi_\psi$ and $\pi_\omega$ so that
$U \Omega_\psi = \Omega_\omega$. 
Hence the transition map  $r_{\omega,\psi}$ is
given by $\C\Phi \mapsto \C (U\Phi)$. Next choose $\Omega \in \bbS\hilbH_\omega$ and $\Psi \in \bbS\hilbH_\psi$.
Let $\tau_\Omega$ and  $\tau_\Psi$ be the corresponding charts  of 
$\bbP\hilbH_\omega$ and $\bbP\hilbH_\psi$, respectively. In these charts, the
transition map $r_{\omega,\psi}$ has the form 
\begin{equation*}
  \begin{split}
    \tau_\Omega \circ r_{\omega,\psi} \circ \tau_\Psi^{-1} :
    \tau_\Psi \big(\bbB_1(\C\Psi)\cap \bbB_1(\C(U^{-1}\Omega))\big) \to C_\Omega , \:
    \Phi \mapsto \frac{U(\Phi +\Psi)}{\langle\Omega,U(\Phi +\Psi)\rangle} - \Omega \ ,
  \end{split}
\end{equation*}
which obviously is holomorphic in $\Phi$. Thus the pure state space $\sP(\fA)$ carries
a unique structure of a complex manifold such that for every irreducible representation
$(\hilbH,\pi)$ the associated canonical embedding $r : \bbP \hilbH \to \sP(\fA)$ is a biholomorphic map onto its image. We call this complex structure on $\sP(\fA)$ \emph{canonical}. 

Next, we equip $\sP(\fA)$ with a canonical K\"ahler structure.
More precisely, we have to construct a smooth hermitian form on $\sP(\fA)$ whose imaginary part is
symplectic. To this end recall that a hermitian 
metric on a complex Hilbert manifold $M$ is a smooth section $h$ of the tensor bundle
$T^0_2M \otimes \C$ of complexified $2$-fold covariant tensors on $M$ 
such that $h$ is fiberwise conjugate linear in the first argument, fiberwise complex linear in the second argument
and positive definite; see  \cite[Sec.\ 1.2]{Klingenberg} and Example \ref{ex:tangent-tensor-bundles} for definitions
and details on tensor bundles. 

Now let $\omega$ be a pure state on $\fA$ as before and consider
the smooth section $s_\omega: \bbB_2(\omega) \to \bbS\hilbH _\omega\subset \hilbH_\omega\setminus \{ 0 \}$
from Corollary \ref{cor:purestate_mapsto_SH} which associates to every state
$\varphi \in \bbB_2(\omega)$ the unique vector $\Phi \in \bbS\hilbH_\omega$ such that
$\langle \Phi,\Omega_\omega\rangle > 0$. Define
$h_\omega : \bbB_2(\omega) \to T^0_2 \sP(\fA)\otimes \C $ as the pullback
of the constant hermitian metric $\langle -,-\rangle$ on $\hilbH\setminus \{0\}$ via the section
$s_\omega$ that means put
\[
  h_{\omega} (\varphi) (X,Y) = \langle T_\varphi s_\omega (X), T_\varphi s_\omega (Y) \rangle \quad \text{for all }
  \varphi\in \bbB_2(\omega), \: X,Y \in T_\varphi\sP(\fA) \ .
\]
Then $h_\omega$ is a smooth hermitian metric on $\bbB_2(\omega)$ since $s$ is an immersion. 
Moreover, the imaginary part of $h_\omega$ is closed, since the imaginary part of
$\langle -,-\rangle$ is a symplectic bilinear form on $\hilbH_\omega$. Hence  $h_\omega$
endows $\bbB_2(\omega)$ with a K\"ahler manifold structure. It remains to show that 
for another pure state $\psi$ the hermitian metrics $h_\psi$ and $h_\omega$ coincide on the
overlap of their domains.  So assume that $\bbB_2(\omega) \cap \bbB_2(\psi)$
is non-empty. This can only be the case when $\omega$ and $\psi$ are in the same superselection
sector, or in other words if their exists a unitary intertwiner $U: \hilbH_\psi \to \hilbH_\omega$ between
the GNS representations $\pi_\psi$ and $\pi_\omega$ as above. Then $s_\omega = U \circ s_\psi$
by definition of the sections $s_\omega$ and $s_\psi$, hence
\begin{equation*}
  \begin{split}
    h_{\omega} (\varphi) (X,Y)&  = \langle T_\varphi s_\omega (X), T_\varphi s_\omega (Y) \rangle  =
  \langle U T_\varphi s_\psi (X), U T_\varphi s_\psi (Y) \rangle   \\
   & = \langle  T_\varphi s_\psi (X),  T_\varphi s_\psi (Y) \rangle = h_{\psi} (\varphi) (X,Y)
  \end{split}
\end{equation*}
for all  $\varphi\in \bbB_2(\omega)\cap \bbB_2(\psi)$ and $ X,Y \in T_\varphi\sP(\fA)$.  
The local hermitian metrics therefore glue together to a global one which we denote by $h$.
Since each of the above local sections $s_\omega$ is a  riemannian immersion, the canonical
projection $p_{\bbS\hilbH}:\bbS\hilbH \to \sP_\pi(\fA)_\textup{n}$ then becomes a riemannian submersion. 
This uniquely determines the riemannian structure on $\sP(\fA)$, hence also $h$ 
since  a hermitian  metric  is uniquely determined by its real part. 
We now obtain the following result.

\begin{thm}
  Endowed with its canonical complex structure, the pure state space $\sP(\fA)$ of a $C^*$-algebra $\fA$
  carries a natural hermitian metric $h$ turning it into a possibly infinite dimensional K\"ahler manifold.
  The hermitian metric $h$ is uniquely determined by the requirement that for every irreducible
  representation $(\hilbH,\pi)$ of  $\fA$ the canonical projection
  $p_{\bbS\hilbH} : \bbS\hilbH \to \sP_\pi(\fA)_\textup{n}$, $\Psi \mapsto r(\C\Psi)$ is a
  riemannian submersion. 
\end{thm}

Following common language in K\"ahler geometry we call the real part of $h$ the \emph{Fubini--Study metric}
on  $\sP(\fA)$. Let us finally show that on each superselection sector $\sP_\pi(\fA)_\textup{n}$ - with $(\hilbH,\pi)$
an irreducible representation as before - the geodesic distance $\delta_\textup{FS}$ of the Fubini-Study metric coincides with the
Fubini--Study distance $d_\textup{FS}$ transferred from $\bbP\hilbH$ to $\sP_\pi(\fA)_\textup{n}$
via the associated uniform isomorphism $r$. Let $\omega$ and $\varphi$ be two distinct states in  $\sP_\pi(\fA)_\textup{n}$
and choose $\Omega,\Phi \in \bbS\hilbH $ which project to $\omega$ and $\varphi$, respectively. After possibly multiplying 
$\Phi$ with a complex number of modulus $1$ we can assume that $\langle \Omega,\Phi\rangle =  |\langle \Omega,\Phi\rangle|$.
By assumption, $\Omega$ and $\Phi$ then span a real plane $E$ in $\hilbH$ whose intersection with the sphere $\bbS\hilbH $
is totally geodesic since it is the fixed point manifold of reflection at $E$ which is a linear isometric isomorphism;
see \cite[1.10.15 Theorem]{Klingenberg}.
Let $\alpha = \arccos \langle \Omega,\Phi\rangle $ and $\Psi \in P$ be the normal vector perpendicular to $\Omega$ such that
$\Phi = \cos \alpha \cdot \Omega + \sin \alpha \cdot \Psi$. Now note that the intersection of $P$ with $\bbS\hilbH$ is a great
circle $C$. The unique shortest geodesic from $\Omega$ to $\Phi$ in $C$ is given by the path
$\gamma: [0,1] \to P \cap \bbS\hilbH$, $t\mapsto \cos (t \alpha) \cdot \Omega + \sin (t \alpha) \cdot \Psi$.
Since the great circle $C$ is totally geodesic, $\gamma$ is also a shortest geodesic in $\bbS\hilbH$ connecting
$\Omega$ and $\Phi$. Its length is obviously $\alpha$. Hence the claimed equality
\begin{equation}
  \label{eq:fubini-study-distance-coincides-geodesic-length}
  d_\textup{FS} (\C\Omega,\C\Phi) = \arccos \ltrans \bbC \Psi, \bbC \Omega \rtrans =
  \delta_\textup{FS} (\omega,\varphi) \:
\end{equation}
holds true if we can yet show that the path $\gamma$ is horizontal; see \cite[1.11.11 Corollary]{Klingenberg}.
To this end fix $t_{0} \in [0,1]$
and consider the path $\tau : [-\pi,\pi]\to \bbS\hilbH$, $s \mapsto e^{is}\gamma(t_0)$. The derivative
$\dot{\tau} (0) = i \, \gamma (t_0)$ then spans the kernel of the tangent map
$T_{\gamma (t_0)} p :T_{\gamma (t_0)} \bbS\hilbH\to T_{r\gamma(t_0)}$. We need to show
$\Re \langle \dot{\tau} (0) , \dot{\gamma} (t_0)\rangle=0$. Since $\| \gamma (t) \|^2 =1$ for all $t\in [0,1]$,
we already know that $\Re \langle \gamma (t_0),\dot{\gamma} (t_0)\rangle =0$. Since the inner product
$\langle \Omega,\Phi\rangle$ is real, the inner product of $\Omega$ and $\Psi$ is so too, hence
$\langle \gamma (t_0),\dot{\gamma} (t_0)\rangle = \Re \langle \gamma (t_0),\dot{\gamma} (t_0)\rangle =0$. Therefore,
\[
  \Re \langle \dot{\tau} (0) , \dot{\gamma} (t_0)\rangle =
  - \Im \langle  \gamma (t_0) , \dot{\gamma} (t_0)\rangle = 0 
\]
and the claim is proved.

\begin{rem}
  Originally, Cirelli, Lanzavecchia and coauthors showed in their work from the 80ies that the pure state
  space of a $C^*$-algebra carries in a natural way the structure of a K\"ahler manifold, albeit the
  proof is scattered over several papers \cite{CirelliHamiltonianVectorFieldsQM,CirelliNormalPureStatesVonNeumannAlgebra,CirelliPureStatesQMKaehlerBundles}. 
  Their work is related to and builds upon the geometric characterization of $C^*$-algebraic state spaces 
  by Alfsen, Hanche-Olsen and Shultz \cite{AlfsenHancheOlsenShultz}.
  In a certain sense, unravelling the K\"ahler manifold structure of the pure state space of
  a $C^*$-algebra can be understood  as  a step forward in Connes' program of noncommutative geometry
  \cite{ConNG} which has the goal to
  describe $C^*$-algebras by geometric means and to use $C^*$-algebras for the geometric description of
  spaces as they appear for example in quantum mechanics where a direct geometric intuition is lacking. 
\end{rem}


\section{Continuous Kadison transitivity theorems}\label{sec:contkadison}

The Kadison transitivity theorem states that whenever a $C^*$-algebra $\fA$ acts
irreducibly on a Hilbert space $\cH$ there exists for every pair of $n$-tuples
of vectors $x_1,\ldots, x_n$ and $y_1,\ldots, y_n$ in $\cH$ such that $x_1,\ldots, x_n$
are linearly independent an element $A \in \fA$ such that $Ax_k = y_k$ for
$k=1,\ldots ,n$; see \cite[Thm.\ 5.4.3]{KadisonRingroseI}. 
However, the solution to this problem is in general not unique. The question arises, then, whether it is possible to choose the solutions $A \in \fA$ so as to depend continuously on the initial data $x_1,\ldots, x_n$ and $y_1,\ldots, y_n$. This is a problem amenable to the theory of selections developed by Ernest Michael in the 1950s, and indeed we use the Michael selection theorem to provide an affirmative answer to our question. In section \ref{sec:main_results}, we recall the necessary terminology and results from Michael's original work on selections \cite{MichaelSelection}, then we prove our main results in Theorems \ref{thm:continuous_Kadison} and \ref{thm:continuous_Kadison_unitary}. In section \ref{sec:principal_fiber_bundles} we use Theorem \ref{thm:continuous_Kadison_unitary} to prove that $p_{\Unitary(\fA)}:\Unitary(\fA) \rightarrow \sP_\omega(\fA)$, $p_{\Unitary(\fA)}(U) = U \cdot \omega$ has the structure of a principal $\Unitary_\omega(\fA)$-bundle, where $\Unitary_\omega(\fA) = \qty{U \in \Unitary(\fA): U \cdot \omega = \omega}$, for any unital $C^*$-algebra $\fA$ and pure state $\omega \in \sP(\fA)$. We provide a few examples where this bundle is nontrivial.

\subsection{Main results}\label{sec:main_results}

The key ingredient in proving our continuous Kadison transitivity theorems is the Michael selection theorem. We provide this result and the necessary definitions below.

\begin{defn}[{\cite{MichaelSelection}}]
Let $X$ and $Y$ be topological spaces. A \textit{carrier} is a function $\phi : X \rightarrow \wp_+(Y)$, where $\wp_+(Y)$ is the set of all nonempty subsets of $Y$. A \textit{selection} for $\phi$ is a continuous function $S:X \rightarrow Y$ such that $S(x) \in \phi(x)$ for all $x \in X$. The carrier $\phi$ is \textit{lower semicontinuous} if for every open set $V \subset Y$, the set
\[
\qty{x \in X: \phi(x) \cap V \neq \varnothing}
\]
is open in $X$. Equivalently, $\phi$ is lower semicontinuous if for every $x_0 \in X$, $y_0 \in \phi(x_0)$, and neighborhood $V$ of $y_0$, there exists a neighborhood $U$ of $x_0$ such that $\phi(x) \cap V \neq \varnothing$ for all $x \in U$. 
\end{defn}

We will use the latter description of lower semicontinuity in our proofs. If $Y$ is metrizable, as it will be for our applications, the neighborhood $V$ may be taken to be a ball of radius $\varepsilon > 0$ centered on $y_0$. The space $X$ will always be metrizable in our applications as well. We now state the Michael selection theorem for reference.

\begin{thm}[{\cite[Thm.\ 3.2$''$]{MichaelSelection}}]
Let $X$ be a paracompact Hausdorff space and let $Y$ be a real or complex Banach space. If $\phi:X \rightarrow \wp_+(Y)$ is a lower semicontinuous carrier such that $\phi(x)$ is closed and convex for all $x \in X$, then there exists a selection for $\phi$.
\end{thm}

To apply the Michael selection theorem to the representation theory of $C^*$-algebras, we will use the following two results. The first is a lemma used in proving the Kadison transitivity theorem.

\begin{lem}[{\cite[Lem.\ 5.4.2]{KadRinFTOAI}}]\label{lem:transitivity_in_B(H)}
Let $\cH$ be a Hilbert space and let $e_1,\ldots, e_n \in \cH$ be an orthonormal system. For any vectors $z_1,\ldots, z_n \in \cH$ such that $\norm{z_i} \leq r$ for all $i$, there exists $T \in \frB(\cH)$ such that $\norm{T} \leq (2n)^{1/2}r$ and $Te_i = z_i$ for all $i$. If there exists a self-adjoint operator $S \in \fB(\cH)$ such that $Se_i = z_i$ for all $i$, then $T$ may be chosen to be self-adjoint.
\end{lem}

The following theorem provides norm bounds on the elements of the $C^*$-algebra produced by the Kadison transitivity theorem. It will be instrumental in proving lower semicontinuity of the carriers that we consider. The essence of the statement and proof are contained in Theorem 2.7.5 in \cite{PedersenCAlgAutomorphisms}, however we consider a $C^*$-algebra with a representation rather than a $C^*$-subalgebra of $\fB(\cH)$. Therefore we provide a full proof.

\begin{thm}\label{thm:Pedersen}
Let $\frA$ be a $C^*$-algebra and let $(\cH, \pi)$ be a nonzero irreducible representation. If $x_1,\ldots, x_n \in \cH$ are linearly independent and $T \in \frB(\cH)$, then there exists $A \in \frA$ such that $\norm{A} \leq \norm{T}$ and $\pi(A)x_i = Tx_i$ for all $i$. If $T$ is self-adjoint, then $A$ may be chosen to be self-adjoint.
\end{thm}

\begin{proof}
First suppose $T$ is self-adjoint. Let $P \in \frB(\cH)$ be the projection onto $\vecspan\qty{x_1,\ldots, x_n, Tx_1,\ldots, Tx_n}$ and define $S = PTP$, which is self-adjoint, satisfies $\norm{S} \leq \norm{T}$, and has $Sx_i = Tx_i$ for all $i$. Extend $x_1,\ldots, x_n$ to a basis $x_1,\ldots, x_m$ of the above span. By the Kadison transitivity theorem, there exists $A \in \frA_\tn{sa}$ such that $\pi(A)x_i = PTx_i$ for all $i = 1,\ldots, m$, hence $\pi(A)P = S$. Assuming $\pi(A)^kP = S^k$ for some $k \in \bbN$, we have
\[
\pi(A)^{k+1}P = \pi(A)^k S = \pi(A)^k PS = S^{k+1},
\]
so $\pi(A)^{k}P = S^k$ for all $k \in \bbN$ by induction.

Define a continuous function $f:\bbR \rightarrow \bbR$ by
\[
f(t) = \begin{cases} \norm{S} &: t \geq \norm{S} \\ t &: \abs{t} \leq \norm{S} \\ -\norm{S} &: t \leq -\norm{S}\end{cases}
\]
Note that $f(A)$ is well-defined by continuous functional calculus and lies in $\frA$ regardless of whether $\frA$ is unital or not since $f(0) = 0$. Furthermore, $f(A)$ is self-adjoint, $\norm{f(A)} \leq \norm{S} \leq \norm{T}$, and $\pi(f(A)) = f(\pi(A))$. We also have $f(S) = S$ since $f$ restricts to the identity on $\sigma(S)$. 

We show that $\pi(f(A))x_i = Sx_i$ for all $i$. Given $i \leq n$ and $\varepsilon > 0$, we choose a real polynomial $g$ such that $\abs{f(t) - g(t)} < \varepsilon/2\norm{x_i}$ whenever $\abs{t} \leq \max(\norm{A}, \norm{S})$. Note that $g(\pi(A))x_i = g(S)x_i$ since $\pi(A)^kP = S^k$ for all $k \in \bbN$. Then
\begin{align*}
\norm{\pi(f(A))x_i - Sx_i} &\leq \norm{f(\pi(A))x_i - g(\pi(A))x_i} + \norm{g(S)x_i - Sx_i}\\
&\leq \norm{(f-g)(\pi(A))}\norm{x_i} + \norm{(g - f)(S)}\norm{x_i} < \varepsilon.
\end{align*}
Since $\varepsilon > 0$ was arbitrary, this implies $\pi(f(A))x_i = Sx_i = Tx_i$, as desired.

For the general case, we again consider $S = PTP$, where $P$ is defined as before, but $T$ is not necessarily self-adjoint. We still have $\norm{S} \leq \norm{T}$ and $Sx_i = Tx_i$ for all $i$. The map $\abs{S}(\cH) \rightarrow S(\cH)$, $\abs{S}x \mapsto Sx$ is a well-defined bijective isometry, and may therefore be extended to a unitary $\cH \rightarrow \cH$. By the self-adjoint case above, there exists $A \in \frA_\tn{sa}$ such that $\norm{A} \leq \norm{\abs{S}} = \norm{S} \leq \norm{T}$ and $\pi(A)x_i = \abs{S}x_i$ for all $i$. By the Kadison transitivity theorem, there exists a unitary $U$ in the unitization of $\frA$ such that $U\abs{S}x_i = Sx_i$ for all $i$. Then $UA \in \frA$, $\norm{UA} \leq \norm{T}$, and $\pi(UA)x_i = Sx_i = Tx_i$ for all $i$.
\end{proof}

We are now ready to prove the ``continuous Kadison transitivity theorem'' in the general and self-adjoint cases. For notation, when $\cH$ is a Hilbert space we denote elements of the Hilbert space $\cH^n$ by bold letters $\x = (x_1,\ldots, x_n)$ and elements of $\cH^{2n}$ by pairs of bold letters $(\x, \y)$. Given an element $T \in \fB(\cH)$ and $n \in \bbN$, we denote $T^{\oplus n} = T\oplus \cdots \oplus T \in \fB(\cH^n)$.

\begin{thm}\label{thm:continuous_Kadison}
Let $\fA$ be a $C^*$-algebra, let $(\cH, \pi)$ be a nonzero irreducible representation, and let $n$ be a positive integer. Let
\[
X = \qty{(\x, \y) \in \cH^{2n}: x_1,\ldots, x_n \tn{ are linearly independent}},
\]
equipped with the subspace topology inherited from $\cH^{2n}$. There exists a continuous map $A : X \rightarrow \fA$ such that
\begin{equation}\label{eq:selection_criterion}
\pi(A(\x,\y))x_i = y_i \quad \tn{for all} \, \, i = 1,\ldots, n
\end{equation}
for all $(\x, \y) \in X$. Similarly, defining
\[
X_\tn{sa} = \qty{(\x, \y) \in X: \exists\, T \in \fB(\cH)_\tn{sa} \tn{ s.t.\ } Tx_i = y_i \tn{ for all }  i =1,\ldots, n},
\]
there exists a continuous map $A :X_\tn{sa} \rightarrow \fA_\tn{sa}$ satisfying \eqref{eq:selection_criterion} for all $(\x, \y) \in X_\tn{sa}$.
\end{thm}

\begin{proof}
Since $X$ and $X_\tn{sa}$ are subspaces of $\cH^{2n}$, they are metrizable, hence paracompact Hausdorff. We will use the Michael selection theorem for the carrier $\phi:X \rightarrow \wp_+(\fA)$ defined by
\[
\phi(\x, \y) = \qty{A \in \fA: \pi(A)x_i = y_i \tn{ for all $i = 1,\ldots, n$}}.
\]
For the self-adjoint case, we note that $\fA_\tn{sa}$ is a real Banach space with the topology inherited from $\fA$, and we define $\phi_\tn{sa} :X_\tn{sa} \rightarrow \wp_+(\fA_\tn{sa})$ by  
\[
\phi_\tn{sa}(\x, \y) = \fA_\tn{sa} \cap \phi(\x, \y).
\] 
By the Kadison transitivity theorem, $\phi(\x, \y)$ and $\phi_\tn{sa}(\x, \y)$ are nonempty for all $(\x, \y) \in X$ and $(\x, \y) \in X_\tn{sa}$, respectively.  Given $(\x, \y) \in X$, $t \in [0,1]$, and $A, B \in \phi(\x, \y)$, we have
\[
\pi(tA + (1 - t)B)x_i = t \pi(A) x_i + (1 - t)\pi(B)x_i = ty_i + (1 - t)y_i = y_i
\]
for all $i$, so $\phi(\x, \y)$ is convex. An identical argument shows that $\phi_\tn{sa}(\x, \y)$ is convex for all $(\x, \y) \in X_\tn{sa}$. Furthermore, if $\hat x_i : \fB(\cH) \rightarrow \cH$ denotes the evaluation map $\hat x_i(T) = Tx_i$, then we see that
\[
\phi(\x, \y) = \bigcap_{i=1}^n (\hat x_i \circ \pi)^{-1}(\qty{y_i}),
\]
so $\phi(\x, \y)$ is closed since $\hat x_i \circ \pi$ is continuous for each $i$. For the self-adjoint case, we note that $\phi_\tn{sa}(\x, \y) = \fA_\tn{sa} \cap \phi(\x, \y)$ is closed in $\fA_\tn{sa}$ since $\fA_\tn{sa}$ has the subspace topology.

All that remains to show is that $\phi$ is lower semicontinuous, then the result will follow immediately from the Michael selection theorem. Fix $(\x_0, \y_0) \in X$, $A_0 \in \phi(\x_0, \y_0)$, and let $\varepsilon > 0$; replace $X$ and $\phi$ by $X_\tn{sa}$ and $\phi_\tn{sa}$ for the self-adjoint case. Given $(\x, \y) \in X$, let $e_1(\x),\ldots, e_n(\x) \in \cH$ be the orthonormal basis obtained by applying the Gram-Schmidt method to $x_1,\ldots, x_n$, and let $\lambda_{ij}(\x) \in \bbC$ be such that $e_i(\x) = \sum_{j=1}^n \lambda_{ij}(\x)x_j$. Note that each $\lambda_{ij}(\x)$ is a continuous function $X \rightarrow \bbC$. Moreover, the matrix $\Lambda_\x = (\lambda_{ij}(\x))$ defines an invertible element $\Lambda_\x \in \fB(\cH^n)$. Observe that the map $X \rightarrow \fB(\cH^n)$, $(\x, \y) \mapsto \Lambda_\x$ is continuous, $\Lambda_\x \x = \be(\x)$, and $[\Lambda_\x, T^{\oplus n}] = 0$ for all $T \in \fB(\cH)$. Let $O_1$ be the preimage of the open ball of radius $\norm{\Lambda_{\x_0}}$ centered on $\Lambda_{\x_0}$ under the map $(\x, \y) \mapsto \Lambda_\x$. Let $O_2$ be the preimage of the open ball of radius $\varepsilon/(4n\norm{\Lambda_{\x_0}})$ centered at zero under the map $X\rightarrow \cH^n$, $(\x, \y) \mapsto \y - \pi(A_0)^{\oplus n}\x$, which is also continuous. Then $O = O_1 \cap O_2$ is a neighborhood of $(\x_0, \y_0)$ in $X$ and $(\x, \y) \in O$ implies 
\[
\norm{\Lambda_\x} < 2 \norm{\Lambda_{\x_0}} \qqtext{and} \norm{\y - \pi(A_0)^{\oplus n} \x} < \frac{\varepsilon}{4n\norm{\Lambda_{\x_0}}}.
\]
For the self-adjoint case, we set $O_\tn{sa} = X_\tn{sa} \cap O$.

Given $(\x, \y) \in O$, set $\z(\x, \y) = \Lambda_\x \y$ and observe that $A \in \phi(\x, \y)$ if and only if $A \in \phi(\be(\x), \z(\x, \y))$ since $\Lambda_\x$ is invertible and commutes with $\pi(A)^{\oplus n}$ for all $A \in \fA$. For ease of notation we now suppress the arguments of $\be$ and $\z$. We estimate
\begin{align*}
\norm{z_i - \pi(A_0)e_i} &\leq \norm{\z - \pi(A_0)^{\oplus n}\be}\\
&<2\norm{\Lambda_{\x_0}}\norm{\y - \pi(A_0)^{\oplus n}\x} < \frac{\varepsilon}{2n}.
\end{align*}
By Lemma \ref{lem:transitivity_in_B(H)}, there exists $T \in \fB(\cH)$ such that $Te_i = z_i - \pi(A_0)e_i$ for all $i$ and $\norm{T} \leq \varepsilon/\sqrt{2n} < \varepsilon$.  In the self adjoint case, we observe that we may choose $T$ to be self-adjoint since $(\be, \z) \in X_\tn{sa}$ and $\pi(A_0)$ is self-adjoint, so there exists a self-adjoint operator mapping $e_i$ to $z_i - \pi(A_0)e_i$ for all $i$. By Theorem \ref{thm:Pedersen}, there exists $A_1 \in \fA$ such that $\norm{A_1} \leq \norm{T} < \varepsilon$ and $\pi(A_1)e_i = z_i - \pi(A_0)e_i$. In the self-adjoint case, we may choose $A_1$ to be self-adjoint. Defining $A = A_0 + A_1$, we see that $\norm{A - A_0} < \varepsilon$ and
\[
\pi(A)e_i = \pi(A_0)e_i + \pi(A_1)e_i = z_i
\]
for all $i$, which implies $A \in \phi(\be, \z) = \phi(\x, \y)$. In the self-adjoint case, we have $A \in \fA_\tn{sa}$ by choice of $A_1$, so $A \in \phi_\tn{sa}(\x, \y)$. This proves lower semicontinuity, completing the proof.
\end{proof}

\begin{rem}
Suppose $(\cH, \pi, \Omega)$ is the GNS representation of $\omega \in \sP(\fA)$. If we set $n = 1$ in the previous theorem and fix $x = \Omega$, then we get a continuous map $A: \cH \rightarrow \fA$ such that $\pi(A(y))\Omega = y = q(A(y))$ for all $y \in \cH$, where $q:\fA \rightarrow \cH$ is the quotient map. We see that $A$ is a right inverse for $q$. Since $q$ is a surjective bounded linear map between Banach spaces, the existence of a continuous right inverse is guaranteed by the Bartle-Graves theorem \cite{BartleGraves,MichaelSelection}. However, the existence of a continuous \textit{linear} right inverse is equivalent to $\ker q = \fN = \qty{A \in \fA: \omega(A^*A) = 0}$ having a closed complement in $\fA$. It is easy to see that this holds in certain special cases, such as when $\fA$ is commutative or finite-dimensional, or $\fA = \fB(\cK)$ for some Hilbert space $\cK$ and $\omega$ is a state corresponding to a unit vector in $\cK$. However, broader conditions under which $\fN$ is complemented (or whether this is true in general) are unknown to the authors.
\end{rem}

We now move towards a unitary version of Theorem \ref{thm:continuous_Kadison}.
We do so through a series of lemmata.

\begin{lem}\label{lem:Stiefel_section}
Let $\cH$ be a Hilbert space, let $x_1,\ldots, x_n \in \cH$ be an orthonormal system. Given $\varepsilon > 0$, there exists $\delta > 0$ such that for any orthonormal system $y_1,\ldots, y_n \in \cH$ with $\norm{x_i - y_i} < \delta$ for all $i$, there exists a unitary $U \in \Un(\cH)$ such that 
\begin{enumerate}
\item[\tn{(i)}] $\cK_n = \vecspan\qty{x_1,\ldots, x_n, y_1,\ldots, y_n}$ is invariant under $U$, 
\item[\tn{(ii)}] $U$ acts as the identity on $\cK^\perp_n$, 
\item[\tn{(iii)}] $\norm{I - U} < \varepsilon$, 
\item[\tn{(iv)}] and $Ux_i = y_i$ for all $i$.
\end{enumerate}
\end{lem}

\begin{proof}
We prove the lemma by induction on $n$. Consider the case when $n = 1$. Given $x, y \in \bbS \cH$, we define $U_{x,y} \in \Un(\cH)$ by having $U_{x,y}$ act as the identity on the orthogonal complement of $\cK = \vecspan\qty{x,y}$ and defining $U_{x,y}$ on $\cK$ by
\begin{equation}\label{eq:U_xy}
U_{x,y}z = \ev{y, x}z - \ev{y, z}x + \ev{x, z}y.
\end{equation}
One can check that this is indeed unitary and satisfies $U_{x,y}x = y$. When $\dim \cK = 1$, we have $y = \ev{x, y} x$, so we see that $U_{x,y}|_{\cK}$ is multiplication by $\ev{x,y}$. When $\dim \cK = 2$, the eigenvalues of $U_{x,y}|_{\cK}$ are
\[
\lambda_{x,y}^\pm = \Re \ev{x, y} \pm i \sqrt{1 - (\Re \ev{x, y})^2}.
\] 
Thus, $\sigma(U_{x,y}) \subset \qty{\lambda_{x,y}^+, \lambda_{x,y}^-, 1}$. Since $I - U_{x,y}$ is normal, its norm is given by its spectral radius, so
\begin{equation}\label{eq:U_xy_norm_identity}
\norm{I - U_{x,y}} = \abs{\lambda_{x,y}^+ - 1} = \abs{\lambda_{x,y}^- - 1} = \sqrt{2 - 2\Re \ev{x,y}} = \norm{x - y}.
\end{equation}
Therefore setting $\delta = \varepsilon$ and $U = U_{x,y}$ works for the base case.

Suppose the lemma is true for some $n$ and let $x_1,\ldots, x_{n+1}$ be an orthonormal system. Choose $\delta' > 0$ such that for any orthonormal system $y_1,\ldots, y_{n}$ with $\norm{x_i - y_i} < \delta'$ for all $i \leq n$, there exists a unitary $V \in \Un(\cH)$ satisfying (i), (ii), $\norm{I - V} < \varepsilon/3$, and $Vx_i = y_i$ for $i \leq n$. Let $\delta = \min(\delta', \varepsilon/3)$ and let $y_1,\ldots, y_{n+1}$ be an orthonormal system with $\norm{x_i - y_i} < \delta$ for all $i$. Since $V$ leaves $\cK_n$ invariant and acts as the identity on $\cK_n^\perp$, we see that
\[
Vx_{n+1} = VPx_{n+1} + (I - P)x_{n+1} \in \cK_{n+1},
\]
where $P$ is the projection onto $\cK_n$. Likewise $Vy_{n+1} \in \cK_{n+1}$, so $V$ leaves $\cK_{n+1}$ invariant. Since $\cK_{n+1}^\perp \subset \cK_n^\perp$, we also know that $V$ acts as the identity on $\cK_{n+1}^\perp$.

Set $z = Vx_{n+1}$ and note that $y_1,\ldots, y_n, z$ is an orthonormal system since it is the image of $x_1,\ldots, x_{n+1}$ under the unitary $V$. Furthermore,
\[
\norm{z - y_{n+1}} \leq \norm{(V-I)x_{n+1}} + \norm{x_{n+1} - y_{n+1}} < \frac{\varepsilon}{3} + \delta \leq \frac{2\varepsilon}{3}
\]
Consider the unitary $W = U_{z, y_{n+1}}V$. Since $z \in \cK_{n+1}$, we see that $W$ leaves $\cK_{n+1}$ invariant and acts as the identity on $\cK_{n+1}^\perp$. Since $z$ and $y_{n+1}$ are orthogonal to all $y_i$ with $i \leq n$, it follows that $Wx_i = y_i$ for all $i \leq n +1$. Finally,
\begin{align*}
\norm{I - W} &\leq \norm{I - U_{z, y_{n+1}}} + \norm{U_{z, y_{n+1}} - U_{z, y_{n+1}}V}\\
&\leq \norm{z - y_{n+1}} + \norm{I - V} < \varepsilon,
\end{align*}
completing the proof.
\end{proof}

\begin{rem}
We note that the $n = 1$ case of Lemma \ref{lem:Stiefel_section} can also be accomplished by polar decomposition. Assuming $x, y \in \bbS \cH$ and $\ev{x,y} \neq 0$, consider the operator
\[
A_{x,y} = \frac{\ev{x,y}}{\abs{\ev{x,y}}}P_yP_x + (1 - P_y)(1 - P_x),
\]
where $P_x$ and $P_y$ are the projections onto $\bbC x$ and $\bbC y$, respectively. The unitary $V_{x,y}$ obtained from the polar decomposition $A_{x,y} = V_{x,y}\abs{A_{x,y}}$ satisfies (i), (ii), and (iv). For fixed $x$, the map $y \mapsto V_{x,y}$ is norm-continuous on the domain $\qty{y \in \bbS \cH: \ev{x, y} \neq 0}$ and fulfills $V_{x,x} = I$. Therefore, $V_{x,y}$ satisfies (iii) for small enough $\delta$. 
\end{rem}

The following lemma contains the heart of the unitary version of the continuous Kadison transitivity theorem. 

\begin{lem}\label{lem:unitary_Kadison_lemma2}
Let $\fA$ be a unital $C^*$-algebra, let $(\cH, \pi)$ be a nonzero irreducible representation, and let $n$ be a non-negative integer. Define
\[
Y_+ = \qty{(\x, y) \in \cH^{n+1} \times \bbS \cH : \begin{array}{l} x_1,\ldots, x_{n+1} \tn{ are orthonormal, } \ev{x_{n+1}, y} > 0, \\ \tn{and } \ev{x_i, y} = 0 \tn{ for all }  i \leq n  \end{array}},
\]
equipped with the subspace topology. There exists a continuous map $U: Y_+ \rightarrow \Un(\fA)$ such that
\begin{equation}\label{eq:U_transitive_lemma}
\pi(U(\x, y))x_{n+1} = y \qqtext{and} \pi(U(\x, y))x_i = x_i
\end{equation}
for all $(\x, y) \in X$ and $i \leq n$. Here, $ \Un(\fA)$ denotes the group
of unitary elements of $\fA$.
\end{lem}

\begin{proof}
The function $\theta:Y_+ \rightarrow [0, \pi/2)$ defined as
\[
\theta(\x, y) = \cos^{-1} \ev{x_{n+1}, y}
\] 
is continuous on $Y_+$. If we set $Y'_+ = \qty{(\x, y) \in Y_+ : \ev{x_{n+1}, y} < 1}$, then we have another continuous map $w:Y'_+ \rightarrow \cH$ given by
\[
w(\x, y) = \frac{y - \ev{x_{n+1}, y}x_{n+1}}{\norm{y - \ev{x_{n+1}, y}x_{n+1}}}
\]
and $\qty{x_{n+1}, w(\x, y)}$ is a basis for $\vecspan\qty{x_{n+1}, y}$. In this basis, the unitary $U_{x_{n+1}, y}$ defined in Equation \eqref{eq:U_xy} is represented by the matrix
\[
\mqty(\cos \theta(\x, y) & - \sin \theta(\x, y) \\ \sin \theta(\x, y) & \cos \theta(\x, y))
\]
when restricted to $\vecspan\qty{x_{n+1}, y}$. Given $(\x, y) \in Y'_+$, we also define an operator $T_{\x, y} \in \fB(\cH)_\tn{sa}$ which acts as the zero operator on $\vecspan\qty{x_{n+1}, y}^\perp$ and, when restricted to $\vecspan\qty{x_{n+1}, y}$, is represented by the matrix
\[
\mqty(0 & i\theta(\x, y) \\ -i\theta(\x, y) & 0)
\]
with respect to the basis $\qty{x_{n+1}, w(\x, y)}$. Observe that $\norm{T_{\x,y}} = \theta(\x, y)$ and $U_{x_{n+1},y} = e^{iT_{\x,y}}$.

Note that $Y_+'$ is metrizable, hence paracompact Hausdorff. Given $(\x, y) \in Y_+$, define $\cK(\x, y) = \vecspan\qty{x_1,\ldots, x_n, y}$ and define a carrier $\phi:Y'_+ \rightarrow \wp_+(\fA_\tn{sa})$ by
\[
\phi(\x, y) = \qty{A \in \fA_\tn{sa} : \pi(A)|_{\cK(\x, y)} =T_{\x, y}|_{\cK(\x, y)} \tn{ and } \norm{A} \leq \theta(\x, y)}
\]
By Theorem \ref{thm:Pedersen}, $\phi(\x, y)$ is nonempty for all $(\x, y) \in Y'_+$. We see that $\phi(\x, y)$ is closed and convex by the same arguments used in the proof of Theorem \ref{thm:continuous_Kadison}, along with the fact that the closed ball of radius $\theta(\x, y)$ is closed and convex. For lower semicontinuity, fix $(\x, y) \in Y'_+$, $A_0 \in \phi(\x, y)$, and $\varepsilon > 0$. Use continuity and positivity of $\theta$ on $Y_+'$ to choose a neighborhood $O$ of $(\x, y)$ such that for all $(\u, v) \in O$, we have 
\[
\abs{1 - \frac{\theta(\u, v)}{\theta(\x, y)}} < \frac{\varepsilon}{2\norm{A_0}}.
\]
Apply Lemma \ref{lem:Stiefel_section} to the orthonormal system $x_1,\ldots, x_{n+1}, w(\x, y)$  and the number
\[
\varepsilon' = \min\qty(2, \frac{\varepsilon}{4\norm{A_0}})
\]
to find a $\delta > 0$ with the properties described in Lemma \ref{lem:Stiefel_section}. By continuity of $w$ on $Y_+'$, we may shrink $O$ such that for all $(\u, v) \in O$, we have $\norm{x_i - u_i} < \delta$ and $\norm{w(\x, y) - w(\u, v)} < \delta$.

Now, given $(\u, v) \in O$, there exists a unitary $V \in \Un(\cH)$ such that $Vx_i = u_i$ for all $i$, $Vw(\x, y) = w(\u, v)$, and $\norm{I - V} < \varepsilon'$. The fact that $\norm{I - V} < 2$ implies that $-1 \notin \sigma(V)$, so we can use the continuous functional calculus to apply the principal branch of the logarithm and obtain a self-adjoint operator $S = -i \Log V$ with $\norm{S} \leq \pi$. Note that $S$ leaves $\cG = \vecspan\qty{x_1,\ldots, x_{n+1}, w_{x_{n+1}, y}, u_1,\ldots, u_{n+1}, w_{u_{n+1}, v}}$ invariant since $V$ leaves $\cG$ invariant. By Theorem \ref{thm:Pedersen}, we can obtain a self-adjoint operator $B \in \fA$ such that $\pi(B)|_{\cG} = S|_{\cG}$ and $\norm{B} \leq \norm{S}$. Hence $W = e^{iB}$ acts as $V = e^{iS}$ on this subspace and by continuous functional calculus,
\begin{align*}
\norm{I - W} &= \sup_{\lambda \in \sigma(W)} \abs{\lambda - 1} = \abs{e^{i\norm{B}} - 1} \\
&\leq \abs{e^{i\norm{S}} - 1} = \sup_{\lambda \in \sigma(V)} \abs{\lambda - 1} = \norm{I - V}.
\end{align*}
It is easy to check that 
\[
A = \frac{\theta(\u, v)}{\theta(\x, y)}W A_0 W^{-1} \in \phi(\u, v)
\]
using the values of $V$ and $T_{\x, y}$  on $x_1,\ldots, x_{n+1}, w(\x, y)$, and the values of $V^{-1}$ and $T_{\u, v}$ on $u_1,\ldots, u_{n+1}, w(\u, v)$. Finally, observe that
\begin{align*}
\norm{A_0 - A} &\leq \norm{A_0 - WA_0W^{-1}} + \abs{1 - \frac{\theta(\u, v)}{\theta(\x, y)}} \norm{A_0}\\
&\leq \qty(2 \norm{I - W} + \abs{1 - \frac{\theta(\u, v)}{\theta(\x, y)}})\norm{A_0} < \varepsilon,
\end{align*}
as desired. This proves lower semicontinuity of $\phi$.

The Michael selection theorem now gives a continuous selection $A:Y'_+ \rightarrow \fA_\tn{sa}$ of $\phi$. We extend $A$ to $Y_+$ by defining $A(\x, y) = 0$ whenever $\ev{x_{n+1}, y} = 1$, equivalently, when $x_{n+1} = y$. Then $A:Y_+ \rightarrow \fA_\tn{sa}$ is continuous on $Y'_+$ since $Y'_+$ is open in $Y_+$ and $A$ is continuous on $Y_+ \setminus Y'_+$ by continuity of $\theta$ on $Y_+$ and the fact that $\norm{A(\x, y)} \leq \theta(\x, y)$ for all $(\x, y) \in Y_+$. Exponentiating $A$ yields a continuous map $U:Y_+ \rightarrow \Un(\fA)$, $U(\x, y) = e^{iA(\x, y)}$ that acts as $U_{x_{n+1}, y}$ on $\cK(\x, y)$, thereby satisfying \eqref{eq:U_transitive_lemma}.
\end{proof}

The purpose of our final lemma is to remove the condition that $\ev{x_{n+1}, y} > 0$ and replace it with the condition that $\ev{x_{n+1}, y} \notin \bbR_{\leq 0}$.

\begin{lem}\label{lem:unitary_Kadison_lemma}
Let $\fA$ be a unital $C^*$-algebra, let $(\cH, \pi)$ be a nonzero irreducible representation, and let $n$ be a non-negative integer. Define
\[
Y = \qty{(\x, y) \in \cH^{n+1} \times \bbS \cH : \begin{array}{l} x_1,\ldots, x_{n+1} \tn{ are orthonormal, } \ev{x_{n+1}, y} \notin  \bbR_{\leq 0}, \\ \tn{and } \ev{x_i, y} = 0 \tn{ for all }  i \leq n  \end{array}},
\]
equipped with the subspace topology. There exists a continuous map $U: Y \rightarrow \Un(\fA)$ such that
\begin{equation}\label{eq:unitary_Kadison_eq3}
\pi(U(\x, y))x_{n+1} = y \qqtext{and} \pi(U(\x, y))x_i = x_i
\end{equation}
for all $(\x, y) \in Y$ and $i \leq n$. 
\end{lem}

\begin{proof}
The angle $\alpha: Y \rightarrow (-\pi, \pi)$ defined by taking the principal branch of the logarithm:
\[
\alpha(\x, y) = \Im \Log \ev{x_{n+1}, y}
\]
is continuous on $Y$. The map $Y \rightarrow X_\tn{sa}$, $(\x, y) \mapsto (\x, 0, \ldots, 0, \alpha(\x, y)x_{n+1})$ is continuous, where $X_\tn{sa}$ is as in Theorem \ref{thm:continuous_Kadison}. We may therefore compose with the map $A:X_\tn{sa} \rightarrow \fA_\tn{sa}$ from Theorem \ref{thm:continuous_Kadison} and exponentiate to obtain a continuous map $V:Y \rightarrow \Un(\fA)$ such that
\begin{align*}
\pi(V(\x, y))x_{n+1} = e^{i\alpha(\x, y)}x_{n+1} = \frac{\ev{x_{n+1}, y}}{\abs{\ev{x_{n+1},y}}} x_{n+1}
\end{align*}
and $\pi(V(\x, y))x_i = x_i$ for all $i \leq n$. 

Note that
\[
\ev{\frac{\ev{x_{n+1}, y}}{\abs{\ev{x_{n+1}, y}}} x_{n+1}, y} = \abs{\ev{x_{n+1}, y}} > 0.
\]
Thus, we have a continuous map $Y \rightarrow Y_+$, $(\x, y) \mapsto (\pi(V(\x, y))^{\oplus n+1} \x, y)$ and we compose this with the continuous map from Lemma \ref{lem:unitary_Kadison_lemma2} to obtain a continuous map $W:Y \rightarrow \Un(\fA)$. Defining $U:Y \rightarrow \Un(\fA)$ by $U(\x, y) = W(\x, y)V(\x, y)$, we see that $U$ is continuous and satisfies \eqref{eq:unitary_Kadison_eq3}.
\end{proof}

We are finally ready to prove the unitary version of the continuous Kadison transitivity theorem.

\begin{thm}\label{thm:continuous_Kadison_unitary}
Let $\fA$ be a unital $C^*$-algebra, let $(\cH, \pi)$ be a nonzero irreducible representation, and let $n$ be a positive integer. Let 
\[
X_\tn{u} = \qty{(\x, \y) \in X: \exists\, T \in \Un(\cH) \tn{ s.t.\ } Tx_i = y_i \tn{ for all } i =1,\ldots, n},
\]
equipped with the subspace topology, where $X$ is as in Theorem \ref{thm:continuous_Kadison}. For every $(\x_0, \y_0) \in X_\tn{u}$, there exists a neighborhood $(\x_0, \y_0) \in O \subset X_\tn{u}$ and a continuous map $U:O \rightarrow \Un(\fA)$ such that
\begin{equation}\label{eq:unitary_selection}
\pi(U(\x, \y))x_i = y_i \quad \tn{ for all }  i = 1,\ldots, n
\end{equation}
for all $(\x, \y) \in O$.
\end{thm}

\begin{proof}
As in the proof of Theorem \ref{thm:continuous_Kadison}, let $\Lambda : X_\tn{u} \rightarrow \fB(\cH^n)^\times$ be the continuous map obtained by applying the Gram-Schmidt method to $x_1,\ldots, x_n$. Recall that for $T \in \fB(\cH)$, we have $T^{\oplus n}\x = \y$ if and only if $T^{\oplus n}\be(\x) = \z(\x, \y)$, where $\be(\x) = \Lambda_\x \x$ and $\z(\x, \y) = \Lambda_\x \y$. Since $(\x, \y) \mapsto (\be(\x), \z(\x, \y))$ is continuous, it suffices to prove the theorem with $X_\tn{u}$ replaced by 
\[
X_\tn{u}^\tn{on} = \qty{(\x, \y) \in X_\tn{u} : x_1,\ldots, x_n \tn{ are orthonormal}}.
\]
Therefore, suppose $(\x_0, \y_0) \in X_\tn{u}^\tn{on}$.

Suppose the theorem is true when $\x_0 = \y_0$. Then for arbitrary $(\x_0, \y_0) \in X_\tn{u}^\tn{on}$, we can find a neighborhood $(\x_0, \x_0) \in O \subset X_\tn{u}^\tn{on}$ and a continuous function $U:O \rightarrow \Un(\fA)$ for which \eqref{eq:unitary_selection} holds. By the Kadison transitivity theorem, there exists $V \in \Un(\fA)$ such that $\pi(V)^{\oplus n} \y_0 = \x_0$. Then $O' = (I_{\cH^n} \oplus \pi(V)^{\oplus n})^{-1}(O) \subset X_\tn{u}^\tn{on}$ is a neighborhood of $(\x_0, \y_0)$ and $O' \ni (\x, \y) \mapsto V^{-1}U(\x, \pi(V)^{\oplus n}\y)$ is a continuous map satisfying \eqref{eq:unitary_selection}. Thus it suffices to prove the theorem for $(\x_0, \x_0) \in X_\tn{u}^\tn{on}$.

Let $Z_n = \qty{\x \in \cH^n : x_1,\ldots, x_n \tn{ are orthonormal}}$ with the subspace topology. Now suppose that for every $\x \in Z_n$, we have a neighborhood $\x \in O \subset Z_n$ and a continuous map $U: O \rightarrow \Un(\fA)$ such that 
\begin{equation}\label{eq:selection_unitary_simpler}
\pi(U(\y))^{\oplus n}\x = \y
\end{equation} 
for all $\y \in O$. Given $(\x_0, \x_0) \in X_\tn{u}^\tn{on}$ and such a neighborhood $\x_0 \in O \subset Z_n$, we see that $O' = O \times O \subset X_\tn{u}^\tn{on}$ is a neighborhood of $(\x_0, \x_0)$ and the map $O' \ni (\x, \y) \mapsto U(\y)U(\x)^{-1}$ satisfies \eqref{eq:unitary_selection}. 

We prove the hypothesis about $Z_n$ by induction on $n$. The $n = 1$ case follows from the $n = 0$ case of Lemma \ref{lem:unitary_Kadison_lemma}: one takes $O = \qty{y \in \bbS \cH: \ev{x, y} \notin \bbR_{\leq 0}}$ and the map $O \ni y \mapsto (x, y) \mapsto U(x, y)$ does the trick, where $U :Y \rightarrow \Un(\fA)$ is as in Lemma \ref{lem:unitary_Kadison_lemma}. Assume the hypothesis is true for some $n \geq 1$ and let $\x \in Z_{n+1}$. Let $P :Z_{n+1} \rightarrow Z_n$ be the projection onto the first $n$ components. We have a neighborhood $P\x \in O \subset Z_n$ and a continuous map $\tilde V: O \rightarrow \Un(\fA)$ such that \eqref{eq:selection_unitary_simpler} holds for $P\x$ and $\y \in O$. Define $V:O \rightarrow \Un(\fA)$ by $V(\y) = \tilde V(\y)\tilde V(P\x)^{-1}$ so that $V$ is continuous, satisfies \eqref{eq:selection_unitary_simpler} for $P\x$ and $\y \in O$, and has $V(P\x) = 1$. Define
\begin{align*}
O' = \qty{\y \in P^{-1}(O): \ev{y_{n+1}, \pi(V(P\y))x_{n+1}} \notin \bbR_{\leq 0}}
\end{align*}
Then $\x \in O' \subset Z_{n+1}$ and  $O' \rightarrow Y$, $\y \mapsto (\y, \pi(V(P\y))x_{n+1})$ is continuous. Composing this with the map from Lemma \ref{lem:unitary_Kadison_lemma} yields a continuous map $W:O' \rightarrow \Un(\fA)$, and one can easily check that $U:O' \rightarrow \Un(\fA)$, $U(\y) = W(\y)^{-1}V(P\y)$ satisfies \eqref{eq:selection_unitary_simpler}. This completes the proof.
\end{proof}

\begin{cor}\label{cor:continuous_Kadison_automorphism}
Let $\fA$ be a unital $C^*$-algebra and let $\omega \in \sP(\fA)$. There exists a continuous map $U:\bbB_2(\omega) \rightarrow \Unitary(\fA)$ such that 
\[
U_\psi \cdot \omega = \psi
\]
for all $\psi \in \bbB_2(\omega)$. If $\fA$ is non-unital and $\omega \in \sP(\fA)$, then there exists a continuous map $\alpha:\bbB_2(\omega) \rightarrow \Aut(\fA)_\tn{n}$ such that
\[
\omega \circ \alpha_\psi = \psi
\]
for all $\psi \in \bbB_2(\omega)$.
\end{cor}

\begin{proof}
Let $(\cH, \pi, \Omega)$ be the GNS representation of $\omega$. 
Corollary \ref{cor:purestate_mapsto_SH} states continuity of the map $\bbB_2(\omega) \rightarrow \bbS \cH$, $\psi \mapsto \Psi$ where $\Psi$ represents $\psi$ and $\ev{\Psi, \Omega} > 0$. Hence $\bbB_2(\omega) \rightarrow Y$,  $\psi \mapsto (\Omega, \Psi)$ is well-defined, where $Y$ is as in Lemma \ref{lem:unitary_Kadison_lemma} with $n = 0$. Composing with the map from Lemma \ref{lem:unitary_Kadison_lemma} yields  a continuous map $U:\bbB_2(\omega) \rightarrow \Unitary(\fA)$, $\psi \mapsto U_\psi$ such that
\[
\pi(U_\psi)\Omega = \Psi
\]
for all $\psi \in \bbB_2(\omega)$. This implies that $U_\psi \cdot \omega = \psi$, as desired.

In the non-unital case, we consider the unitization $\tilde \fA$ and the isometry $\sP(\fA) \rightarrow \sP(\tilde \fA)$, $\psi \mapsto \tilde \psi$ where $\tilde \psi$ is the unique extension of $\psi$ to a pure state on $\tilde \fA$. This gives a continuous map $\bbB_2(\omega) \rightarrow \bbB_2(\tilde \omega)$, and we may apply the unital case to get a continuous map $U:\bbB_2(\tilde \omega) \rightarrow \Unitary(\tilde \fA)$ such that $U_{\tilde\psi} \cdot \tilde \omega = \tilde\psi$ for all $\psi \in \bbB_2(\omega)$. Since $\fA$ is a two-sided ideal in $\tilde \fA$, we have a continuous function $\Unitary(\tilde \fA) \rightarrow \Aut(\fA)_\tn{n}$, $U \mapsto (A \mapsto U^*AU)$ and composing with this function yields the desired map $\bbB_2(\omega) \rightarrow \Aut(\fA)_\tn{n}$.
\end{proof}

\subsection{Construction of principal fiber bundles}\label{sec:principal_fiber_bundles}

In this section, given a unital $C^*$-algebra $\fA$ and a pure state $\omega \in \sP(\fA)$, we construct a principal fiber bundle $p_{\Unitary(\fA)} :\Unitary(\fA) \rightarrow \sP_\omega(\fA)$, where 
\[
[p_{\Unitary(\fA)}(U)](A) = (U \cdot \omega)(A) = \omega(U^*AU).
\] Crucially, Corollary \ref{cor:continuous_Kadison_automorphism} is used to construct local trivializations. Note that $\Unitary(\fA) \times \sP(\fA)_\tn{n} \rightarrow \sP(\fA)_\tn{n}$, $(U, \omega) \mapsto U \cdot \omega$ is a continuous group action. We then discuss a strategy for showing this bundle is nontrivial and provide a few examples.

\begin{cor}
  Let $\fA$ be a unital $C^*$-algebra, assume that $\omega \in \sP(\fA)$ is a
  pure state, and let
  $p_{\Unitary(\fA)}:\Unitary(\fA) \rightarrow \sP_\omega(\fA)$ be the
  continuous surjection which maps $U\in \Unitary(\fA)$ to $p(U) = U \cdot \omega$.
  Furthermore, let the isotropy group
  \[ \Unitary_\omega(\fA) = \qty{U \in \Unitary(\fA): U \cdot \omega = \omega}\]
  act on $\Unitary(\fA)$ and on itself by right multiplication. The fibers of $p_{\Unitary(\fA)}$ then are exactly the orbits of $\Unitary_\omega(\fA)$ and  $p_{\Unitary(\fA)}:\Unitary(\fA) \rightarrow \sP_\omega(\fA)$ is a locally trivial principal $\Unitary_\omega(\fA)$-bundle.
\end{cor}

\begin{proof}
Suppose $U, V \in \Unitary(\fA)$ and $U \cdot \omega = V \cdot \omega$. Then $V^*U \cdot \omega = \omega$, so $V^*U \in \Unitary_\omega(\fA)$ and $U = V(V^*U)$ is in the same orbit as $V$. Conversely, if $U$ and $V$ are in the same orbit, then there exists $W \in \Unitary_\omega(\fA)$ such that $U = VW$, hence $U \cdot \omega = VW \cdot \omega = V \cdot \omega$, so $U$ and $V$ are in the same fiber.

We now construct local trivializations for $p_{\Unitary(\fA)}$ which for the rest of the proof we abbreviate by $p$. Fix $\psi_0 \in \sP_\omega(\fA)$. By Corollary \ref{cor:continuous_Kadison_automorphism}, we know there exists a continuous map $U:\bbB_2(\psi_0) \rightarrow \Unitary(\fA)$ such that $U_\psi \cdot \psi_0 = \psi$ for all $\psi \in \bbB_2(\psi_0)$. Fix a unitary $V \in \Unitary(\fA)$ such that $V\omega = \psi_0$. Define $\varphi:p^{-1}(\bbB_2(\psi_0)) \rightarrow \bbB_2(\psi_0) \times \Unitary_\omega(\fA)$ by
\begin{align*}
\varphi(W) = (p(W), V^*U_{p(W)}^*W).
\end{align*}
Observe that 
\[
V^*U_{p(W)}^*W \omega = V^*U_{p(W)}^*p(W) = V^*\psi_0 = \omega,
\]
so $\varphi$ is indeed well-defined. Furthermore, $\varphi$ is manifestly continuous, equivariant on fibers, and commutes with the projections onto $\bbB_2(\psi_0)$. The function $\bbB_2(\psi_0) \times \Unitary_\omega(\fA) \rightarrow p^{-1}(\bbB_2(\psi_0))$ defined by 
\[
(\psi, W) \mapsto U_\psi V W
\]
is well-defined since 
\[
p(U_\psi V W) = U_\psi V W \omega = U_\psi V \omega = U_\psi \psi_0 = \psi \in \bbB_2(\psi_0).
\] Furthermore, it is continuous and it is a two-sided inverse for $\varphi$ since
\begin{align*}
\varphi(U_\psi V W) = (\psi, V^*U^*_\psi U_\psi VW) = (\psi, W)
\end{align*}
and 
\[
\varphi(W) = (p(W), V^*U^*_{p(W)}W) \mapsto U_{p(W)}V V^*U^*_{p(W)}W = W.
\]
This proves that $\varphi$ is a homeomorphism. This provides a trivializing cover for $p$ with local trivializations that are equivariant on fibers, completing the proof.
\end{proof}

By looking at the fundamental groups of $\sP_\omega(\fA)$, $\Unitary(\fA)$, and $\Unitary_\omega(\fA)$, one may be able to determine that the bundle $p_{\Unitary(\fA)} :\Unitary(\fA) \rightarrow \sP_\omega(\fA)$ is nontrivial for a particular unital $C^*$-algebra $\fA$ and state $\omega \in \sP(\fA)$. We exhibit some examples below. If the bundle is trivial, then $\Unitary(\fA)$ is homeomorphic to $\Unitary_\omega(\fA) \times \sP_\omega(\fA)$, and in each example we show that this leads to a contradiction.

If $\fA = M_n(\bbC)$ for some integer $n \geq 2$ and $\omega$ is any pure state, then $\sP_\omega(\fA) \cong \bbC \bbP^{n-1}$, so $\sP_\omega(\fA)$ is simply connected. The unitary group $\Unitary(\fA) = \Unitary(n)$ is path-connected and has fundamental group $\pi_1(\Unitary(n)) \cong \bbZ$. The stabilizer is $\Unitary_\omega(\fA) \cong \Unitary(1) \times \Unitary(n-1)$ which has fundamental group $\pi_1(\Unitary_\omega(\fA)) \cong \bbZ \times \bbZ$. Therefore the bundle is not trivial because
\[
\pi_1(\Unitary(\fA)) \cong \bbZ \not\cong \bbZ \times \bbZ \cong \pi_1(\Unitary_\omega(\fA) \times \sP_\omega(\fA)).
\]

If $\fA = \fB(\cH)$ for a separable, infinite-dimensional Hilbert space $\cH$ and $\omega$ is a pure normal state, then $\sP_\omega(\fA) \cong \bbP \cH$ is an Eilenberg-MacLane space of type $K(\bbZ, 2)$,  $\Unitary(\fA) = \Unitary(\cH)$ is contractible by Kuiper's theorem \cite{Kuiper}, and $\Unitary_\omega(\fA) \cong \Unitary(1) \times \Unitary(\cH)$. Therefore the bundle is not trivial because
\[
\pi_1(\Unitary(\fA)) \cong \qty{0} \not\cong \bbZ \cong \pi_1(\Unitary_\omega(\fA) \times \sP_\omega(\fA)).
\]

We can also show that the bundle is nontrivial for any UHF algebra. If $\fA$ is a UHF algebra and
$\omega \in \sP(\fA)$, then $\sP_\omega(\fA) \cong \bbP \cH_\omega$ and $\cH_\omega$ is a separable,
infinite dimensional Hilbert space \cite{Glimm_UHF_Algebras}, so $\bbP \cH_\omega$ is again a $K(\bbZ, 2)$.
In the following we will determine the homotopy groups of $\Unitary(\fA)$ and $\Unitary_\omega(\fA)$
which will then entail nontriviality of the bundle
$p_{\Unitary(\fA)} :\Unitary(\fA) \rightarrow \sP_\omega(\fA)$.
The computation relies on two major results, namely on Glimm's observation that the isomorphism
class of a UHF algebra can be encoded by its associated supernatural number \cite{Glimm_UHF_Algebras}
and on a theorem of Gl\"ockner \cite[Thm.~1.13]{glockner2010homotopy}.
Before we come to the computation of the homotopy groups $\pi_k(\Unitary(\fA))$ and
$\pi_k(\Unitary_\omega(\fA))$ we therefore first state these results in the form needed here and
provide a few preliminaries.

Recall from e.g.~\cite[Sec.~7.4]{RordamLarsenLaustsenKtheory}
that by a \emph{supernatural number} one understands a sequence
$n= (n_i)_{i\in \N}$ of elements
$n_i \in \{ 0, 1,\ldots ,\infty\} = \N \cup\{\infty\}$.
By slight abuse of language one sometimes writes 
\[
  n = \prod_{i\in \N} p_i^{n_i} \ ,
\]
where $\{ p_0,p_1,\ldots\}$ is the set of primes listed in increasing order,
and regards the right hand side of this formula as a formal prime
factorization of the supernatural number $n$. The product of two supernatural numbers $n,m$
is given by
\[
  nm = \prod_{i\in \N} p_i^{n_i+m_i}\ ,  
\]
but their sum is in general not defined. 
Associated to a supernatural number $n=(n_i)_{i\in \N}$ is the additive subgroup
$Q(n) \subset \Q$ consisting of all fractions $\frac pq$, where $p,q$ are integers
and $q$ has the prime decomposition
\[
  q = \prod_{i\in \N} p_i^{q_i} 
\]
such that $q_i \leq n_i$ for all $i\in \N$ and only finitely many of the $q_i$ are nonzero. 
By construction, $Q(n)$ contains $1$, and each additive supgroup $A\subset \Q$ containing $1$ equals
$Q(n)$ for some supernatural number $n$. Furthermore, two groups $Q(n)$ and  $Q(m)$ are isomorphic
if and only there are positive natural numbers $n',m'$ such that $n n' = mm'$;
see \cite[Sec.~7.4]{RordamLarsenLaustsenKtheory} for details. 

To further clarify language let us remind the reader that by a UHF algebra one understands
a $C^*$-algebra $\fA$ which can be identified with the colimit of a countable strict inductive
system of type I factors of finite dimension and unital $*$-homomorphisms 
\[
  \fA_0 \lhook\joinrel\xrightarrow{\iota_{0,1}\:\:} \fA_1
  \lhook\joinrel\xrightarrow{\iota_{1,2}\:\:}\ldots  \lhook\joinrel\xrightarrow{\iota_{i-1,i}} \fA_i
  \lhook\joinrel\xrightarrow{\iota_{i,i+1}}\ldots \ . 
\]
Recall that strictness of the inductive system means that each of the unital $*$-homomorphisms
$\iota_i$ is injective and that by a type I factor of finite dimension
one understands a von Neumann algebra which is $*$-isomorphic to the matrix algebra $M_n (\bbC)$
for some $n\in \N$.
Therefore,  each of the $C^*$-algebras $\fA_i$ is $*$-isomorphic to
a matrix algebra $M_{n_i} (\bbC)$ such that the sequence of ranks $(n_i)_{i\in \N}$ is increasing
and $n_i$ is a divisor of $n_j$ for all $i\leq j$. Following Glimm \cite{Glimm_UHF_Algebras}, we say
that  $\fA$ is \emph{generated} by the inductive system $(\fA_i)_{i\in \N}$
of \emph{type} $(n_i)_{i\in \N}$. We always assume that the type is \emph{unbounded} meaning
that $\lim_{j\to \infty} n_j =\infty$. As in \cite{Glimm_UHF_Algebras}, a UHF algebra $\fA$ therefore has 
to be infinite dimensional. Glimm further associates to a UHF algebra $\fA$  a supernatural number
$\delta_\fA$ as follows. Write 
$n_i = \prod_{j\in \N} p_j^{\delta_{i,j}}$, where the $\delta_{i,j} \in \N$ are unique by prime decomposition,   
and put $\delta_j = \sup\{ \delta_{i,j} : \: i\in \N\}$. By construction, $\delta_\fA =(\delta_j)_{j\in \N}$
then is a supernatural number which according to \cite{Glimm_UHF_Algebras}
uniquely determines the isomorphism class of the UHF algebra $\fA$; see also
\cite[Thm.~7.4.5]{RordamLarsenLaustsenKtheory}. 

\begin{lem}\label{lem:UHF_unitary_closures}
  Assume to be given a strict inductive system $(\fA_j)_{j \in \bbN}$ of 
  unbounded type  $(n_j)_{j\in \N}$ such that $\fA_i \subset \fA_j$ for all $i\leq j$
  and such that the morphisms of the inductive system are given by inclusion. 
  Let $\fA$ be the UHF algebra generated by $(\fA_j)_{j \in \bbN}$ that is let
  $\fA = \overline{\bigcup_{j\in \bbN} \fA_j}$. 
  If $\omega \in \sP(\fA)$ and $\omega|_{\fA_j} \in \sP(\fA_j)$
  for all $j$, then
\begin{equation}\label{eq:UHF_unitary_closures}
\Unitary_\omega(\fA) = \overline{\bigcup_{j\in \bbN} \Unitary(\fA_j) \cap \Unitary_\omega(\fA)}.
\end{equation}
and
\begin{equation}\Unitary(\fA_j) \cap \Unitary_\omega(\fA) = \Unitary_{\omega|_{\fA_j}}(\fA_j) \cong \Unitary(1) \times \Unitary(n_j - 1) \ .
  \end{equation}
\end{lem}

\begin{proof}
  Suppose $U \in \Unitary_\omega(\fA)$, fix $\varepsilon > 0$, and let $(\cH, \pi, \Omega)$ be the GNS representation of $\omega$. Lemma 3.1 in \cite{Glimm_UHF_Algebras} states that $\Unitary(\fA) = \overline{\bigcup_{j\in\bbN} \Unitary(\fA_j)}$, so there exists $j \in \bbN$ and $V \in \Unitary(\fA_j)$ such that $\norm{U - V} < \varepsilon/2$.  The fact that $U \in \Unitary_\omega(\fA)$ implies that $\pi(U)\Omega = \lambda \Omega$ for some $\lambda \in \Unitary(1)$, hence
\[
\norm{\pi(V)\Omega - \lambda \Omega} = \norm{\pi(V - U)\Omega} < \frac{\varepsilon}{2}.
\]
Now, $\cH_j \defeq \pi(\fA_j)\Omega$ is a finite-dimensional subspace of $\cH$ and $\pi$ restricts to a cyclic representation $\pi_j:\fA_j \rightarrow \fB(\cH_j)$ with cyclic unit vector $\Omega$ representing $\omega|_{\fA_j}$. Since $\omega|_{\fA_j}$ is pure by hypothesis, $\pi_j$ is an irreducible representation. In particular, since $\fA_j \cong M_{n_j}(\bbC)$, we know $\pi_j$ is a $*$-isomorphism. There exists a unitary $W \in \Unitary(\fA_j)$ such that $\pi_j(W) = U_{\pi(V)\Omega, \lambda\Omega}$, where $U_{\pi(V)\Omega, \lambda \Omega}$ is as defined in Lemma \ref{lem:Stiefel_section}. By \eqref{eq:U_xy_norm_identity} of Lemma \ref{lem:Stiefel_section}, $\norm{I - W} = \norm{\pi(V)\Omega - \lambda \Omega} < \varepsilon/2$ and $\pi(W)\pi(V)\Omega = \lambda \Omega$. Then $WV \in \Unitary(\fA_j) \cap \Unitary_\omega(\fA)$ and 
\[
\norm{U - WV} \leq \norm{U - V} + \norm{V - WV} < \varepsilon,
\]
as desired. 

Note that $U_0 \in \Unitary_{\omega|_{\fA_j}}(\fA_j)$ if and only if $U_0 \in \Unitary(\fA_j)$ and $\omega(U_0^*AU_0) = \omega(A)$ for all $A \in \fA_j$. Therefore it is clear that $\Unitary(\fA_j) \cap \Unitary_\omega(\fA)  \subset \Unitary_{\omega|_{\fA_j}}(\fA_j)$. Conversely, if $U_0 \in \Unitary_{\omega|_{\fA_j}}(\fA_j)$, then $\pi_j(U_0)\Omega = \pi(U_0)\Omega = \lambda_0 \Omega$ for some $\lambda_0 \in \Unitary(1)$, which implies that $U_0 \in \Unitary_\omega(\fA)$, i.e., $\omega(U_0^*AU_0) = \omega(A)$ for all $A \in \fA$. That $\Unitary_{\omega|_{\fA_j}}(\fA_j) \cong \Unitary(1) \times \Unitary(n_j - 1)$ is immediate from the fact that $\fA_j$ is $*$-isomorphic to $M_{n_j}(\bbC)$.
\end{proof}

\begin{lem}\label{lem:Cayley_transform}
Let $\fA$ be a unital $C^*$-algebra generated by the strict inductive system $(\fA_j)_{j \in \bbN}$
of type $(n_j)_{j\in \N}$.
Let $O = \qty{U \in \Unitary(\fA): \norm{I - U} < 2}$ and let $\phi:O \rightarrow \fA_\tn{sa}$ be defined by $\phi(U) = i(I - U)(I + U)^{-1}$. Then $\phi$ is a homeomorphism and $\phi(O \cap \Unitary_\omega(\fA))$ is a closed subspace of $\fA_\tn{sa}$ for all $\omega \in \sP(\fA)$.
\end{lem}

\begin{proof}
If $U \in O$, then $-1 \notin \sigma(U)$ since this would contradict $\norm{I - U} < 2$. Therefore $\phi$ is well-defined. Note that $\phi$ is just multiplication by $-1$ composed with the inverse Cayley transform. In particular, $\phi$ is a homeomorphism  with inverse  $\phi^{-1}(A) = (iI - A)(iI + A)^{-1}$ by continuous functional calculus. 

Let $\omega \in \sP(\fA)$. Then $\phi(O \cap \Unitary_\omega(\fA))$ is closed since $\Unitary_\omega(\fA)$ is closed in $\Unitary(\fA)$ and $\phi$ is a homeomorphism. Let $(\cH, \pi, \Omega)$ be the GNS representation of $\omega$. If $U \in O \cap \Unitary_\omega(\fA)$, then $\pi(U)\Omega = \lambda \Omega$ for some $\lambda \in \Unitary(1)$. Furthermore, there exists a sequence of polynomials $(p_n)$ such that $p_n(\lambda)$ converges to $\phi(\lambda)=i(1 - \lambda)(1 + \lambda)^{-1}$ uniformly on $\sigma(U)$, hence
\[
\pi(\phi(U))\Omega = \lim_{n \rightarrow \infty} \pi(p_n(U)) \Omega = \lim_{n \rightarrow \infty} p_n(\lambda)\Omega = \phi(\lambda)\Omega. 
\]
Now, let $U, V \in O \cap \Unitary_\omega(\fA)$, let $\alpha \in \bbR$, and set $A = \phi(U) + \alpha \phi(V)$. Then the argument above implies that $\pi(A)\Omega = \mu \Omega$ for some $\mu \in \bbR$, hence $\pi(\phi^{-1}(A))\Omega = \phi^{-1}(\mu)\Omega$ by the same argument. Thus, $\phi^{-1}(A) \in O \cap \Unitary_\omega(\fA)$, so $A \in \phi(O \cap \Unitary_\omega(\fA))$. Therefore, $\phi(O \cap \Unitary_\omega(\fA))$ is a subspace of $\fA_\tn{sa}$.
\end{proof}

The last tool we need for the computation of the homotopy groups is a theorem by Gl\"ockner \cite[Thm.~1.13]{glockner2010homotopy}, which says that under the existence of so-called well-filled charts, the homotopy groups of a space $X$ are the direct colimits of the homotopy groups of an ascending sequence of subspaces $X_1 \subset X_2 \subset \cdots$ whose union $\bigcup X_j$ is dense in $X$.  The notion of a well-filled chart is given by Definition 1.7 in the same article. The definition provided there is more general than we need; in fact, our well-filled charts are of a very simple form and the following more restrictive framework will suffice. Let $X$ be a Hausdorff topological group with a sequence of subgroups $(X_j)_{j \in \bbN}$ such that $X_j \subset X_{j+1}$ for all $j \in \bbN$ and $X = \overline{\bigcup_{j} X_j}$. Equip each $X_j$ with its subspace topology. Let $E$ be a
Hausdorff locally convex topological vector space. If $O$ is an open subset of $X$ containing the identity of $X$ and $\phi:O \rightarrow E$ is a homeomorphism such that $\phi(O \cap X_j)$ is a closed
linear subspace of $E$ for all $j \in \bbN$, then $\phi$ is a \emph{well-filled chart} and
\[
\pi_k(X, x) = \colim_{j \in \bbN_x} \pi_k(X_j, x) \tn{ for all $k \in \bbN$ and $x \in \bigcup_{j\in \bbN} X_j$}\ ,
\]
where $\bbN_x = \qty{j \in \bbN: x \in X_j}$ and where the colimit is with respect to the homomorphisms induced by the inclusions $X_j \rightarrow X$ and $X_i \rightarrow X_{j}$ for $i \leq j$. Likewise,
\[
  \pi_0(X) = \colim_{j\in \N} \pi_0(X_j) \ .
\]
This distills what we need from Definition 1.7, Theorem 1.13, Corollary 1.14 and Lemma 8.1 in \cite{glockner2010homotopy}, although the full definition of a well-filled chart is more general.

\begin{thm}\label{thm:UHF_Uomega}
Let $\fA$ be a  UHF algebra  generated by the strict inductive system $(\fA_j)_{j \in \bbN}$
of type $(n_j)_{j\in \N}$.
Denote by $\delta_\fA$ the supernatural number associated to $\fA$. Then  
\begin{equation}\label{eq:relation-classifying-group-type}
  Q(\delta_\fA) = \bigcup_{j\in \N} n_j^{-1}\bbZ
\end{equation}
and 
\begin{equation}\label{eq:homotopy_groups_U(A)_UHF}
\pi_k(\Unitary(\fA)) \cong \begin{cases} 0 &\text{for $k$ even},\\ Q(\delta_\fA) &\text{for $k$ odd}.  \end{cases}
\end{equation}
Furthermore, for every $\omega \in \sP(\fA)$ the homotopy groups of the isotropy group $\Unitary_\omega (\fA)$ are  given by 
\begin{equation}\label{eq:homotopy_groups_Uomega(A)_UHF}
  \pi_k(\Unitary_\omega(\fA)) \cong \begin{cases} 0 &\text{for $k$ even}, \\
    \bbZ \times Q(\delta_\fA) &\text{for } k = 1, \\
    Q(\delta_\fA) &\text{for } k > 1 \text{ and } k \text{ odd}. \end{cases}
\end{equation}
\end{thm}

\begin{proof}
  By
  possibly passing to isomorphic $C^*$-algebras we can assume without loss of generality that
  the inductive system defining $\fA$  is of the form
  \[
    \bbC I \subset  \fA_0 \subset \fA_1 \subset \cdots  \ , 
  \]
  where $I$ is the unit of $\fA$. Then $\fA = \overline{\bigcup_{j\in \bbN} \fA_j}$
  and  $\Unitary(\fA) = \overline{\bigcup_{j \in \bbN} \Unitary(\fA_j)}$, where the latter equality
  has been shown by \cite[Lem.\ 3.1]{Glimm_UHF_Algebras}.

  We want to prove a similar formula for the isotropy group $\Unitary_\omega(\fA)$. 
  To this end observe that by \cite[Cor.\ 3.8]{Powers_UHF_Representations} there exists for every pair
  of pure states $\psi, \omega \in \sP(\fA)$ an automorphism $\alpha \in \Aut(\fA)$ such that $\psi = \omega \circ \alpha$ . Then $\alpha$ restricts to an isomorphism of topological groups $\Unitary_\psi(\fA) \rightarrow \Unitary_\omega(\fA)$. Therefore, without loss of generality, we may choose $\omega$ to be any pure state we like; in particular, we may choose $\omega$ such that $\omega|_{\fA_j}$ is pure for all $j\in \bbN$. Then Lemma \ref{lem:UHF_unitary_closures} implies that $\Unitary_\omega(\fA) = \overline{\bigcup_{j\in\bbN} \Unitary(\fA_j) \cap \Unitary_\omega(\fA)}$.
  
  Formula \eqref{eq:relation-classifying-group-type} is an arithmetic result relating an abelian
  group obtained directly from the sequence $(n_j)_{j\in\N}$ with the abelian group constructed from
  the associated supernatural number $\delta_\fA$. The formula is proved in
  \cite[Lem.~7.4.4 (i)]{RordamLarsenLaustsenKtheory}.

Next we will define a well-filled chart of $\Unitary(\fA)$, which restricts to a well-filled chart of $\Unitary_\omega(\fA)$, whose domain contains the identity $I \in \fA$. Then the remarks preceding the theorem will yield
\begin{equation}\label{eq:colimit_unitary_groups}
\pi_k(\Unitary(\fA)) = \colim_{j \in \bbN} \pi_k(\Unitary(\fA_j))
\end{equation}
and
\begin{equation}\label{eq:colimit_isotropy_groups}
\pi_k(\Unitary_\omega(\fA)) = \colim_{j\in \bbN} \pi_k(\Unitary(\fA_j) \cap \Unitary_\omega(\fA)).
\end{equation}
In particular, $\Unitary(\fA_j) \cong \Unitary(n_j)$ and $\Unitary(\fA_j) \cap \Unitary_\omega(\fA) \cong \Unitary(1) \times \Unitary(n_j - 1)$ by Lemma \ref{lem:UHF_unitary_closures}. These spaces are path-connected, so the homotopy groups are independent of the base point. We will define our well-filled chart, then analyze the colimit.

As in Lemma \ref{lem:Cayley_transform}, let $O = \qty{U \in \Unitary(\fA): \norm{I - U} < 2}$ and define $\phi:O \rightarrow \fA_\tn{sa}$ by $\phi(U) = i(I - U)(I + U)^{-1}$. Then $\phi$ is a homeomorphism and $\phi|_{O \cap \Unitary(\fA_j)}$ is a homeomorphism onto $(\fA_j)_\tn{sa}$ for all $j \in \bbN$, so $\phi$ is a well-filled chart for $\Unitary(\fA)$ and $I \in O$, as desired.

Since $\phi$ is a homeomorphism, the restriction
\[ \phi|_{O \cap \Unitary_\omega(\fA)}:O \cap \Unitary_\omega(\fA) \rightarrow \phi(O \cap \Unitary_\omega(\fA)) \]
is also a homeomorphism when $\phi(O \cap \Unitary_\omega(\fA))$ is given the subspace topology inherited from $\fA_\tn{sa}$. Lemma \ref{lem:Cayley_transform} entails that $\phi(O \cap \Unitary_\omega(\fA))$ is a closed linear subspace of $\fA_\tn{sa}$. For each $j \in \bbN$, Lemma \ref{lem:UHF_unitary_closures} states that $\Unitary(\fA_j) \cap \Unitary_\omega(\fA) = \Unitary_{\omega|_{\fA_j}}(\fA_j)$, so Lemma \ref{lem:Cayley_transform} implies that $\phi(O \cap \Unitary_\omega(\fA) \cap \Unitary(\fA_j))$ is a closed linear subspace of
$(\fA_j)_\tn{sa}$, hence also a closed linear subspace of $\phi(O \cap \Unitary_\omega(\fA))$. We see that $\phi|_{O \cap \Unitary_\omega(\fA)}$ is a well-filled chart with $I \in O \cap \Unitary_\omega(\fA)$, as desired.

We now analyze the colimits. Denote by $\iota_{ij}:\fA_i \rightarrow \fA_j$ the canonical inclusions
for $i\leq j$. Then there exist $*$-isomorphisms $\sigma_i:\fA_i \rightarrow M_{n_i}(\bbC)$ such that
\begin{equation}\label{eq:matrix_algebra_inclusion}
  \sigma_{ij}(A) = \mqty(A&&&\\&A&&\\&&\ddots &\\&&&A)
  \quad \text{for all } A \in M_{n_i}(\bbC)\ ,
\end{equation}
where $\sigma_{ij} = \sigma_j \circ \iota_{ij} \circ \sigma_i^{-1}: M_{n_i}(\bbC) \to M_{n_j}(\bbC)$ and
where there are $n_{ij} = n_j/n_i \in \bbN$ copies of $A$ on the diagonal.
Choose $\omega$ to be the unique pure state determined by setting $\omega(A)$ to be the top left entry of the matrix $\sigma_i(A)$ for all $A \in \fA_i$. The colimit \eqref{eq:colimit_unitary_groups} is isomorphic to the colimit of the homomorphisms on homotopy groups $\pi_k(\Unitary(n_i))$ that are induced by the inclusions \eqref{eq:matrix_algebra_inclusion} restricted to the unitary groups $\Unitary(n_i)$. Furthermore, 
\[
\sigma_i(\Unitary_\omega(\fA) \cap \Unitary(\fA_i)) = \qty{\mqty(z & \\ & U) : z \in \Unitary(1), \, U \in \Unitary(n_i - 1)} \cong \Unitary(1) \times \Unitary(n_i - 1).
\]
Therefore, the colimit \eqref{eq:colimit_isotropy_groups} is isomorphic to the colimit of the homorphisms on homotopy groups $\pi_k(\Unitary(1) \times \Unitary(n_i - 1))$ that are induced by the continuous maps $g_{ij}:\Unitary(1) \times \Unitary(n_i - 1) \rightarrow \Unitary(1) \times \Unitary(n_j - 1)$ defined by
\begin{equation}
  \label{eq:pr1ofg}
  (\operatorname{pr}_1 \circ g_{ij})(z,U) = z
\end{equation}
and
\begin{equation}
  \label{eq:pr2ofg}
  (\operatorname{pr}_2 \circ g_{ij})(z, U) = \mqty(U&&&&&\\&z&&&&\\&&U&&&\\&&&z&&\\&&&&\ddots&&\\&&&&&U) \ ,
\end{equation}
where $\operatorname{pr}_1$ and $\operatorname{pr}_2$ are the projections.

We consider the former colimit. Recall that the map
\[ f_{ij}:\Unitary(n_i) \rightarrow \Unitary(n_j), \: U \mapsto \diag(U, I)\]
  induces isomorphisms on homotopy groups $\pi_k$ for $k < 2n_i$.
  By the Bott periodicity
theorem \cite{BottStableHomotopyClassicalGroups}, the homotopy groups of the unitary groups
$\Unitary(n)$ are given for $k < 2n$ by 
\[
\pi_k(\Unitary(n)) = \begin{cases} 0 &\text{if } k \tn{ is even,} \\ \bbZ &\text{if } k \tn{ is odd.} \end{cases}
\]
Thus, when $k$ is even, the colimit is zero. Fix $k$ odd and let $i_0 \in \bbN$ be the smallest natural number such that $n_{i_0} > k/2$. Choosing a generator $x \in \pi_k(\Unitary(n_{i_0}))$ we obtain generators $(f_{i_0i})_*x \in \pi_k(\Unitary(n_i))$ for all $i \geq i_0$. Then for all $j\geq i \geq i_0 $,
\[
(\sigma_{ij})_*(f_{i_0i})_*x = n_{ij}(f_{i_0 j})_*x,
\]
where we have restricted $\sigma_{ij}$ to the unitary groups. The homomorphisms $(\sigma_{ij})_*$ are thus multiplication by $n_{ij}$.
Now define homomorphisms
\[
  \sigma_i: \pi_k(\Unitary(n_i)) \rightarrow Q(\delta_\fA)=
  \bigcup_{j\in \N} n_j^{-1} \Z\]
as follows. If $i\geq i_0$, let $\sigma_i$ by the unique group homomorphism
mapping the generator $(f_{i_0i})_*x$ to $1/n_{i}$. If $i< i_0$, put
$\sigma_i = \sigma_{i_0}(\sigma_{ii_0})_*$. By construction, the relation
$\sigma_i = \sigma_{j}(\sigma_{ij})_*$ then is fulfilled for all $j\geq i$.
Since the union of the images of
the homomorphisms $\sigma_i$ coincides with  $Q(\delta_\fA)$ and since
$\sigma_i$  is injective for $i\geq i_0$, $Q(\delta_\fA)$
together with the family $(\sigma_j)_{j\in \N}$ is the directed colimit
we are looking for and formula \eqref{eq:homotopy_groups_U(A)_UHF} is proven.

 We now consider the direct system $g_{ij}:\Unitary(1) \times \Unitary(n_i - 1) \rightarrow \Unitary(1) \times \Unitary(n_j - 1)$.
 For the homotopy groups $\pi_k$ with $k > 1$, the analysis proceeds in an analogous way using the fact that
 the embedding $ \Unitary(n_j - 1) \hookrightarrow \Unitary(1) \times \Unitary(n_j - 1)$, $U \mapsto (1,U)$ 
 induces an isomorphism  $\pi_k(\Unitary(n_j - 1)) \to \pi_k(\Unitary(1) \times \Unitary(n_j - 1))$.
 Choose
 $i_0$ such that $n_{i_0}  > k/2+1$. If $k$ is even, then
 $\pi_k(\Unitary(1) \times \Unitary(n_j - 1)) = \pi_k(\Unitary(n_j - 1)) = 0$ for all $j \geq i_0$, hence the colimit
 $\colim\limits_{j\in \N} \pi_k(\Unitary(1) \times \Unitary(n_j - 1))$ is trivial.    
 If $k$ is odd and $j\geq i \geq i_0 $, the homomorphism $(g_{ij})_*$ maps a generator of
 $\pi_k(\Unitary(1) \times \Unitary(n_i - 1))$ to $n_{ij}$ times a generator of $\pi_k(\Unitary(1) \times \Unitary(n_j - 1))$,
 hence the colimit coincides with $Q(\delta_\fA)$ as before. In case $k = 1$ we have $\pi_1(\Unitary(1) \times \Unitary(n_i - 1)) \cong \bbZ \times \bbZ$ for all $i\geq i_0$.
 Denote 
 by $f_j : \Unitary (1) \to \Unitary (1) \times \Unitary (n_j-1) $ the map
 $z \mapsto (1,\diag (z,I))$ and by
 $h_j : \Unitary (1) \to \Unitary (1) \times \Unitary (n_j-1) $ the map
 $z \mapsto (z,I)$. After choice of a generator $x \in \pi_1(\Unitary (1))$
 the elements $x_j = (h_j)_* x$ and $y_j  = (f_j)_* x$ then are generators of
 $\pi_1(\Unitary (1) \times \Unitary (n_j-1))$.  Inspection of equations
 \eqref{eq:pr1ofg} and \eqref{eq:pr2ofg} then shows that for $j \geq i$ 
\begin{align*}
(g_{ij})_*x_i &= x_j + (n_{ij} - 1)y_j \ ,\\
(g_{ij})_*y_i &= n_{ij}y_j \ .
\end{align*}
In other words, $(g_{ij})_*$ is given by the matrix 
\[
\mqty(1 & 0 \\ n_{ij} - 1 & n_{ij}).
\]
Now let the homomorphisms
$g_j : \pi_1(\Unitary(1) \times \Unitary(n_j - 1)) \rightarrow \bbZ \times Q$
be given by multiplication by the matrix
\[
\mqty(1 & 0 \\ 1 + n_j^{-1} & n_j^{-1}) \ .
\]
Then one checks easily that $ g_j (g_{ij})_* = g_i$ for all $j\geq i$. 
The union of the images of the maps $g_j$ covers  $\bbZ \times Q$. Moreover, each
of the maps $g_j$ is injective, hence  $\bbZ \times Q$ together with the family
of maps $(g_j)_{j\in \N}$ provides the colimit of the inductive system
of abelian groups $\left( \pi_1(\Unitary(1) \times \Unitary(n_j - 1)), (g_{ij})_*\right)_{i\leq j}$.
This finishes the proof of \eqref{eq:homotopy_groups_Uomega(A)_UHF}.
\end{proof}

\begin{rem}
  The proof of the theorem applies to more general situations. Namely, if $\fA$ is the colimit of a strict inductive system of $C^*$-algebras
  $(\fA_j)_{j \in J}$, not necessarily countable, then the same argument as above
  using \cite[Thm.\ 1.13]{glockner2010homotopy} yields
\[
\pi_k(\Unitary(\fA), U) = \colim_{j \in J_U} \pi_k(\Unitary(\fA_j), U) \tn{ for all $k \in \bbN$ and $U \in \bigcup_{j \in J} \Unitary(\fA_j)$},
\]
where $J_U = \qty{j \in J: U \in \Unitary(\fA_j)}$. This result was shown by Handelman in
\cite[Prop.\ 4.4]{HandelmanK0_AF_Algebras} through a different method of proof.
Schr\"oder computed in \cite{SchroederUHFalgebras} the homotopy groups of the regular group of a UHF algebra which is
homotopy equivalent to its unitary group. Schr\"oder's result therefore entails ours.  
However, to our knowledge, the homotopy groups $\pi_k(\Unitary_\omega(\fA))$ for a UHF algebra $\fA$ and pure state $\omega$ have
not been computed before. 
\end{rem}

We now can show the claimed nontriviality of
the bundle $\Unitary(\fA) \rightarrow \sP_\omega(\fA)$. 

\begin{cor}
  For every infinite dimensional UHF algebra $\fA$ and pure state $\omega$ on it
  the bundle  $p_{\Unitary(\fA)} : \Unitary(\fA) \rightarrow \sP_\omega(\fA)$ is nontrivial. 
\end{cor}

\begin{proof}
  As  a consequence of the preceding theorem, the   rationalized fundamental  groups
of $\Unitary(\fA)$
\[ \pi_1(\Unitary(\fA))\otimes_\Z \Q \cong Q(\delta_\fA) \otimes_\Z Q \cong \Q\] and of the trivial bundle
$\Unitary_\omega(\fA)\times \sP_\omega(\fA)$ 
\[\pi_1(\Unitary_\omega(\fA)\times \sP_\omega(\fA))\otimes_\Z \Q \cong (\bbZ \times A (\delta_\fA))\otimes_\Z \Q
\cong \Q^2\]
are not isomorphic, hence $p_{\Unitary(\fA)} : \Unitary(\fA) \rightarrow \sP_\omega(\fA)$ can not be trivial.
\end{proof}

The theorem also allows to compute the topological K-theory of a UHF algebra.

\begin{cor}
  Under the assumptions of the theorem, the K-theory of the UHF algebra $\fA$ is given by
  \begin{equation}
    \label{eq:ktheoryUHF}
    K_k(\fA) =
    \begin{cases}
      Q(\delta_\fA) & \text{for } k=0 ,\\
      0  & \text{for } k=1 .
    \end{cases}
  \end{equation}
\end{cor}

\begin{proof}
  Denote by $\GL_n (\fA)$ and $\Unitary_n (\fA)$ the groups of invertible respectively unitary
  $n\times n$-matrices with entries in $\fA$. Note that both are topological groups in a natural
  way and that $\Unitary_n (\fA)$ is a deformation retract of  $\GL_n (\fA)$
  \cite[Sec.~8.1]{BlaKTOA2nd}. Let $\GL_\infty (\fA)$ and $\Unitary_\infty (\fA)$
  denote the colimits within the category of topological groups of the inductive systems
  $(\GL_n (\fA))_{n\in \N}$ and $(\Unitary_n (\fA))_{n\in \N}$, respectively. 
  The topological K-theory of $\fA$ can now be defined as the homotopy groups 
  \begin{equation}
    \label{eq:topKtheory}
    K_k(\fA) = \begin{cases}
      \pi_1 \big( \GL_\infty (\fA)\big) \quad \text{for } k = 0 \ , \\
      \pi_0 \big( \GL_\infty (\fA)\big) \quad \text{for } k = 1 \ .
    \end{cases}
  \end{equation}
  Using as before Gl\"ockner's results \cite{glockner2010homotopy} on the homotopy groups of colimits of
  direct systems of Banach Lie groups possessing well-filled charts or the sightly stronger
  direct limit charts one concludes that
  \[
   \pi_k \big( \GL_\infty (\fA)\big) = \colim_{n\in \N} \pi_{k} \big( \GL_n (\fA)\big)  \ . 
  \]
  Since $\GL_n (\fA)$ and $\Unitary_n (\fA)$ are homotopy equivalent, one obtains
  \begin{equation}
    \label{eq:colimiglunitary}
    \pi_k \big( \GL_\infty (\fA)\big) = \colim_{n\in \N} \pi_{k} \big( \Unitary_n (\fA)\big)  \ .
  \end{equation}
  Note that this equality also holds true  when $\fA$ is replaced by any of the $C^*$-algebras
  $\fA_j \cong M_{n_j} (\C)$ defining $\fA$.   
  Now observe that when $i_0$ is chosen such that $n_{i_0} > 2k$ there is for all $m\geq n$ and $j\geq i > i_0$ 
  a commutative diagram
  \[
  \begin{tikzcd}
    \pi_{k} \big( \Unitary_n (\fA_i)\big) \arrow{r} \arrow[d,"\cdot n_{ij}"] &
    \pi_{k} \big( \Unitary_m (\fA_i)\big) \arrow[d,"\cdot n_{ij}"] \\
    \pi_{k} \big( \Unitary_n (\fA_j)\big) \arrow{r}   &
    \pi_{k} \big( \Unitary_m (\fA_j)\big) \ .
  \end{tikzcd}
  \]
  The horizontal morphisms in this diagram are induced by embeddings of the form
  $A \mapsto \diag (A,I)$ and are isomorphisms. The vertical morphisms are multiplication by $n_{ij}$,
  using notation from the proof of the theorem. 
  For all $m\geq n$ the induced maps 
  \[
     \colim_{j\in \N} \pi_{k} \big( \Unitary_n (\fA_j)\big) \to
    \colim_{j\in \N} \pi_{k} \big( \Unitary_m (\fA_j)\big) 
  \]
  are therefore isomorphisms, hence for all $n\geq 1$
  \[
    \pi_{k} \big( \Unitary_n (\fA)\big) = \pi_{k} \big( \Unitary (\fA)\big) =
    \begin{cases}
      Q(\delta_\fA) & \text{for } k \text{ even}\ , \\
      0 & \text{for } k \text{ odd} \ . \\
    \end{cases}
  \]  
  By \eqref{eq:topKtheory} and \eqref{eq:colimiglunitary} this entails the claim. 
\end{proof}

\begin{rem}
  The topological K-theory of a UHF algebra is well known; see e.g.~\cite{SchroederUHFalgebras},
  or \cite[Sec.~7.4]{RordamLarsenLaustsenKtheory}. 
  The virtue of the approach presented here is that it avoids the claim occasionally  
  made in the K-theoretic literature that the colimit topology on $\GL_\infty (\fA)$ is
  compatible with the underlying group structure. For $\fA$ infinite dimensional, this is in general not true
  as has been shown in \cite{TatShiHirGTURILTGCGD,YamILGLG}. 
  The appropriate way is to define $\GL_\infty (\fA)$ as the colimit of the direct system
  $(\GL_n (\fA))_{n\in \N}$ within the category of topological groups. Under this concept,
  $\GL_\infty (\fA)$ is the union of the groups $\GL_n (\fA)$ endowed with the natural group structure. 
  The correct topology turning it into a topological group is the \emph{bamboo shoot topology} \cite{TatShiHirGTURILTGCGD}
  which in general does not coincide with the colimit topology. Gl\"ockner's approach \cite{glockner2010homotopy}
  to determine the homotopy groups of colimits of Lie groups or manifolds addresses this fact. 
\end{rem}

%
%
\section{The fiberwise GNS construction}
\label{sec:fiberwiseGNS}
In this section we consider a norm-defined  $C^*$-algebra fiber bundle  $p:\fA \rightarrow X$
and construct some naturally associated fiber bundles.
For the convenience of the reader we collected in Appendices \ref{sec:fiber-bundles} and \ref{sec:smooth-fiber-bundles} 
several fundamental notions from infinite dimensional bundle theory which are used in the following. 

\subsection{The setup}
\label{sec:setup}
Throughout this section we assume to be given a norm-defined $C^*$-algebra bundle $p:\fA \to X$ over a locally path connected Hausdorff topological space $X$. 
This means in particular that $p:\fA \rightarrow X$ is a continuous surjection
such that $\fA_x = p^{-1}(x)$ carries the structure of a $C^*$-algebra for each $x \in X$,
that the typical fiber is a $C^*$-algebra $\fF$ and finally that the structure group 
is the group $\Aut(\fF)_{\textup{n}}$ of all automorphisms of $\fF$ endowed with the norm topology.

Given the  $C^*$-algebra bundle $p:\fA \to X$ we may construct the dual bundle,
a Banach bundle whose fibers are the dual spaces $\fA_x^*$ with the norm topology. More precisely, we set 
\[
   \fA^* = \coprod_{x \in X} \fA_x^*
\]
as a disjoint union of sets and let $\projAStar : \fA^* \rightarrow X$ be the natural projection.
For every local trivialization $(\varphi,O)$ of $\fA$, we define
$\varphi_* : \fA^*_{|O} = (\projAStar)^{-1}(O) \rightarrow O \times \fF^*$ by
$\varphi_*(x, \omega) =  (x, \omega \circ \varphi_x^{-1})$ for each $x\in X$. Since  $\varphi_x^{-1}$ is a
$*$-isomorphism, the map $\varphi_{*,x} : \fA_x^* \rightarrow \fF^*$ is a bijective linear isometry. Furthermore,
given local trivializations $(\varphi_i,O_i)$, $i=1,2$, of $\fA$, the transition map
\[
  g_{*,12}: O_1 \cap O_2 \rightarrow \Aut (\fF^*)_{\textup{n}}, \quad
  x \mapsto \varphi_{1,*, x} \circ \varphi_{2,*,x}^{-1} = (\varphi_{1,x} \circ \varphi_{2,x}^{-1})_*
\]
is continuous by norm continuity of the map
$\Aut (\fF) \rightarrow \GL (\fF^*)$, $\alpha \mapsto \alpha_* = (\alpha^{-1})^*$.
As for the underlying topology, $\fA^*$ will be endowed with the coarsest topology such that for each local
trivialization $(\varphi,O)$  of $\fA$ the set $\fA^*_{|O}$ is open and the map $\varphi_* $ is continuous.
Since all transition functions are continuous with values in $\Aut(\fF)_{\textup{n}}$, each local trivialization
$\varphi_* $ then is a homeomorphism, and $\projAStar: \fA^* \to X$ becomes a norm defined Banach vector bundle
with typical fiber the dual $\fF^*$.

Next we construct the subbundle $\projAPure: \sP(\fA) \to X$ of $\projAStar :\fA^* \to X$
consisting of \emph{fiberwise pure states} on $\fA$. As a set, let
\[ \sP(\fA) = \coprod_{x \in X}  \sP(\fA_x) \]
and endow $\sP(\fA)$  with the subspace topology from $\fA^*$. The restriction of $\projAStar$ to $\sP(\fA)$
will be denoted by $\projAPure$. By construction, $\projAPure : \sP(\fA) \to X$
then is a continuous surjection. Given a local trivialization $(\varphi,O)$ of $\fA$, the restriction of 
$\varphi_*$ to $\sP(\fA)_{|O} = \sP(\fA) \cap \fA^*_{|O}$ then maps each fiber
$\sP(\fA_x)$ onto $\sP(\fF)$ since $\varphi_{x}$ is a $*$-isomorphism. So
$\varphi_*|_{\sP(\fA)_{|O}}: \sP(\fA)_{|O} \to O \times \sP(\fF) $ is a local trivialization of $\sP(\fA)$.  
By the following lemma the automorphism group $\Aut(\fF)_{\textup{n}}$ acts effectively on $\sP(\fF)$, so
the family of such restrictions forms a trivializing atlas of $\sP(\fA)$ with norm continuous
transition maps. Hence $(\sP(\fA),X,\projAPure,\sP(\fF),\Aut(\fF)_{\textup{n}})$ is a
subbundle of the dual bundle $\projAStar : \fA^* \to X$.
We call $\projAPure: \sP(\fA) \to X$   the \emph{pure state bundle} of $\fA$.

\begin{lem}\label{lem:efective-action-pure-state-space}
  For any  $C^*$-algebra $\fB$, the automorphism group $\Aut (\fB)_{\textup{n}}$ acts effectively and
  continuously on the pure state space $\sP (\fB)$ by
  \[ \Aut (\fB)_\tn{n} \times \sP (\fB)_\tn{n} \to \sP (\fB)_\tn{n}, \quad 
    (\alpha,\omega) \mapsto \omega \circ \alpha^{-1} \ .
  \]  
\end{lem}

\begin{proof}
  We only need to show that  $\Aut (\fB)$ acts effectively on the pure state space. Continuity is obvious. To this end suppose that the automorphism $\alpha \in \Aut (\fB)$ leaves every pure state invariant.
  Then for every irreducible representation $(\ssH, \pi)$ of $\fB$, every $\Psi \in \ssH$, and every
  $A \in \fB$, we have $\langle \Psi, \pi(\alpha(A))\Psi \rangle = \langle \Psi, \pi(A)\Psi \rangle$, which implies
  that $\pi(\alpha(A)) = \pi(A)$ for all $A \in \fB$. This entails that $\rho \circ \alpha = \rho$ for a reduced
  atomic representation $\rho$ in the sense of \cite[Sec.~10.3]{KadisonRingroseII}. If $\fB$ is unital, a reduced
  atomic representation is faithful by \cite[Prop.~10.3.10]{KadisonRingroseII}, so it follows that $\alpha$ coincides with
  the identity automorphism and $\Aut (\fB)$ acts effectively on $\sP(\fB)$. If $\fB$ is non-unital, then the unique extension of $\alpha$ to a unital $*$-automorphism on the unitization of $\fB$ leaves every pure state on the unitization invariant, hence the non-unital case follows from the unital case. 
\end{proof}
 
\subsection{The GNS Hilbert bundle}
\label{sec:hilbert-bundle}
In this section, we construct  Hilbert bundles over pure state spaces whose fibers correspond to the Hilbert spaces of the associated GNS representations. The most general case, which is the last one presented in this section, is the construction of a Hilbert bundle $\ssH \to \sP(\fA)$ associated to a $C^*$-algebra bundle $\fA \to X$ which we call the \emph{GNS Hilbert bundle}. We begin with some preliminary constructions.

We start with a Hilbert space $\hilbH$ and construct a Hilbert bundle
$\projGNS:\ssH \to S$ over the space $ S:= \sP_* (\fB(\hilbH))_{\textup{n}}$ of pure normal states of the
$C^*$-algebra of bounded linear operators on $\hilbH$. Note that $S$ is assumed to carry the metric uniformity
induced by the norm on the dual $\fB (\hilbH)^*$.

As a set, let
\[
  \ssH = \coprod_{\varrho \in S} \hilbH_\varrho
\]  
be the disjoint union  of the GNS Hilbert spaces $\hilbH_\varrho$. Denote by
$\projGNS \!\! :  \ssH \to S$ the projection which associates to every vector
$\Psi \in \hilbH_\varrho$ the ``footpoint'' state $\varrho$.

Recall from Cor.~\ref{cor:PH_superselection_metric_equivalence} that the map
$r: \bbP \hilbH \to S$ which associates to each ray $\C\Psi$ the pure state $\psi$ it represents is a
bi-Lipschitz isomorphism of uniform spaces. For $v \in \bbS \hilbH$ let
$O_v \subset S$ be the open ball of radius $2$ around the state $r (\C v)$. Let $s_v:O_v \to \bbS \hilbH $
be the norm-continuous section of the canonical projection $\bbS\hilbH \to S$ from
Cor.~\ref{cor:purestate_mapsto_SH} such that $\langle s_v(\varrho),v \rangle >0$ for all $\varrho \in O_v$.  
Given a pure normal state $\varrho \in O_v$, denote by  $R_{v,\varrho}$  the cyclic representation
$(\hilbH, \id_{\fB(\hilbH)},\Rho_{v,\varrho})$ where $\Rho_{v,\varrho} = s_v (\varrho)$.

By construction, the state associated to the unit vector $\Rho_{v,\varrho}$ coincides with $\varrho$.
Hence by Prop.~\ref{prop:functorial-GNS} (\ref{ite:functoriality-sector-spaces})
the cyclic representation $R_{v,\varrho}$ is unitarily equivalent to the GNS representation
$(\hilbH_\varrho,\pi_\varrho,\Omega_\varrho)$ via a unique unitary map
$U_{v,\varrho} :\hilbH_\varrho \to\hilbH$ which maps $\Omega_\varrho$ to $\Rho_{v,\varrho}$. 
The unitary $U_{v,\varrho}$ can actually be written down explicitly. It is given by
\begin{equation}
  \label{eq:representation-unitary-S-rho}
    U_{v,\varrho} (A + \fN_\varrho ) = A(\Rho_{v,\varrho})\quad\text{for all } A \in \fB(\hilbH) \ .  
\end{equation}
%

From these data we can now define a local trivialization of $\ssH$ over $O_v$
by the following formula: 
\[
  \chi_v : \projGNS^{-1}(O_v) = \ssH|_{O_v} \to O_v \times \hilbH , \: \Psi \mapsto
  \left( \varrho, U_{v,\varrho} \Psi \right)  \quad\text{where } \varrho =  \projGNS (\Psi) \ . 
\]
Now endow $\ssH$ with the coarsest topology so that for every $v\in\bbS\hilbH$ the set $\ssH|_{O_v}$
is open and the local trivialization $\chi_v$ is continuous. 

Before we  determine the transition functions let us take a step back and consider
the tautological line bundle
$\projLine :\sL \to S \cong \bbP \hilbH$
whose fiber over $\varrho \in S$ consists of all vectors in the complex line
$r^ {-1} (\varrho)$. Then $\sL$  is a smooth line bundle over $S$ with structure group given by
$\Unitary (1)$. By construction  $\sL$ coincides with the pullback by $r^{-1}$ of the tautological line
bundle over $\bbP \cH$. A local smooth frame of $\sL$ over $O_v$ is given by the section
$s_v :O_v\to \bbS \hilbH$, $\varrho \mapsto \Rho_{v,\varrho}$, hence
\[
  \tau_v : \sL|_{O_v} \to O_v \times \C, \:  \Psi \mapsto
  \left( \projLine (\Psi)   , \langle \Rho_{v, p\raisebox{0.5pt}{$_{{}_{\sL}}$} (\Psi) }, \Psi \rangle \right) 
\]
is a local trivialization. Now let $w$ be another unit vector in $\hilbH$ and consider the  map
\[
  h_{v,w} : O_v\cap O_w  \to \Unitary (1), \: \varrho \mapsto \langle \Rho_{v,\varrho}, \Rho_{w,\varrho} \rangle \ .
\]
Then
\begin{equation}
  \label{eq:transition-condition-cyclic-vectors}
  \Rho_{v,\varrho} =  \overline{h_{v,w} (\varrho)} \cdot \Rho_{w,\varrho}
  \quad\text{for all } \varrho \in O_v \cap O_w \ ,
\end{equation}
hence one obtains for all $z \in \C$
\[
  \tau_v \circ \tau_w ^{-1} (\varrho , z) = (\varrho , h_{v,w} (\varrho) \cdot z) \ .
\]  
Therefore, the family $(h_{v,w})_{v,w\in \bbS\hilbH}$ forms a \v{C}ech cocycle of transition functions defining the
tautological line bundle over $S$. 

For the bundle $\projGNS: \ssH \to S$,
Eq.\ \eqref{eq:transition-condition-cyclic-vectors} implies that the map
\[
  \chi_v \circ \chi_w^{-1} : (O_v\cap O_w ) \times \hilbH \to (O_v\cap O_w )  \times \hilbH , \:
  (\varrho , A  (\Rho_{w,\varrho})) \mapsto  (\varrho , A  (\Rho_{v,\varrho})) \ ,
\]
is given by the fiberwise action of  $\overline{h_{v,w}} $  meaning that
\begin{equation}
  \label{eq:transition-maps-dual-action-tautologial-cocycle}
  \chi_v \circ \chi_w^{-1} (\varrho, \Psi) =  \overline{h_{v,w} (\varrho)} \cdot \Psi
  \quad\text{for all } (\varrho,\Psi) \in  (O_v\cap O_w ) \times  \hilbH \ .
\end{equation}  
On the one hand this implies that each transition function
\[
  O_v\cap O_w \to \Unitary (\sH) , \:
  \varrho \mapsto \chi_{v,\varrho} \circ \chi_{w,\varrho}^{-1}
\]
is a continuous map and therefore every local trivialization $\chi_v$  a
homeomorphism. Hence, $\projGNS: \ssH \to S$ is a Hilbert bundle as claimed.
On the other hand, Eq.\ \eqref{eq:transition-maps-dual-action-tautologial-cocycle}
entails that the structure group of the fiber bundle  $\projGNS: \ssH \to S$ can be
reduced to $\Unitary (1)$ and that a defining \v{C}ech cocycle is given  by the
\v{C}ech cocycle
$(\overline{h}_{v,w})_{v,w\in \bbS\hilbH}$ defining the dual tautological line bundle
$\sL^* \to S$.

\begin{rem}
  The argument above shows that $\projGNS: \ssH \to S$ can be identified with the
  associated bundle $\operatorname{Fr} (\sL^*)\times_{\Unitary(1)} \hilbH$, where
  $\operatorname{Fr} (\sL^*) \to S$ denotes the bundle of unitary frames of the dual tautological bundle.
  Note that $\operatorname{Fr} (\sL^*) \to S$  is a $\Unitary (1)$-principal bundle by construction.    
\end{rem}

We continue with a  $C^*$-algebra $\fF$ and construct a locally trivial Hilbert bundle
$\projGNS: \ssH \to \sP (\fF)_{\textup{n}}$ over the pure state space of $\fF$ which we assume to be endowed
with the metric induced by the norm on $\fF^*$.
As a set, the Hilbert bundle $\ssH$ we construct is the disjoint union of the
GNS Hilbert spaces $\hilbH_\varrho$, $\varrho \in \sP (\fF)$. The projection
$\projGNS :\ssH \to \sP (\fF)$ maps each element of $\hilbH_\varrho$ to $\varrho$.
Since the pure state space $\sP (\fF)$ is the disjoint union of sectors, we obtain
the decomposition
\begin{equation}
  \label{eq:decomposion-hilbert-bundle-sector-parts}
  \ssH = \coprod_{\varrho \in \sP (\fF)} \hilbH_\varrho =
  \coprod_{S\in {\ssec} (\fF)}\coprod_{\varrho \in S } \hilbH_\varrho \ ,
\end{equation}
where ${\ssec} (\fF)$ denotes the space of superselection sectors of $\fF$.

Now fix for every  sector $S\in {\ssec} (\fF)$ a pure state $\varrho_S\in S$
and a cyclic representation $(\hilbH_S,\pi_S,\Omega_S)$ such $\Omega_S$ is a unit
vector representing the state $\varrho_S$, e.g.~the GNS representation of $\varrho_S$.
We then apply over each sector $S$  the construction of
a Hilbert bundle $p_{\ssH_S} : \ssH_S \to S$ described above. Since each sector is
an open connected component of $\sP (\fF)_{\textup{n}}$ and by the decomposition
(\ref{eq:decomposion-hilbert-bundle-sector-parts}), the projection
$\projGNS: \ssH \to \sP (\fF)_{\textup{n}}$ carries a unique structure of a
Hilbert bundle such that each subspace
$\ssH_{|S} = \ssH_S$ is open  and such that the canonical injection 
$\ssH_S \hookrightarrow \ssH$ is an embedding of Hilbert bundles.
Note that for different sectors $S,S'\in {\ssec} (\fF)$, the associated Hilbert spaces
$\hilbH_S$ and $\hilbH_{S'}$ might not be (canonically) isomorphic, so the
typical fiber might change from component to component. 
%
%

A \v{C}ech cocyle of transition functions for the Hilbert bundle
$\projGNS: \ssH \to\sP (\fF)_{\textup{n}} $ is given as follows.
Given a sector $S$ and a vector $v\in \bbS\hilbH_S$ let $O_{S,v}\subset S$ be the open
ball of radius $2$ around the state $r_S (\C v)$, where
$r_S : \bbP\hilbH_S \to S = \sP_{\pi_S} (\fF)_{\textup{n}}$ denotes the bi-Lipschitz
isomorphism from Cor.~\ref{cor:PH_superselection_metric_equivalence}.
Let $s_{v} : O_{S,v} \to \bbS\hilbH_S$ be the unique section
according to Cor.~\ref{cor:purestate_mapsto_SH} such that $\ev{s_{v}(\varrho),v} > 0 $
for all $\varrho \in O_{S,v}$. The family
of transition functions
\[
  h_{S,v,w} : O_{S,v}\cap O_{S,w}  \to \Unitary (1), \:
  \varrho \mapsto \langle s_{v} (\varrho) , s_{w} (\varrho) \rangle \ ,
\]
where $S$ runs through the superselection sectors of $\fF$ and 
$v,w$ through the unit vectors of $\hilbH_S$ 
then is a \v{C}ech cocycle whose dual \v{C}ech cocycle
$(\overline{h}_{S,v,w})_{S\in {\ssec} (\fF), \: v,w\in \bbS\hilbH_S}$
characterizes  the Hilbert bundle
$\projGNS: \ssH \to \sP (\fF) $ with typical fiber $\hilbH_S$ over the sector $S$
up to isomorphism. 

Before we tackle the final and general case we need an auxiliary result showing that the
unitary associated to an automorphism of a $C^*$-algebra according to
Prop.~\ref{prop:functorial-GNS} (\ref{ite:naturality-GNS}) depends continuously in the norm topology
on the argument. This result will be used to show that out of a system of transition functions
for a $C^*$-algebra bundle $\fA$ we get a system of $\Unitary (\hilbH)$-valued transition maps
for the Hilbert bundle to be constructed. 

For the following proposition, let $\Iso(\fB, \fC)_\tn{n}$ be the set of $*$-isomorphisms between $C^*$-algebras $\fB$ and $\fC$, with the subspace topology inherited from the norm topology on the bounded linear operators $\fB \rightarrow \fC$. Likewise, $\Unitary(\cH, \tilde{\hilbH})_\tn{n}$ is the set of unitary operators $\cH \rightarrow \tilde{\hilbH}$ with the norm topology.

\begin{prop}\label{prop:continuous-dependancy-unitary-automorphism}
Let $\fB$ be a $C^*$-algebra, $(\hilbH,\pi)$ a nonzero irreducible representation,
  and $r: \bbP\hilbH \to \sP_\pi(\fB)_{\textup{n}}$ the uniform isomorphism associated to $(\hilbH,\pi)$. For every $\Omega \in \bbS\hilbH$, let $s_\Omega : \bbB_2(r(\bbC\Omega)) \to \bbS\cH$ be the unique section of the canonical projection $r \circ p_{\bbS\hilbH}$
  such that $\langle \Omega, s_\Omega (\omega) \rangle > 0$ for all $\omega\in \bbB_2 (r(\bbC\Omega))$, and let $\xi_\Omega:\fB \rightarrow \hilbH$ be the map $\xi_\Omega(B) = \pi(B)\Omega$. Let $\fC$ be another $C^*$-algebra
  with nonzero irreducible representation $(\tilde{\hilbH},\tilde{\pi})$ and associated maps $\tilde r$, $\tilde s_{\Psi}$, and $\tilde\xi_{\Psi}$ for $\Psi \in \bbS \tilde{\hilbH}$. Define
  \[
  O = \qty{(\alpha, \Omega, \Psi) \in \Iso(\fB, \fC)_\tn{n} \times \bbS \hilbH \times \bbS \tilde{\hilbH} : \alpha_*r(\bbC \Omega) \in \bbB_2(\tilde r(\bbC\Psi))}
  \]
  Then the following hold true:
  \begin{enumerate}[{\rm (i)}]
    \item\label{ite:O_nonempty} If $\fB = \fC$ and $(\cH,\pi) = (\tilde{\hilbH},\tilde{\pi})$, then for each $\Omega, \Psi \in \bbS \hilbH$ there exists $\alpha \in \Aut(\fB)$ such that $(\alpha, \Omega, \Psi) \in O$.
    \item\label{ite:O_open} The set $O$ is open in $\Iso(\fB,\fC)_\tn{n} \times \bbS\hilbH \times \bbS\tilde{\hilbH}$.
    \item\label{ite:O_Phi_continuous} The map $\Phi:O \rightarrow \bbS\tilde{\hilbH}$, $\Phi(\alpha, \Omega, \Psi) = \tilde s_{\Psi}\alpha_*r(\bbC\Omega)$ is continuous.
    \item\label{ite:O_U_continuous} The unique map $U:O \rightarrow \Unitary(\hilbH,\tilde{\hilbH})_\tn{n}$ making the diagram 
    \[
    \begin{tikzcd}[column sep = 1.5cm]
    \fB \arrow[r,"\alpha"] \arrow[d,"\xi_{\Omega}"']& \fC \arrow[d,"\tilde\xi_{\Phi(\alpha,\Omega,\Psi)}"]\\
    \hilbH \arrow[r,"{U{(\alpha,\Omega,\Psi)}}"']& \tilde{\hilbH}
   \end{tikzcd}
    \]
    commute for all $(\alpha, \Omega,\Psi) \in O$ is continuous.
  \end{enumerate}
\end{prop}

\begin{proof}
(\ref{ite:O_nonempty}) This follows from Theorem \ref{thm:superselection_sector_equivalences} when $\fB$ is unital. When $\fB$ is non-unital, we may extend $r(\bbC \Omega)$ and $r(\bbC\Psi)$ to pure states on the unitization $\widetilde \fB$ in the same superselection sector. We may then use Theorem \ref{thm:superselection_sector_equivalences} to find a unitary in $\widetilde \fB$ that generates an inner automorphism that relates the two extended states. Since $\fB$ is a two-sided ideal in $\widetilde \fB$, this automorphism restricts to an automorphism $\alpha$ on $\fB$ such that $(\alpha, \Omega, \Psi) \in O$.

(\ref{ite:O_open}) Observe that $O$ is the preimage of $[0,2)$ under the continuous map
\[
  \Iso(\fB,\fC)_\tn{n} \times \bbS\hilbH \times \bbS\tilde{\hilbH} \rightarrow [0,2], \quad (\alpha,\Omega,\Psi) \mapsto \norm{\alpha_*r(\bbC\Omega) - \tilde r(\bbC\Psi)} \ .
\]

(\ref{ite:O_Phi_continuous}) Fix $(\alpha, \Omega, \Psi) \in O$ and $\varepsilon > 0$, and choose $\delta > 0$. Let $(\alpha', \Omega', \Psi') \in O$ such that $\norm{\alpha - \alpha'} < \delta$, $\norm{\Omega - \Omega'} < \delta$, and $\norm{\Psi - \Psi'} < \delta$. Choose $\delta$ small enough so that this implies $(\alpha', \Omega', \Psi) \in O$, which we can do since $O$ is open. We can choose $\delta$ small enough so that $\norm{\Phi(\alpha, \Omega, \Psi) - \Phi(\alpha', \Omega', \Psi)} < \varepsilon/2$ by continuity of $\tilde s_\Psi$ and continuity of $(\alpha, \Omega) \mapsto \alpha_*r(\bbC\Omega)$. We must show that $\delta$ can be chosen small enough so that $\norm{\Phi(\alpha', \Omega', \Psi) - \Phi(\alpha', \Omega', \Psi')} < \varepsilon/2$. Since $\Phi(\alpha', \Omega', \Psi)$ and $\Phi(\alpha', \Omega', \Psi')$ both represent the state $\alpha'_*r(\bbC\Omega')$ in the representation $\tilde{\pi}$, there exists $\lambda \in \Unitary(1)$ such that $\Phi(\alpha', \Omega', \Psi) = \lambda \Phi(\alpha', \Omega', \Psi')$. Then
\begin{align*}
\sqrt{1 - \frac{1}{4}\norm{\alpha_*'r(\bbC\Omega') - \tilde r(\bbC\Psi)}^2} &= \ev{\Psi, \Phi(\alpha', \Omega', \Psi)}\\
&= \lambda \ev{\Psi - \Psi', \Phi(\alpha', \Omega', \Psi')} + \lambda \ev{\Psi', \Phi(\alpha', \Omega', \Psi')}\\
&= \lambda \ev{\Psi - \Psi', \Phi(\alpha', \Omega', \Psi')} \\
&\quad \quad + \lambda \sqrt{1 - \frac{1}{4}\norm{\alpha_*'r(\bbC\Omega') - \tilde r(\bbC \Psi')}^2}.
\end{align*}
We may shrink $\delta$ to make $\lambda \ev{\Psi - \Psi', \Phi(\alpha', \Omega', \Psi')}$ arbitrarily small and the two square roots above arbitrarily close to each other. It follows that we can make $\abs{1 - \lambda} < \varepsilon/2$, yielding $\norm{\Phi(\alpha',\Omega', \Psi) - \Phi(\alpha', \Omega', \Psi)} < \varepsilon/2$, as desired.

(\ref{ite:O_U_continuous}) Recall that the map $U$ as described above exists because $(\cH, \pi, \Omega)$ and $(\tilde{\hilbH}, \tilde{\pi} \circ \alpha, \Phi(\alpha, \Omega, \Psi))$ are both cyclic representations representing the pure state $r(\bbC\Omega)$. In particular, we have
\begin{align}
\label{eq:properties-unitary-intertwiner}  
U{(\alpha, \Omega, \Psi)}\Omega = \Phi(\alpha, \Omega, \Psi) \qqtext{and} U{(\alpha, \Omega, \Psi)}\pi(B) = \tilde{\pi}(\alpha(B))U{(\alpha, \Omega, \Psi)}.
\end{align}
for all $(\alpha, \Omega, \Psi) \in O$ and $B \in \fB$.

Fix $(\alpha, \Omega, \Psi) \in O$ and $\varepsilon > 0$, and choose $\delta > 0$. Let $(\alpha', \Omega', \Psi') \in O$ such that $\norm{\alpha - \alpha'} < \delta$, $\norm{\Omega - \Omega'} < \delta$, and $\norm{\Psi - \Psi'} < \delta$. Choose $\delta$ small enough so that this implies $(\alpha', \Omega, \Psi), (\alpha', \Omega', \Psi) \in O$, which we can do since $O$ is open. Let $v \in \bbS\cH$ be arbitrary and choose $B, B' \in \fB$ such that $\norm{B}, \norm{B'} \leq 1$ and $v = \pi(B)\Omega = \pi(B')\Omega'$, which we can do by Theorem \ref{thm:Pedersen}.
Observe the following:
\begin{align*}
\norm{U(\alpha, \Omega, \Psi)v - U(\alpha', \Omega, \Psi)v} &= \norm{\tilde{\pi}(\alpha(B))\Phi(\alpha, \Omega, \Psi) - \tilde{\pi}(\alpha'(B))\Phi(\alpha', \Omega, \Psi)}\\
&\leq \norm{\alpha - \alpha'} + \norm{\Phi(\alpha, \Omega, \Psi) - \Phi(\alpha',\Omega, \Psi)}\\
  \norm{U(\alpha', \Omega, \Psi)v - U(\alpha', \Omega', \Psi)v} &= \norm{\tilde{\pi}(\alpha'(B))\Phi(\alpha',\Omega,\Psi) - \tilde{\pi}(\alpha'(B))U(\alpha',\Omega',\Psi)\Omega}\\
  &\leq \norm{\Phi(\alpha',\Omega,\Psi) - \Phi(\alpha',\Omega',\Psi)} +
    \norm{\Omega - \Omega'}\\
 \norm{U(\alpha',\Omega',\Psi)v - U(\alpha',\Omega',\Psi')v} &= \norm{\tilde{\pi}(\alpha'(B'))\Phi(\alpha',\Omega',\Psi) - \tilde{\pi}(\alpha'(B'))\Phi(\alpha',\Omega',\Psi')}\\
&\leq \norm{\Phi(\alpha',\Omega',\Psi) - \Phi(\alpha',\Omega',\Psi')}.
\end{align*}
By continuity of $\Phi$, all of the quantities on the right can be made less than $\varepsilon/6$ by shrinking $\delta$, independently of $v$. Thus, by the triangle inequality, we may choose $\delta$ such that $\norm{U(\alpha,\Omega,\Psi)v - U(\alpha',\Omega',\Psi')v} < 5\varepsilon/6$ for all $v \in \bbS\cH$, which implies $\norm{U(\alpha,\Omega,\Psi) - U(\alpha',\Omega',\Psi')} < \varepsilon$, proving continuity of $U$.
\end{proof}

   Finally we consider the general case of a norm defined $C^*$-algebra bundle
  $p:\fA\to X$ with typical fiber $\fF$ and construct a locally trivial Hilbert bundle $\projGNS : \ssH \to \sP (\fA) $. 
  As before, we put as a set
\[
  \ssH = \coprod_{\varrho \in \sP(\fA)} \hilbH_\varrho \ .
\]
Let $(\varphi_i)_{i\in I}$ be an atlas of local trivializations
$\varphi_i: p^{-1} (O_i)\to O_i \times \fF$, where each trivializing
domain $O_i$  is an open subset of $X$.  
The intersections $O_i\cap O_j$ will be denoted by $O_{ij}$ and
the transition functions by $g_{ij} :O_{ij} \to \Aut (\fF)_\textup{n}$. 

As explained above, the trivializing atlas $(\varphi_i)_{i\in I}$
induces an atlas of local trivializations
$\varphi_{i,*} : \sP(\fA)|_{O_i} \to O_i \times \sP(\fF)$, $i\in I$, of
the pure state bundle $p_{\sP(\fA)}:\sP(\fA)\to X$. 
Now choose for every  sector $S\in {\ssec} (\fF)$ a pure state $\varrho_S\in S$
and a cyclic representation $(\hilbH_S,\pi_S,\Omega_S)$ such that the unit vector
$\Omega_S$ represents the state $\varrho_S$. Prop.~\ref{prop:functorial-GNS} (\ref{ite:naturality-GNS}) shows that given two sectors $R, S \in \ssec(\fF)$ for which there exists an automorphism $\alpha \in \Aut(\fF)$ such that $\alpha_*\varrho_S = \varrho_R$, the representation $(\cH_R, \pi_R, \Omega_R)$ is unitarily equivalent to a cyclic representation in $\cH_S$. This defines an equivalence relation on sectors. Therefore, after possibly switching to unitarily equivalent representations, we can assume that all sectors which are equivalent with respect to this relation are represented on the same Hilbert space. In particular, we can assume that for $i,j \in I$ with $O_{ij} \neq \emptyset$, the Hilbert spaces $\cH_S$ and $\cH_{\overline{g_{ij}(x)}(S)}$ coincide for all $x \in O_{ij}$, where
$\overline{g_{ij}(x)}:\ssec (\fF) \to \ssec (\fF)$ denotes the permutation of
the sectors induced by the automorphism $g_{ij}(x)$ of $\fF$ according to
Prop.~\ref{prop:functorial-GNS} (\ref{ite:functoriality-sector-spaces}).

Given $i\in I$, a sector $S\in \ssec (\fF)$ and a unit vector
$v \in \bbS\hilbH_S$ we define 
\begin{equation}\label{eq:oisv}
O_{i,S,v}:=\varphi_{i,*}^{-1} (O_i \times O_{S,v})\subset \sP (\fA),
\end{equation}
where as before $O_{S,v}\subset S$ denotes the open ball of radius $2$
around the state $r_S (\C v)$. Here, $r_S:\bbP\cH_S \rightarrow \sP_{\pi_S}(\fF)_\tn{n}$ is the uniform isomorphism associated to the representation $(\cH_S, \pi_S, \Omega_S)$.  Recall that we have a canonical section $s_v :O_{S,v}\to \bbS\hilbH_S$.
For every $\varrho \in O_{i,S,v}$ define $\Omega_{i,S,v}(\varrho) = s_v (\varphi_{i,*} (\varrho)) \in \hilbH_S$, where here and in what follows we often
identify $\varphi_{i,*} (\varrho)$ with its projection to $\sP (\fF)$.
Then $(\hilbH_S,\pi_S,\Omega_{i,S,v}(\varrho))$ is a cyclic representation of $\fF$. 
According to Prop.~\ref{prop:functorial-GNS} (\ref{ite:naturality-GNS}) there exists a
unique unitary $U_{i,v,\varrho} :\hilbH_\varrho \to \hilbH_S$ such that the  diagram 
\begin{equation}
  \label{dia:chart-unitaries-gns-hilbert-bundle}
  \begin{tikzcd}
  {\fA_{p(\varrho)}} \arrow[r,"{\varphi_{i,p(\varrho)}}"] \arrow[d,"\xi_\varrho"]& \fF \arrow[d,"{\xi_{\varphi_{i,*}(\varrho)}}"]\\
  \hilbH_\varrho \arrow[r,"{U_{i,v,\varrho}}"']& \hilbH_S 
 \end{tikzcd}
\end{equation}
commutes and $U_{i,v,\varrho} \Omega_\varrho = \Omega_{i,S,v}(\varrho)$. Here, we have abbreviated the projection
$p_{\sP (\fA)}: \sP(\fA) \to X$ by $p$.

\begin{thm}\label{thm:hilbertbundlegeneralcase}
Define local trivializations over the open sets $O_{i,S,v} \subset \sP(\fA)$ of \eqref{eq:oisv} by 
\begin{equation}
  \label{eq:local-trivialization-gns-hilbert-bundle}
  \chi_{i,S,v} : \sH|_{O_{i,S,v}} \to O_{i,S,v} \times \hilbH_S, \quad \Psi \mapsto (\varrho, U_{i,v,\varrho}\Psi),
  \quad \text{where } \varrho = p_{\sH} (\Psi)  \ .
\end{equation}
Endow $\sH$ with the coarsest topology such that each of the subsets $\sH|_{O_{i,S,v}}$ is open and all maps
$\chi_{i,S,v}$ are continuous. Then $p_{\sH} : \sH \to \sP (\fA)$ is a Hilbert bundle.
\end{thm}

\begin{proof}
We will show that the transition functions
\begin{equation}
  \label{eq:transition-functions-gns-hilbert-bundle}
  h_{i,S,v;j,R,w} : O_{i,S,v;j,R,w} := O_{i,S,v} \cap O_{j,R,w} \to \Unitary (\hilbH_S)_\textup{n} , \quad
  \varrho \mapsto  U_{i,v,\varrho}\circ U_{j,w,\varrho}^{-1}
\end{equation}
are continuous for all $i,j\in I$, sectors $S,R\in \ssec (\fF)$, and unit vectors $v,w\in \bbS \hilbH_S$ such that $O_{i,S,v;j,R,w} \neq \emptyset$, so that $p_{\sH} : \sH \to \sP (\fA)$ is a Hilbert bundle as claimed. 

To prove this, note that $O_{i,S,v;j,R,w} \neq \emptyset$ implies $O_{ij} \neq \emptyset$ and $\cH_R = \cH_S$ since for any $\varrho \in O_{i,S,v;j,R,w}$, we have $g_{ij}(p(\varrho))_*(\varphi_{j,*}(\varrho)) = \varphi_{i,*}(\varrho)$, hence $S = \overline{g_{ij}(p(\varrho))}(R)$.  Consider the following diagram. 
\[
  \begin{tikzcd}
    \fF \arrow[rr,"g_{ij}(p(\varrho))"] \arrow[ddd,"{\xi_{\varphi_{j,*}(\varrho)}}"']
    & & \fF \arrow[ddd,"{\xi_{\varphi_{i,*}(\varrho)}}"] \\
    &  {\fA_{p(\varrho)}} \arrow[ru,"{\varphi_{i}}"'] \arrow[d,"\xi_\varrho"]
       \arrow[lu,"{\varphi_{j}}"] \arrow[d,"\xi_\varrho"]& \\
    &  \hilbH_\varrho \arrow[ld,"{U_{j,w,\varrho}\!\!}"'] \arrow[rd,"{\!\!U_{i,v,\varrho}}"] \\
    \hilbH_S \arrow[rr,"h_{i,S,v;j,R,w}(\varrho)"']& &  \hilbH_S 
  \end{tikzcd}
\]
The diagram commutes by definition of the  $g_{ij}$ and $h_{i,S,v;j,R,w}$
and by construction of the unitaries $U_{j,w,\varrho}$ and $U_{i,v,\varrho}$.
Moreover, 
\[
g_{ij}(p(\varrho))_*r_R(\bbC\Omega_{j,R,w}(\varrho)) = g_{ij}(p(\varrho))_*(\varphi_{j,*}(\varrho)) = \varphi_{i,*}(\varrho) = r_S(\bbC\Omega_{i,S,v}(\varrho)).
\]
Thus, $(g_{ij}(p(\varrho)), \Omega_{j,R,w}(\varrho), \Omega_{i,S,v}(\varrho)) \in O$ and 
\[
\Omega_{i,S,v}(\varrho) = s_v(g_{ij}(p(\varrho))_* r_R(\bbC\Omega_{j,R,w}(\varrho))) =  \Phi(g_{ij}(p(\varrho)), \Omega_{j,R,w}(\varrho), \Omega_{i,S,v}(\varrho)),
\]
where $O$ and $\Phi$ are as in Prop.~\ref{prop:continuous-dependancy-unitary-automorphism}
for the case where $\fB =\fC=\fF$ and the representations $(\hilbH,\pi)$ and $(\tilde{\hilbH},\tilde{\pi})$
are given by $(\hilbH_S,\pi_S)$ and $(\hilbH_R,\pi_R)$, respectively.
Since $h_{i,S,v;j,R,w}(\Omega_{j,R,w}(\varrho)) = \Omega_{i,S,v}(\varrho)$, we know that
\begin{equation}
  \label{eq:representation-transition-functions-gns-hilbert-bundle}
h_{i,S,v;j,R,w}(\varrho) = U \big(g_{ij}(p(\varrho)), \Omega_{j,R,w}(\varrho),\Omega_{i,S,v}(\varrho)\big)
\end{equation}
where $U$ is as in Prop.~\ref{prop:continuous-dependancy-unitary-automorphism}. Since $g_{ij}(p(\varrho))$, $\Omega_{j,R,w}(\varrho)$, and $\Omega_{i,S,v}(\varrho)$ all depend continuously on $\varrho$ and $U$ is continuous, the map $h_{i,S,v;j,R,w}$ is continuous. \end{proof}

\begin{defn}
  We call the Hilbert bundle $p_\ssH : \ssH \to \sP(\fA)$ the \emph{GNS Hilbert bundle} associated to the 
  $C^ *$-algebra bundle $p: \frA \to X$. 
\end{defn}

\begin{rem}\label{rem:bundletrivial}
  By Kuiper's theorem \cite{Kuiper}, the unitary group of a separable infinite dimensional Hilbert space is contractible, hence any Hilbert bundle whose fibers are separable infinite-dimensional Hilbert spaces is a trivial bundle \cite{SchottenloherUnitaryStrongTopology}. Note that in a footnote to \cite[Theorem (3)]{Kuiper}, Kuiper indicates that
  the contractibility of the unitary group holds also for non-separable Hilbert spaces, therefore Hilbert bundles with such fibers have to be trivial as well.  
  Let us remind the reader at this point that if $\fA$ is a separable $C^*$-algebra and $\omega \in \sS(\fA)$,
  then the Hilbert space $\ssH_\omega$ of its GNS representation is separable. 
\end{rem}

\begin{rem}
  Given a norm-defined $C^*$-algebra bundle $p: \fA \to  X$ as before and  a continuous section
  $\omega : X \to \sP(\fA)$ of the associated pure state bundle, the pullback bundle
  $\omega^* \ssH \to X$ of the GNS Hilbert bundle along the section $\omega$ 
  is a
  Hilbert bundle over $X$. 
\end{rem}

\newcommand{\Nbundle}{\mathscr{N}}
\subsection{The subbundle of Gelfand ideals}
We now construct a locally trivial fiber bundle $p_\Nbundle:\Nbundle \rightarrow \sP(\fA)$ out of the Gelfand ideals of elements in the pure state bundle $\sP(\fA)$, retaining the notation of Section \S\ref{sec:hilbert-bundle} throughout. As a set, we define
\[
\Nbundle = \bigsqcup_{\varrho \in \sP(\fA)} \fN_\varrho.
\]
For each sector $S \in \ssec(\fF)$, let $\fN_S$ denote the Gelfand ideal of $\varrho_S$. Given a sector $S \in \ssec(\fF)$ and a unit vector $v \in \bbS\cH_S$, Corollary \ref{cor:continuous_Kadison_automorphism} yields a continuous map $\alpha_{S,v}:O_{S,v} \rightarrow \Aut(\fF)_\tn{n}$ such that
\[
\alpha_{S,v}(\psi)_*\psi = r_S(\bbC v)
\]
for all $\psi \in O_{S,v}$. We may also find a fixed automorphism $\beta_{S,v} \in \Aut(\fF)$ such that 
\[
\beta_{S,v,*}r_S(\bbC v) = \varrho_S.
\]
We will prove the following result:

\begin{thm}\label{thm:Nbundleresult}
For each $i \in I$, $S \in \ssec(\fF)$, and unit vector $v \in \bbS \cH_S$, define a local trivialization
\[
\chi^\Nbundle_{i,S,v}:\Nbundle|_{O_{i,S,v}} \rightarrow O_{i,S,v} \times \fN_S, \]
by
\[ A \mapsto (\varrho, \beta_{S,v}(\alpha_{S,v}(\varphi_{i,*}(\varrho))(\varphi_i(A)))), \quad \varrho = p_\Nbundle(A) \ .
\]
Give $\Nbundle$ the unique topology such that each $\Nbundle|_{O_{i,S,v}}$ is open and each local trivialization $\chi^{\Nbundle}_{i,S,v}:\Nbundle|_{O_{i,S,v}} \rightarrow O_{i,S,v} \times \fN_S$ is a homeomorphism. Then $p_\Nbundle :  \Nbundle \to \sP(\fA)$ is a fiber bundle.
\end{thm}

\begin{proof}
One can check that $\chi_{i,S,v}^\Nbundle$ is well-defined from the definitions of $\alpha_{S,v}(\varphi_{i,*}(\varrho))$ and $\beta_{S,v}$. 
Let $H_S$ be the subspace of $\Aut(\fF)_\tn{n}$ that leaves $\fN_S$ invariant. 
Given $i, j \in I$, $S, R \in \ssec(\fF)$, and $v, w \in \bbS \cH_S = \bbS\cH_R$ such that $O_{i,S,v;j,R,w} \neq \emptyset$, the transition function
\begin{align*}
h^\Nbundle_{i,S,v;j,R,w}:O_{i,S,v;j,R,w} &\rightarrow H_S,  \\
\varrho &\mapsto \beta_{S,v}\alpha_{S,v}(\varphi_{i,*}(\varrho)) g_{ij}(p(\varrho))\alpha_{R,w}(\varphi_{j,*}(\varrho))^{-1}\beta_{R,w}^{-1}
\end{align*}
is continuous since each term in the composition is a continuous function of $\varrho$.
 The following proposition shows that $H_S$ is a closed subgroup of $\Aut(\fF)_\tn{n}$ which acts effectively on $\fN_S$. In particular, this implies that $\Nbundle$ is a fiber bundle. \end{proof} 

\begin{prop}\label{prop:gelfandidealkeyprop}
Let $\fB$ be a $C^*$-algebra, let $\omega \in \sP(\fB)$, and let $\fN$ be the Gelfand ideal of $\omega$. Given $\alpha \in \Aut(\fB)$, the following are equivalent:
\begin{enumerate}[{\rm (i)}]
  \item\label{ite:N_alpha_invariant} $\alpha(\fN) \subset \fN$,
  \item\label{ite:omega_alpha_invariant} $\alpha_*\omega = \omega$,
  \item\label{ite:N_alpha_equality} $\alpha(\fN) = \fN$.
\end{enumerate}
If $H = \qty{\alpha \in \Aut(\fB): \alpha(\fN) \subset \fN}$, then $H$ is a closed subgroup of $\Aut(\fB)$ and $H$ acts effectively on $\fN$ by evaluation.
\end{prop}

\begin{proof}
To see (\ref{ite:N_alpha_invariant}) $\Rightarrow$ (\ref{ite:omega_alpha_invariant}), observe that $\alpha(\fN)$ is the Gelfand ideal of $\alpha_*\omega$. Maximality of Gelfand ideals \cite[Thm.~5.3.5]{MurphyCAOT} implies $\alpha_*\omega = \omega$.  That (\ref{ite:omega_alpha_invariant}) $\Rightarrow$ (\ref{ite:N_alpha_equality}) and (\ref{ite:N_alpha_equality}) $\Rightarrow$ (\ref{ite:N_alpha_invariant}) are trivial.  From these equivalences we see that $H$ is a subgroup of $\Aut(\fB)$. If $(\alpha_n)_{n \in \bbN}$ is a sequence in $H$ and $\alpha_n \rightarrow \alpha \in \Aut(\fB)$, then for any $A \in \fN$ we have $\alpha_n(A) \rightarrow \alpha(A)$, hence $\alpha(A) \in \fN$ since $\alpha_n(A) \in \fN$ and $\fN$ is closed. If $\alpha \in H$ and $\alpha(A) = A$ for all $A \in \fN$, then $\alpha(A) = A$ for all $A \in \ker \omega$ since $\ker \omega = \fN + \fN^*$ \cite[Thm.~5.3.4]{MurphyCAOT}. Finally, let $B \in \fA$ such that $\omega(B) = 1$. Since $\fA = \ker \omega \oplus \bbC B$, it remains to show that $\alpha(B) = B$. There exists $C \in \ker \omega$ and $\lambda \in \bbC$ such that $\alpha(B) = C + \lambda B$. Applying $\omega$ to both sides yields $\lambda = 1$. Since $\alpha(C) = C$, we see that $\alpha^n(B) = nC + B$ for all $n \geq 1$. But then
\[
\norm{B} = \norm{\alpha^n(B)} = \norm{nC + B} \geq n \norm{C} - \norm{B},
\]
which is true for all $n$ if and only if $C = 0$. Therefore $\alpha(B) = B$, so $\alpha = \id_\fB$. This proves that $H$ acts effectively.
\end{proof}


\begin{rem}It is easy to check that the topology on $\Nbundle$ in Theorem~\ref{thm:Nbundleresult} is the subspace topology inherited from $\sP(\fA) \times \fA$. In particular, $\Nbundle$ is a subspace of the pullback bundle $(p_{\sP(\fA)})^*\fA$ since $(p_{\sP(\fA)})^*\fA$ is 
a subspace of $\sP(\fA) \times \fA$ as well. The local trivializations $\chi^\Nbundle_{i,S,v}$ extend to local trivializations
\[((p_{\sP(\fA)})^*\fA)|_{O_{i,S,v}} \rightarrow O_{i,S,v} \times \fF,\] so $\Nbundle$ is a \emph{subbundle} of $(p_{\sP(\fA)})^*\fA$, for subbundles as defined in Appendix \ref{sec:fiber-bundles}.
\end{rem}

\subsection{The smooth case}\label{sec:smooth_case}
It is not immediately clear whether the notion of a topological $C^*$-algebra bundle
can be generalized to the smooth case. Let us explain the problem in  more detail and assume
that  $p:\fA \to X$ is a topological $C^*$-algebra bundle. Assume further that 
the total space $\fA$ and the base space $X$ both carry the structure of Banach manifolds
and that the projection $p$ is smooth. To endow $p:\fA \to X$ with a smooth fiber bundle structure
one needs to select an atlas of local trivializations $(\varphi_i, O_i)_{i\in I}$
such that the corresponding transition functions 
$g_{ij} :O_{ij} \to \Aut (\fF)$ to the automorphism group $\Aut (\fF)$ of the typical fiber $C^*$-algebra $\fF$
are smooth. The main question now is whether  $\Aut (\fF)$ carries the structure of a Lie group,
so that smoothness of the transition functions $g_{ij}$ makes sense. 
This problem seems not to have been studied in the literature before. We therefore provide an answer and show the following.   

\begin{prop}\label{prop:Aut_Lie_group}
  The automorphism group  $\Aut (\fF)$ of a $C^*$-algebra $\fF$ carries  the structure of a real Banach-Lie group with Lie algebra given by
  the space $\sD$ of symmetric bounded derivations on $\fF$, i.e.,
  the space of all derivations $\delta \in  \fB (\fF)$
  such that $(\delta (A))^* = \delta (A^*)$  for all $A\in \fF$.
\end{prop}

\begin{proof}
  The group $\GL (\fF)$ of invertible linear endomorphisms of $\fF$ is
  a topological group in the norm topology, and $\Aut (\fF)$ is a closed
  subgroup, hence a topological group as well. Denote by $\Aut_0 (\fF)$
  the connected component of  the unit element $\iota = \id_\fF$ in the
  automorphism group. According to \cite[Thm.~7]{KadisonRingroseDAOA},
  $\Aut_0 (\fF)$ is an open subgroup of  $\Aut (\fF)$ hence it 
  suffices to show that the connected component $\Aut_0 (\fF)$
  carries the structure of a Banach Lie group. 

  Denote by $\bbB_2 (\iota)$
  the ball of radius $2$ in the automorphism group around the unit.
  By \cite[Thm.~7]{KadisonRingroseDAOA},  $\bbB_2 (\iota)$ is contained in
  $\Aut_0 (\fF)$, and each element $\alpha\in \bbB_2 (\iota)$
  lies on a norm continuous one-parameter subgroup of automorphisms
  $\alpha_t$, $t\in \R$ such that $\alpha_1=\alpha$.
  By norm continuity of the one-parameter group, the one-parameter group can be represented
  in the form 
  \begin{equation}
  \label{eq:representation-one-parametr-group}  
    \alpha_t = \exp (t \delta) := \sum_{k=0}^\infty \frac{1}{k!} (t\delta)^k \ ,
  \end{equation}
  where $\delta:\fF \to \fF$ is a bounded operator given as
  the limit $\delta = \lim_{t\to 0} \frac{\alpha_t -\iota}{t}$ ; see e.g.~\cite[Th.~9.6.1]{HillePhillips}.
  Since $\alpha_t$ consists of automorphisms of $\fA$ one readily checks
  that $\delta$ is a symmetric bounded derivation.
  The space $\sD$ of symmetric bounded derivations on $\fF$ clearly forms a
  closed real Lie subalgebra of the Banach algebra $\fB (\fF)$ endowed with the commutator
  as Lie bracket.  Hence, $\sD$ is a Banach-Lie algebra.

  We observed that every element of $\bbB_2 (\iota)$ can be written in the form $\exp (\delta)$
  for some $\delta$ in the Lie algebra $\sD$. Conversely, $\exp (\delta) \in \Aut_0 (\fF)$
  for every $\delta \in \sD$ since  $t \mapsto \exp (t \delta)$ is a one-parameter group
  of automorphisms. This indicates that the exponential map restricted to $\sD$ can serve as
  a differentiable chart for $\Aut_0 (\fF)$.  
  To verify this recall that restricted to a sufficently small ball $\bbB_r (0)\subset \fB (\fF)$
  around the origin the exponential function $\exp : \fB (\fF) \to \GL (\fF)$ is a diffeomorphism onto its image
  by the inverse function theorem. After possibly shrinking $r$ we can achieve that
  $\widetilde{O} := \exp (\bbB_r(0))$ is an open neighborhood of $\iota$ in $\GL (\fF)$
  contained in the unit ball around the identity.
  The inverse of $\exp :\bbB_r(0) \to \widetilde{O}$ is then given by the logarithm series
  \[
    \Log : \widetilde{O} \to \bbB_r(0), \: \eta \mapsto \sum_{k=1}^\infty (-1)^{k-1} \frac{1}{k}(\eta -\iota) \ . 
  \]
  The restriction $\chi_\iota : = \exp|_{ \sD \cap \bbB_r(0)} : \sD \cap \bbB_r(0) \to \Aut_0(\fF)$ therefore is a homeomorphism onto
  $O := \Aut_0(\fF) \cap \widetilde{O}$. For each $\alpha \in \Aut_0(\fF)$ let 
  $\chi_\alpha : O\alpha \to W $ with $W := \sD \cap \bbB_r(0)$ denote the homeomorphism
  which maps $\gamma \in O\alpha$ to $ \Log ( \gamma \alpha^{-1} )$. We interpret $\chi_\alpha$
  as a chart of $\Aut_0(\fF)$ defined over $O\alpha$. For $\alpha,\beta \in \Aut_0(\fF)$ such that
  $O_{\alpha\beta} := O\alpha \cap O\beta \neq \emptyset$ the transition map 
  \[
    \chi_\alpha \circ \chi_\beta^{-1} : \chi_\beta(O_{\alpha\beta}) \to \chi_\alpha (O_{\alpha\beta}), \:
    \delta \mapsto \Log ( \exp (\delta)\beta \alpha^{-1})  
  \] 
  now is smooth, hence the family of charts $\big(\chi_\alpha\big)_{\alpha \in \Aut_0(\fF)}$
  is a smooth atlas which defines a manifold structure on $\Aut_0(\fF)$.
  Multiplication and inversion are smooth with respect to this manifold structure since they are
  on the ambient Lie group $\GL (\fF)$ whose manifold structure is also defined  
  by the exponential function. Hence $\Aut_0(\fF)$ carries a canonical Banach Lie group structure
  modeled on the Banach Lie algebra $\sD$.  By translation, the Banach Lie group structure can be extended in
  a unique way to the whole group $\Aut (\fF)$ and the claim is proved. 
\end{proof}

The observation that $\Aut (\fF)$ is a Banach Lie group now allows us to provide the following definition.

\begin{defn}\label{def:smooth-c*-algebraic-fiber-bundle}
  By a smooth $C^*$-algebra bundle one understands a topological $C^*$-algebra bundle 
  $p:\fA\to X$ 
  which in addition is endowed with
  \begin{enumerate}[(i)]
  \item
    Banach manifold structures on the total space $\fA$ and the base space $X$ such that the projection $p$ is smooth,
    and
  \item\label{ite:smooth-local-trivialization}
    a trivializing atlas $(\varphi_i,O_i)_{i\in I}$ such that the charts $\varphi_{i}:p^{-1}(O_i) \rightarrow O_i \times \fF$ and corresponding transition functions $g_{ij}:O_{ij} \to \Aut(\fF)$, $x \mapsto \varphi_{i,x} \circ \varphi_{j,x}^{-1}$ are smooth.  
  \end{enumerate}
\end{defn}

Given a smooth $C^*$-algebra bundle structure on $p:\fA \to X$, the corresponding dual fiber bundle  
$\projAStar: \fA^* \to X$ is a smooth Banach fiber bundle as well since the map
$\tau : \Aut(\fF) \to \GL(\fF^*)$, $\alpha \mapsto \alpha_*$ is of class $\sC^\infty$.
More precisely, if $(\varphi_i,O_i)_{i\in I}$ is a smooth trivializing atlas in the sense of
Def.~\ref{def:smooth-c*-algebraic-fiber-bundle} (\ref{ite:smooth-local-trivialization}), then
$(\varphi_{*,i},O_i)_{i\in I}$ is a trivializing atlas of $\projAStar: \fA^* \to X$
with smooth transition functions
\[ g_{*,ij} = \tau \circ g_{ij}: O_{ij}= O_i\cap O_j \to \GL(\fF^*) \ . \]
By smoothness of the transition functions,
the compositions
\[
  \varphi_{*,i} \circ \varphi_{*,j}^{-1} : O_{ij} \times \fF^* \to O_{ij} \times \fF^*,
  \: (x,\omega) \mapsto  \big(x, g_{*,ij}(x) ( \omega )\big ) 
\]
are smooth, hence the family $(\varphi_{*,i})_{i\in I}$ induces a smooth structure on $\fA^*$
and becomes a smooth trivializing atlas. Therefore, $\projAStar: \fA^* \to X$ inherits the structure
of a smooth fiber bundle from the smooth $C^*$-algebra bundle $p: \fA \to X$.

Next let us show that also the pure state bundle can be equipped with a smooth fiber bundle structure whenever
the fiber bundle $p: \fA \to X$ is smooth. 
As explained before, $\varphi_{*,i}$ restricts to a homeomorphism from the restricted
pure state bundle $\sP (\fA)_{|O_i}$ to the product $O_i\times  \sP (\fF)$. The latter space
carries a canonical structure of a Banach manifold. Recall that the Banach Lie group $ \Aut (\fF)$ leaves
$\sP (\fF)$ invariant and acts effectively on $\sP (\fF)$ by Lemma \ref{lem:efective-action-pure-state-space}.
By the above argument it is clear that $ \Aut (\fF)$ acts smoothly on $\fF^*$, but this does not immediately
entail that the action on $\sP (\fF)$ is smooth as well;
cf.~Remark \ref{rem:injectivity-tangent-map-not-entailing-submanifold-infinite-dimensional-case}. 
The following result resolves this problem. 

\begin{lem}\label{lem:smoothness-action-automorphism-group-pure-state-space}
  The action
  \[
     \Upsilon: \Aut (\fF) \times \sP (\fF) \to \sP (\fF), \: (\alpha , \omega) \mapsto \Upsilon (\alpha ,\omega) = \alpha_* \omega 
  \]
  is smooth.   
\end{lem}

\begin{proof}
  For fixed $\alpha$, consider a sector $S$ of $\fF$ and choose an irreducible representation
  $(\hilbH,\pi)$ of $\fF$ such that $S$ coincides with the space $\sP_\pi (\fF)$
  of pure $\pi$-normal states. Note that the representation $(\hilbH,\pi\alpha)$
  then is ireducible as well and that $\overline{S} = \sP_{\pi\alpha} (\fF)$. 
  Let $r : \bbP\hilbH \to \sP_\pi (\fF)$ be the uniform isomorphism from
  Cor.~\ref{cor:PH_superselection_metric_equivalence} and $\Psi \in \bbS\hilbH$.
  Consider the smooth chart $\tau_\Psi: \bbB_1 (\C\Psi) \to C_\Psi$ of the projective Hilbert space $\bbP\hilbH$
  around the ray $\C\Psi$ as defined in Theorem \ref{thm:holomorphic-atlas-projective-space}.
  For every $v \in C_\Psi$, the map
  $\gamma_v: \R \to \sP (\fF)$, $t \mapsto r \tau_\Psi^{-1} (t v)$ then is a smooth path in
  $\sP_\pi (\fF)$, and $\dot{\gamma}_v (0)$ is a tangent vector of $\sP (\fF)$  at $\varrho = r (\C\Psi)$.
  Let $\delta$ be an element of the Lie algebra $\sD$ of $\Aut (\fF)$. Now compute
  \begin{equation}
  \begin{split}
  \label{eq:derivative-action-automorphism-group-pure-states}
    T_{\alpha,\varrho} \Upsilon \, & \big(\delta,\dot{\gamma}_v (0)\big)  =
    \left. \frac{d}{ds} \right|_{s=0} \Upsilon \big( \alpha \exp(s\delta),\varrho\big) +
    \left. \frac{d}{dt} \right|_{t=0} \Upsilon \big( \alpha , \gamma_v (t)\big) = \\
    & = \left. \frac{d}{ds} \right|_{s=0} \varrho \circ \exp (-s\delta) \circ \alpha^{-1} +
        \left. \frac{d}{dt} \right|_{t=0} \langle  \tau_\Psi^{-1} (t v), \pi\alpha^{-1} (\,\cdot \,)  \tau_\Psi^{-1} (t v) \rangle   = \\
    & = - \alpha_* ( \delta^* \varrho ) + \langle \Psi, \pi\alpha (\,\cdot \,) v\rangle + \langle v, \pi\alpha (\,\cdot \,) \Psi \rangle \ .
  \end{split}  
  \end{equation}
  The right hand side is obvously jointly continuous in $\alpha$, $\Psi$, $\delta$ and $v$,
  so $\Upsilon$ is of class $\sC^1$. Since the right hand side of
  \eqref{eq:derivative-action-automorphism-group-pure-states} is in addition linear in $\Psi$, $\delta$ and $v$
  and continuously differentiable in $\alpha$ by the argument just provided, one concludes that $\Upsilon$ is even of class $\sC^\infty$
  and the claim is proved. 
\end{proof}
Note that the restricted maps
$\varphi_{*,i,\sP(\fA)} := \varphi_{*,i}|_{\sP (\fA)_{O_i}}: \sP (\fA)_{|O_i} \rightarrow O_i \times \sP(\fF)$ with $i\in I$
form a trivializing atlas of the pure state bundle $\projAPure: \sP(\fA) \to X$. Since each of the spaces
$O_i \times \sP(\fF)$ carries the natural product manifold structure,
Lemma \ref{lem:smoothness-action-automorphism-group-pure-state-space} entails that the compositions
\[
  \varphi_{*,i,\sP(\fA)} \circ \varphi_{*,j,\sP(\fA)}^{-1} : O_{ij} \times \sP (\fF)  \to O_{ij} \times \sP(\fF),
  \: (x,\omega) \mapsto  \big(x, \Upsilon (g_{ij}(x), \omega )\big ) 
\]
are smooth. Arguing as before, the pure state bundle $\projAPure: \sP(\fA) \to X$ therefore inherits a smooth fiber
bundle structure from the one on the bundle $p: \fA\to X$.

Finally we consider the GNS Hilbert bundle  $\projGNS : \ssH \to \sP (\fA) $ in the smooth case.
Given an atlas $(\varphi_i,O_i)_{i\in I}$ of smooth trivializations of $p: \fA\to X$ consider the local
trivializations $\chi_{i,S,v}:\sH|_{O_{i,S,v}}\to O_{i,S,v} \times \hilbH_S$ defined by \eqref{eq:local-trivialization-gns-hilbert-bundle}, where $S \in \ssec (\fF)$ is a sector and $v\in \bbS\hilbH_S$. The
sets $O_{i,S,v}$ are open in the pure state bundle $\sP(\fA)$ and inherit a manifold  structure
from the ambient  $\sP(\fA)$. Therefore, each of the cartesian products $O_{i,S,v} \times \hilbH_S$
carries the product manifold structure. The Hilbert bundle $\projGNS : \ssH \to \sP (\fA) $ can now
be endowed with a compatible smooth structure if we can yet show that the transition functions 
$h_{i,S,v;j,R,w} : O_{i,S,v}\cap O_{j,R,w}  \to \Unitary (\hilbH_S)$ given by
Eq.~\eqref{eq:transition-functions-gns-hilbert-bundle} are smooth. By
equation \eqref{eq:representation-transition-functions-gns-hilbert-bundle}, the transition function
$h_{i,S,v;j,R,w}$ is smooth whenever the maps $\Omega_{i,S,v}: O_{i,S,v}\to\hilbH_s$,
$\varrho \mapsto s_v (\varphi_{i,*} (\varrho))$ 
and $U : O \to \Unitary (\hilbH_S)$ from Prop.~\ref{prop:continuous-dependancy-unitary-automorphism}
applied to the case where 
$\fB =\fC=\fF$, $(\hilbH,\pi)=(\hilbH_S,\pi_S)$ and $(\tilde{\hilbH},\tilde{\pi})=(\hilbH_R,\pi_R)$ are smooth.
Smoothness of $\Omega_{i,S,v}$ is clear since by Corollary \ref{cor:purestate_mapsto_SH}
the section $s_v$ is smooth and since $\varphi_{i,*}$ is a smooth local trivialization of the pure state
bundle. Smoothness of $U$ is shown in the following, where we silently use notation from
Prop.~\ref{prop:continuous-dependancy-unitary-automorphism}.

\begin{prop}\label{prop:smooth-dependancy-unitary-automorphism}
  Let $\fB$ be a $C^*$-algebra, $(\hilbH,\pi)$ and $(\hilbH,\tilde{\pi})$ two
  irreducible representations, and let $O$  be the open and non-empty set 
  \[
    O = \qty{(\alpha, \Omega, \Psi) \in \Aut(\fB)_\textup{n} \times \bbS\hilbH \times \bbS\tilde{\hilbH}
      : \alpha_*r(\bbC \Omega) \in \bbB_2(\tilde r(\bbC\Psi))}.
  \]
  Then the following holds true:
  \begin{enumerate}[{\rm (i)}]
  \item\label{ite:O_Phi_smooth}
    The map $\Phi:O \rightarrow \bbS\tilde{\hilbH}$,
    $(\alpha, \Omega, \Psi) \mapsto \tilde s_{\Psi}\alpha_*r(\bbC\Omega)$ is smooth.
  \item\label{ite:O_U_smooth}
     The unique map $U:O \rightarrow \Unitary(\hilbH)$ making the diagram 
    \[
    \begin{tikzcd}[column sep = 1.5cm]
    \fB \arrow[r,"\alpha"] \arrow[d,"\xi_{\Omega}"']& \fB \arrow[d,"\tilde\xi_{\Phi(\alpha,\Omega,\Psi)}"]\\
    \hilbH \arrow[r,"{U{(\alpha,\Omega,\Psi)}}"']& \hilbH
   \end{tikzcd}
    \]
    commute for all $(\alpha, \Omega,\Psi) \in O$ is smooth.
   \end{enumerate} 
 \end{prop}
 \begin{proof}
   (\ref{ite:O_Phi_smooth})
   First observe that the map
   \[ f:\Aut(\fB)_\textup{n}\times\bbS\hilbH \to  \bbS\hilbH, \: (\alpha,\Omega) \mapsto \tilde r^{-1} \alpha_* r(\bbC\Omega)  \] 
   is smooth by Lemma \ref{lem:smoothness-action-automorphism-group-pure-state-space} and since $r$ and $\tilde r$ are holomorphic.
   Next recall from Section \ref{sec:background-projective-hilbert-spaces} that for a given unit vector $\Psi \in \bbS\hilbH$
   and the corresponding pure state $\psi = \tilde r (\C\Psi)$ the section $\tilde s_\Psi: \bbB_1 (\psi) \to \bbS\hilbH$
   maps every pure state of the form
   $\omega = \tilde r (\C\Omega)$ with $\Omega \in \bbS\hilbH\setminus \bbS\hilbH \cap C_\Psi$  to the unit vector 
   \[
     \tilde s_\Psi (\omega) =  \sigma_\Psi^{-1} \big( \tau_\Psi(\C\Omega), 1\big ) =
     \frac{\Omega}{\langle\Psi,\Omega\rangle \cdot \sqrt{1+\left\|  \frac{\Omega}{\langle\Psi,\Omega\rangle} -\Psi  \right\|^2}} =
     \frac{\langle\Omega,\Psi\rangle}{|\langle\Omega,\Psi\rangle|}\Omega \ .
   \]
   The right hand side is obviously a smooth map on the open set
   \[ W= \{ (\Psi,\Omega) \in \bbS\hilbH \times \bbS\hilbH : \: \Omega \notin C_\Psi\} \ . \]
   Since the projection $p: \bbS\hilbH \to \bbP\hilbH$ is a surjective submersion, this implies that 
   the map
   \[
      g : \tilde W \to  \bbS\hilbH, \:  (\Psi,\Omega) \mapsto \tilde s_\Psi (\tilde r (\C \Omega) )
   \]
   is smooth, where $\tilde W\subset \bbS \hilbH \times \bbP \hilbH$ is the open set
   \[
     \tilde W = (\id_{\bbS\hilbH}\times p) (W)=  \{(\Psi, \scrl) \in \bbS \hilbH \times \bbP \hilbH : \:  \scrl \in \bbB_1 (\C\Psi) \} \ .
   \]
   Now observe that the image of an element $(\alpha, \Omega, \Psi) \in O $ under the map $\Phi$ can be rewritten as 
   \begin{equation*}
     \begin{split}
       \Phi (\alpha, \Omega, \Psi) & =  \tilde s_{\Psi}\alpha_* r(\bbC\Omega) =
       \sigma_\Psi^{-1} \big( \tau_\Psi ( \tilde r^{-1} \alpha_* r(\bbC\Omega)) , 1\big ) =
      g \big( \Psi, f(\alpha,\Omega) \big) \ .
     \end{split}
   \end{equation*}
   By smoothness of $f$ and $g$, $\Phi$ then has to be smooth as well.

   (\ref{ite:O_U_smooth}) Fix a vector $v \in \hilbH$.
   Given $\Omega \in \bbS\hilbH $ choose $B\in \fB$ such that $v = \pi (B)\Omega$
   and put $O_\Omega = \{ (\alpha,\Psi)\in \Aut (\fB)_{\textup{n}}\times \bbS\hilbH : \: (\alpha,\Omega,\Psi) \in O \}$.
   The map
   \[
     U( \cdot, \Omega, \cdot)v : O_\Omega \to  \hilbH , \: (\alpha,\Psi) \mapsto  U( \alpha, \Omega, \Psi)v
   \]
   then is smooth by  (\ref{ite:O_Phi_smooth}) and since by
   Eq.~\eqref{eq:properties-unitary-intertwiner}
   \[
     U( \alpha, \Omega, \Psi)v=  U( \alpha, \Omega, \Psi)\pi (B)\Omega = \tilde\pi (\alpha (B)) U( \alpha, \Omega, \Psi) \Omega =
     \tilde\pi (\alpha (B)) \Phi (\alpha, \Omega, \Psi) \ .
   \]
   Next fix $(\alpha,\Psi)\in \Aut (\fB)_{\textup{n}}\times \bbS\hilbH$ such that
   $(\alpha,\Omega,\Psi)\in O$ and let $\gamma: \R \to \bbS\hilbH$ be a smooth path such that
   $\gamma (0)=\Omega$ and $(\alpha,\gamma (t),\Psi)\in O$ for all $t\in \R$. Then compute using Eq.~\eqref{eq:properties-unitary-intertwiner} again:
   \begin{equation*}
     \begin{split}
       \left. \frac{d}{dt} \right|_{t=0} &  U( \alpha, \gamma (t), \Psi)v
       = \lim_{t\to 0} \frac{U( \alpha, \gamma (t), \Psi)v - U( \alpha, \Omega , \Psi)v}{t} = \\
      = & \lim_{t\to 0} \frac{\tilde\pi (\alpha(B)) U( \alpha, \gamma (t), \Psi)\Omega -
           \tilde \pi(\alpha(B)) U( \alpha, \Omega , \Psi)\Omega}{t} = \\
       = & \lim_{t\to 0}\tilde\pi(\alpha(B)) \frac{U(\alpha,\gamma (t),\Psi)(\Omega -\gamma (t))
         + \big(\Phi(\alpha,\gamma (t),\Psi)-\Phi(\alpha,\Omega,\Psi)\big)}{t} = \\
       = &  \tilde \pi(\alpha(B)) \left(  \frac{\partial \Phi(\alpha,\Omega,\Psi)}{\partial \Omega}
           - U( \alpha, \Omega , \Psi) \right)  \cdot \gamma^\prime (0)
     \end{split}
   \end{equation*}
   The right hand side is continuous in $(\alpha,\Omega,\Psi) \in O$, hence
   the map $U( \, - \, )v : O \to  \hilbH$ is continuously differentiable for all
   $v\in \hilbH$. Therefore $U$ is $\sC^1$.  By induction one concludes that $U$ is
   $\sC^\infty$.
 \end{proof}

 \noindent
 In summary, we obtain the following result.

 \begin{thm}
   Let $p:\fA \to X$ be a smooth $C^*$-algebra bundle. Then the associated dual bundle
   $\projAStar : \fA^* \to X$, the pure state bundle $\projAPure: \sP(\fA) \to X$
   and the GNS Hilbert bundle
   $\projGNS:\sH\to X$ all carry natural smooth structures compatible with their underlying
   fiber bundle structures. In particular this means for the GNS Hilbert bundle that there exists an
   atlas of smooth local trivializations whose transition functions have values in the unitary group
   of the (local) typical fiber Hilbert space. 
 \end{thm}


\section{Examples of quantum systems}\label{sec:examples}

We give two simple physical examples to illustrate how the fiberwise GNS construction may be used and how one might obtain a norm-continuous family of pure states. The first is a finite-dimensional example coming from a spin-$\frac{1}{2}$ particle in a magnetic field. The second is an infinite-dimensional example of a non-interacting system obtained by copying the finite-dimensional example at each point of a countable lattice. 

\subsection{Particle in a magnetic field}
\label{subsec:0d_example}

In this section, we consider the trivial $C^*$-algebra bundle 
\[\underline{M}_2(\C):=S^2 \times M_2(\bbC) \rightarrow S^2.\] 
The continuous family of states will be the ground states of a family of self-adjoint operators over $S^2$. To this end, we  
 define a smooth section $H\colon S^2 \rightarrow M_2(\bbC)$ by the Hamiltonians
\[
H_\r = \r \cdot \bm{\sigma} = \mqty(z & x - iy \\ x + iy & -z),
\]
where $\r = (x, y, z) \in S^2$ and $\bm{\sigma} = (\sigma_x, \sigma_y, \sigma_z)$ are the Pauli matrices. Physically, the Hamiltonian $H_\r$ corresponds to the energy of a spin-$\frac{1}{2}$ particle in a magnetic field pointing in the direction $\r$. It is easily verified that the spectrum of this matrix is $\sigma(H_\r) = \qty{-1, 1}$ and that
\begin{equation}\label{eq:Psi_north_pole_chart}
\Psi_\r = \frac{1}{\sqrt{2 + 2z}}\mqty(-x + iy\\ z+1), \,\, z \neq -1 \qqtext{and} \Psi_\r = \mqty(1\\0), \, \, z = -1
\end{equation}
gives the ground state unit vector of $H_\r$. The map $\Psi:S^2 \rightarrow \bbC^2$ is continuous everywhere except when $\r = (0,0,-1)$. It is helpful to note that we can also redefine the ground state unit vector $\Psi$ to be continuous everywhere except when $\r = (0,0,1)$:
\begin{equation}\label{eq:Psi_south_pole_chart}
\Psi_\r = \frac{1}{\sqrt{2-2z}}\mqty(z-1\\x + iy), \, \, z \neq 1 \qqtext{and} \Psi_\r = \mqty(0\\1), \, \, z = 1.
\end{equation}
In either case, the state $\psi_\r(A) = \ev{\Psi_\r, A\Psi_\r}$ may be expressed as
\[
\psi_\r = \tau_0 - \r \cdot \bm{\tau},
\]
where $\tau_0, \tau_x, \tau_y, \tau_z$ is the dual basis of $I, \sigma_x, \sigma_y, \sigma_z$ and $\bm{\tau} = (\tau_x, \tau_y, \tau_z)$. The state $\psi_\r$ is pure since $M_2(\bbC)$ acts irreducibly on $\bbC^2$. Thus, $\psi:S^2 \rightarrow M_2(\bbC)^*$ defines a manifestly smooth section of pure states of the dual bundle, for which we can perform the fiberwise GNS construction of \S\ref{sec:fiberwiseGNS}. In this case, since our $C^\ast$-algebra bundle is finite dimensional, the Hilbert bundle $\cH$ that we obtain is simply the quotient of $\underline{M}_2(\C)$ by $\mathfrak{N} = \{M \mid \psi(M^\ast M) = 0\}$. 

In Remark~\ref{rem:bundletrivial}, we noted that, in the infinite dimensional case, the Hilbert bundle obtained via the fiberwise GNS construction is necessarily trivial. We show that in this finite dimensional example $\cH$ is nontrivial.
First, consider the subbundle $\cE\subset S^2\times\C^2:=\underline{\C}^2$ defined as the kernel of the bundle map $I + H$.  
The fiber above $\r\in S^2$ is the $1$-dimensional subspace of ground states of $H_\r$.
The map
\[S^2\to\cE, \quad \r\mapsto\begin{pmatrix} -x + iy\\ z+1\end{pmatrix}\] is a section of 
$\cE$. This section has a unique zero at $z = -1$, and one can check that this intersection with the zero section
is transverse. It follows that the first Chern class $c_1(\cE)$ is Poincar\'e dual to a point in $S^2$, 
hence $\int_{S^2} c_1(\cE) = 1$. Therefore, $c_1(\cE)$ is non-trivial. 


Now let $\cF$
be the orthonognal complement of $\cE$ in the trivial bundle $\underline{\C}^2$ so that the fiber $\cF_{\r}$ consists of vectors orthogonal to $\Psi_\r$. We then have 
\[\underline{M}_2(\C)\cong \underline{\C}^2\otimes(\underline{\C}^2)^\ast \cong\underline{\C}^2\otimes(\cE^\ast
\oplus\cF^\ast) \cong (\underline{\C}^2\otimes\cE^\ast)\oplus (\underline{\C}^2\otimes\cF^\ast).\] 
Since \[\mathfrak{N}_\r = \{M \mid \langle M\Psi_\r,M\Psi_\r\rangle = 0\} = \{M \mid M\cE_\r = 0\},\] it follows that 
$\mathfrak{N} \cong \underline{\C}^2\otimes\cF^\ast$. Thus, 
\[\cH=\underline{M}_2(\C)/\mathfrak{N}\cong\underline{\C}^2\otimes\cE^\ast,\] implying that $c_1(\cH) = 2c_1(\cE^\ast)\cong -2c_1(\cE)$ is a nontrivial cohomology class.

\begin{rem}
More generally, for any $n$-dimensional Hilbert bundle $\cC$ over $X$, we obtain a $C^*$-algebra bundle as the endomorphism bundle $\fB(\cC)$ of $\cC$. Given a continuous section $H \colon X \to \fB(\cC)$ such that $H_x$ is a self-adjoint operator whose smallest eigenvalue has a one-dimensional eigenspace $\cE_x$ (i.e., $H_x$ is gapped), the $\cE_x$ assemble into a line bundle $\cE$ over $X$. As above, the assignment $A \in \fB(\cC_x) \mapsto \langle \psi_x, A\psi_x\rangle$ for any unit vector $\psi_x \in\cE_x$ gives a continuous section of ground states $\psi \colon X \to \fB(\cC)^*$. The fiberwise GNS construction for $\psi$ gives a Hilbert bundle $\cH$ whose first Chern class  satisfies $c_1(\cH) = c_1(\cC) -nc_1(\cE)$. This is shown as in our example above by establishing that $\cH \cong \cC \otimes \cE^*$.
\end{rem}

\subsection{Non-interacting lattice system}\label{subsec:1d_system}

Perhaps the simplest way of obtaining a norm-continuous family of pure states of an infinite-dimensional $C^*$-algebra $\frA$ is to start with a nonzero irreducible representation $(\cH, \pi)$ and a continuous map $\Omega:X \rightarrow \bbS\cH$ and to lift $\Omega$ to $\sP(\frA)$ using Lemma \ref{lem:vector_state}. Below we give an example of a different flavor.

We copy the finite-dimensional model above to each vertex of a lattice $\bbZ^d$ for some positive $d \in \bbN$. The $C^*$-algebra of this system is the quasi-local algebra
\[
\fA = \overline{\bigcup_{\Lambda \in \wp_f(\bbZ^d)} \fA(\Lambda)} \qqtext{for} \fA(\Lambda) = \bigotimes_{v \in \Lambda} M_2(\bbC),
\]
where $\wp_f(\bbZ^d)$ is the set of all finite subsets of $\bbZ^d$. We also define the local algebra as the dense $*$-algebra $\fA_\tn{loc} = \bigcup_{\Lambda \in \wp_f(\bbZ^d)} \fA(\Lambda)$. For more on quasi-local algebras obtained from lattices see, for example, \cite[Ch.\ 3]{NaaijkensQSSIL}. 

Our parameter space is $X = \prod_{v \in \bbZ^d} S^2$. We will consider both the product and box topologies on $X$. For each $\r = (\r_v)_{v \in \bbZ^d} \in X$, we will construct the ground state $\omega_\r \in \sP(\fA)$ of the interaction\footnote{Although it is standard terminology, ``interaction'' is a bit of a misnomer in this case since the lattice sites are non-interacting.} $\Phi_\r:\wp_f(\bbZ^d) \rightarrow \frA$, defined as 
\[
\Phi_\r(\Lambda) = \begin{cases} H_{\r_v} & \Lambda = \qty{v} \tn{ for some } v \in \bbZ^d, \\ 0 & \tn{otherwise.} \end{cases}
\]
Consider the map $\r \mapsto \Phi_\r$ into the space of bounded, finite-range interactions with the norm
\[
\norm{\Phi} \defeq \sup_{v \in \bbZ^d} \sum_{v \in \Lambda \in \wp_f(\bbZ^d)} \norm{\Phi(\Lambda)}.
\]
Given $\r, \r' \in X$, we have
\[
\norm{\Phi_\r - \Phi_{\r'}} = \sup_{v \in \bbZ^d} \norm{H_{\r_v} - H_{\r_v'}} = \sup_{v \in \bbZ^d} \norm{(\r_v - \r_v') \cdot \bm{\sigma}} = \sup_{v \in \bbZ^d} \norm{\r_v - \r_v'},
\]
This implies that $\r \mapsto \Phi_\r$ is continuous when $X$ is given the box topology, but discontinuous when $X$ is given the product topology. Interestingly, it also shows that composing with the diagonal map $S^2 \rightarrow X$ yields a continuous function into the space of interactions, even though the diagonal map is discontinuous with respect to the box topology.

We construct the ground state of $\Phi_\r$ by an application of the following theorem. The analogous statement for not necessarily pure states is a rephrasal of Theorem 2 in \cite{TakedaInductiveLimits}. The proof of Proposition 5 in the same paper shows that the result holds for pure states.

\begin{thm}\label{thm:pure_state_inductive}
Let $(\fA_i)_{i \in I}$ be an inductive system of $C^*$-algebras indexed by the directed set $I$ and let $\fA$ be the inductive limit. If $\omega_i \in \sP(\fA_i)$ for all $i \in I$ and $\omega_i = \omega_j|_{\fA_i}$ whenever $i \leq j$, then there exists a unique pure state $\omega \in \sP(\fA)$ such that $\omega_i = \omega|_{\fA_i}$ for all $i \in I$.
\end{thm}

Thus, it suffices to define a compatible system of pure states $\omega_{\r, \Lambda} \in \sP(\fA(\Lambda))$. Given $\Lambda \in \wp_f(\bbZ^d)$ we define $\omega_{\r, \Lambda}$ to be the vector state represented by 
\[
\Omega_{\r,\Lambda} = \underset{v\in \Lambda}{\otimes} \Psi_{\r_v} \in \bigotimes_{v \in \Lambda} \bbC^2,
\]
that is, $\omega_{\r, \Lambda} (A) = \langle \Omega_{\r,\Lambda} ,A\Omega_{\r,\Lambda} \rangle$ for all $A\in \fA(\Lambda)$.
The state $\omega_{\r, \Lambda}$ is pure since $\fA(\Lambda)$ acts irreducibly on $\bigotimes_{v \in \Lambda} \bbC^2$. Observe that
\begin{equation}\label{eq:omega_simple_tensor}
\omega_{\r, \Lambda}\qty(\underset{v \in \Lambda}{\otimes} A_v) = \prod_{v \in \Lambda} \psi_{\r_v}(A_v)
\end{equation}
for all simple tensors $\otimes_{v \in \Lambda} A_v$. In particular, if $\Lambda_1 \subset \Lambda_2$, then the above formula implies that $\omega_{\r, \Lambda_1} = \omega_{\r, \Lambda_2}|_{\fA(\Lambda_1)}$ since $\psi_{\r_v}(I) = 1$. Theorem \ref{thm:pure_state_inductive} now yields a unique pure state $\omega_\r \in \sP(\fA)$ that restricts to $\omega_{\r,\Lambda}$ on $\fA(\Lambda)$. We consider the continuity properties of $\omega:X \rightarrow \sP(\frA)$.

We show that $\omega$ is norm-continuous when $X$ is given the box topology. Observe that for any $\r, \r' \in X$ and $A \in \fA(\Lambda)$,
\begin{equation}\label{eq:norm_continuity_lattice_states}
\begin{aligned}
\norm{\omega_\r(A) - \omega_{\r'}(A)} &= \norm{\omega_{\r, \Lambda}(A) - \omega_{\r', \Lambda}(A)} \leq \norm{A} \norm{\omega_{\r, \Lambda} - \omega_{\r', \Lambda}}\\
&\leq 2 \norm{A} \norm{\Omega_{\r, \Lambda} - \Omega_{\r', \Lambda}} \leq 2 \norm{A} \sum_{v \in \Lambda} \norm{\Psi_{\r_v} - \Psi_{\r_v'}},
\end{aligned}
\end{equation}
where in the last step we have used multilinearity of the tensor product, the triangle inequality, and the definition of the norm on the tensor product of Hilbert spaces. Given $\r \in X$ and $\varepsilon > 0$, we may define $\Psi_{\r_v}$ by \eqref{eq:Psi_north_pole_chart} or \eqref{eq:Psi_south_pole_chart} depending on whether $\r_v$ is the north pole or the south pole, or neither, in which case the choice does not matter. Then we may choose a neighborhood $U$ of $\r$ such that
\[
\sum_{v \in \bbZ^d} \norm{\Psi_{\r_v} - \Psi_{\r_v'}} < \frac{\varepsilon}{2}.
\]
for all $\r' \in U$. Again, we see that it is crucial to use the box topology on $X$. With the above inequality, \eqref{eq:norm_continuity_lattice_states} implies that $\norm{(\omega_\r - \omega_{\r'})(A)} < \varepsilon \norm{A}$ for all $A \in \fA_\tn{loc}$, so $\norm{\omega_\r - \omega_{\r'}} \leq \varepsilon$ by the density of $\fA_\tn{loc}$ in $\fA$. This proves norm-continuity of $\omega:X \rightarrow \sP(\frA)$.

When $X$ is given the product topology, we can show that $\r \mapsto \omega_\r$ is weakly$^*$ continuous. It follows from equation \eqref{eq:omega_simple_tensor} and continuity of $\psi:S^2 \rightarrow \sP(M_2(\bbC))$ that $\r \mapsto \omega_\r(A)$ is continuous with respect to the product topology for all simple tensors $A = \otimes_{v \in \Lambda} A_v \in \frA(\Lambda)$. Since any element of $\frA_\tn{loc}$ can be written as a linear combination of such simple tensors, we see that $\r \mapsto \omega_\r(A)$ is continuous for all $A \in \frA_\tn{loc}$. Continuity of $\r \mapsto \omega_\r(A)$ for all $A \in \frA$ can then be checked using density of $\frA_\tn{loc}$ in $\frA$.

Norm-continuity fails, however, when $X$ is given the product topology. This will follow from the theorem below. This result appears in a much more general form as Corollary 2.6.11 in \cite{BratteliRobinsonOAQSMII}. We supply a different proof of the specific part we need under milder assumptions.

\begin{thm}\label{thm:quasilocal_superselection}
  Let $\frA$ be a $C^*$-algebra and let $(\frA_\alpha)_{\alpha \in I}$ be a family
  of $C^*$-subalgebras of $\frA$ such that $\frA = \overline{\bigcup_{\alpha \in I} \frA_\alpha}$. Suppose there exists a symmetric relation $\perp$ on $I$ such that $\alpha \perp \beta$ implies
\[
[\frA_\alpha, \frA_\beta] = \qty{0}.
\]
If $\omega_1, \omega_2 \in \sP(\frA)$ are in the same superselection sector, then for any $\varepsilon > 0$ there exists $\alpha \in I$ such that
\[
\abs{\omega_1(A) - \omega_2(A)} < \varepsilon \norm{A}
\]
for any $\beta \perp \alpha$ and $A \in \frA_\beta$.
\end{thm}

Typically, $I$ is a causal index set  and the relation $\perp$ is defined by $\alpha \perp \beta$ if and only if $\alpha \subset \beta^\perp$ (see Remark \ref{rem:causal_structures}).

\begin{proof}
Since $\omega_1, \omega_2 \in \sP(\frA)$ are in the same superselection sector, they extend to pure states $\tilde \omega_1$, $\tilde \omega_2$ on the unitization $\frA_1$ that are in the same superselection sector. Theorem \ref{thm:superselection_sector_equivalences} yields a unitary $U \in \frA_1$ such that $\tilde \omega_2 = U \cdot \tilde \omega_1$. We may write $U = \lambda + B$ where $\lambda \in \bbC$ and $B \in \frA$.  There exists $\alpha \in I$ and $C \in \frA_\alpha$ such that $\norm{B - C} < \varepsilon/2$. Set $V = \lambda + C$. Given $\beta \in I$ such that $\beta \perp \alpha$ and $A \in \frA_\beta$, we have
\begin{align*}
\abs{\omega_1(A) - \omega_2(A)} &= \abs{\tilde \omega_1(A) - \tilde \omega_1(U^*AU)}\\
&\leq \abs{\tilde \omega_1(U^*UA) - \tilde\omega_1(U^*AV)} + \abs{\tilde\omega_1(U^*AV) - \tilde\omega_1(U^*AU)}\\
&\leq \abs{\tilde\omega_1(U^*(U-V)A)} + \abs{\tilde\omega_1(U^*A(V - U))}\\
&\leq 2\norm{B - C}\norm{A} < \varepsilon \norm{A},
\end{align*}
as desired. Note that we used the fact that $[V,A] = [\lambda + C, A]= 0$ since $\alpha \perp \beta$.
\end{proof}

In our case, the index set is $I = \wp_f(\bbZ^d)$ and $\Lambda_1 \perp \Lambda_2$ if and only if $\Lambda_1 \cap \Lambda_2 = \varnothing$. Intuitively, the theorem states that if two pure states on a quasi-local $C^*$-algebra are in the same superselection sector, then they are equal to each other ``at infinity.'' To prove that our map $\omega$ is not norm-continuous when $X$ has the product topology, we prove that the composition of $\omega$ with the diagonal map $\Delta: S^2 \rightarrow X$ is not norm-continuous. In fact, for distinct $\r, \s \in S^2$, we 
prove that $\omega_{\Delta(\r)}$ and $\omega_{\Delta(\s)}$ are in different superselection sectors. Since   $S^2$ is path connected, the diagonal map is continuous for the product topology, and the superselection sectors are the path components of $\sP(\frA)$ with the norm topology by Theorem \ref{thm:superselection_sector_equivalences}, this implies that $\omega$ is not norm-continuous when $X$ has the product topology. 


Given any $\Lambda \in \wp_f(\bbZ^d)$, choose $v \in \bbZ^d \setminus \Lambda$ and consider $H_\r \in \frA(\qty{v})$. Then
\[
\abs{\omega_{\Delta(\r)}(H_\r) - \omega_{\Delta(\s)}(H_\r)} = \abs{\psi_{\r}(H_\r) - \psi_\s(H_\r)} = \abs{1 - \r \cdot \s} > 0.
\]
If $0 < \varepsilon < \abs{1 - \r \cdot \s}$, then the above line shows that for no $\Lambda \in \wp_f(\bbZ^d)$ can we achieve 
\[
\abs{\omega_{\Delta(\r)}(A) - \omega_{\Delta(\s)}(A)} < \varepsilon \norm{A}
\]
for all $\Lambda' \in \wp_f(\bbZ^d)$ such that $\Lambda \cap \Lambda' = \varnothing$ and $A \in \frA(\Lambda')$. Therefore Theorem \ref{thm:quasilocal_superselection} implies that $\omega_{\Delta(\r)}$ and $\omega_{\Delta(\s)}$ are in different superselection sectors. 

\begin{rem}
Note that this is in contrast to the case of the interaction $\Phi$, where $\Phi$ failed to be continuous with the product topology on $X$, but continuity was restored when we composed with the diagonal map.
\end{rem}

\begin{rem}
It is perhaps worth observing that the box topology has a certain ``quasi-local'' character to it. The intuition behind the topologies used here is that the box topology is fine enough to allow only local deformations to be continuous, which correspond to norm-continuous deformations of states.  Indeed, the connected component of $\r \in X$ is the set of all $\r'$ for which $\r_v \neq \r_v'$ for only finitely many $v \in \bbZ^d$. If we had used the product topology on $X$, then we would be allowing continuous non-local deformations and $\omega$ would only be weakly$^*$ continuous. 

More generally, if $X_1,\ldots, X_n,\ldots $ is a countably infinite collection of Hausdorff, regular, path-connected spaces and $X = \prod_{n=1}^\infty X_n$ with the box topology, then for any $x = (x_n) \in X$, the set
\[
C(x) = \qty{y \in X: x_n \neq y_n \tn{ for only finitely many $n$}}
\]
is the path component of $X$ containing $x$. If we define
\[
C_n(x) = \prod_{k=1}^n X_k \times \prod_{k=n+1}^\infty \qty{x_k},
\]
then $C(x) = \bigcup_{n = 1}^\infty C_n(x)$ as sets. If, in addition, each space $X_n$ is compact, then the subspace topology on $C(x)$ coincides with the union topology induced by the $C_n(x)$.
\end{rem}

%

Finally, we show that $\omega_{\r}$ is indeed a ground state for the interaction $\Phi_\r$. The interaction defines  local Hamiltonians $H_{\r,\Lambda} = \sum_{v \in \Lambda} H_{\r_v}$, where $H_{\r_v}$ is now shorthand for the simple tensor in $\fA(\Lambda)$ with $H_{\r_v}$ in the $v$-component and the identity in every other component. The local Hamiltonians define a derivation on $\fA_\tn{loc} = \bigcup_{\Lambda \in \wp_f(\bbZ^d)} \fA(\Lambda)$ by
\[
\delta_\r(A) = i[H_{\r,\Lambda}, A] \quad \tn{ for } A \in \fA(\Lambda).
\]
The state $\omega_\r$ is a ground state for the one-parameter family of automorphisms generated by this derivation, i.e.\ the time evolution, if and only if the inequality $-i\omega_\r(A^*\delta_\r(A)) \geq 0$ is satisfied for all $A \in \fA_\tn{loc}$ \cite[Theorem 3.4.3]{NaaijkensQSSIL}.
Thus, we take $A \in \fA(\Lambda)$ for some $\Lambda \in \wp_f(\bbZ^d)$ and compute
\begin{align*}
-i\omega_\r(A^*\delta_\r(A)) &= \sum_{v \in \Lambda}\omega_{\r,\Lambda}(A^*[H_{\r_v}, A])\\
&= \sum_{v \in \Lambda} \ev{A\Omega_{\r,\Lambda}, (H_{\r_v}A - A H_{\r_v})\Omega_{\r,\Lambda}} \\
&= \sum_{v \in \Lambda} \ev{A \Omega_{\r, \Lambda}, (H_{\r_v} + I)A\Omega_{\r,\Lambda}} \geq 0.
\end{align*}
We have used the fact that $H_{\r_v}\Omega_{\r,\Lambda} = -\Omega_{\r,\Lambda}$ by definition of $\Omega_{\r, \Lambda}$ and the fact that $H_{\r_v} + I$ is a positive operator to conclude that the above it nonnegative. This verifies that $\omega_\r$ is a ground state for $\Phi_\r$.

In fact, we can show that $\omega_\r$ is the unique ground state for the interaction $\Phi_\r$. First note that for any $\Lambda \in \wp_f(\bbZ^d)$, the derivation $\delta_{\r}$ restricts to a bounded derivation $\delta_{\r, \Lambda}: \fA(\Lambda) \rightarrow \fA(\Lambda)$. If $\tilde\omega_\r:\fA \rightarrow \bbC$ is any ground state for $\Phi_\r$, then for any $\Lambda \in \wp_f(\bbZ^d)$, we have $-i\tilde\omega_{\r, \Lambda}(A^*\delta_{\r, \Lambda}(A)) \geq 0$ for all $A \in \fA(\Lambda)$, where $\tilde\omega_{\r, \Lambda} = \tilde\omega_{\r}|_{\fA(\Lambda)}$. Thus, $\tilde\omega_{\r, \Lambda}$ is a ground state for $\delta_{\r, \Lambda}$. By \cite[Thm.~5.3.37]{BraRobOAQSM2}, the set of ground states for $\delta_{\r, \Lambda}$ is a convex, weakly$^*$ compact, face of $\sS(\fA(\Lambda))$, and is therefore the closed convex hull of the pure ground states. Any pure state of $\fA(\Lambda)$ can be represented by a unit vector in $\bigotimes_{v \in \Lambda} \bbC^2$. It is therefore a finite-dimensional linear algebra problem to show that there exists a unique pure ground state $\delta_{\r, \Lambda}$, and this is indeed the case because the lowest eigenvalue of $H_{\r, \Lambda}$, which is $-\abs{\Lambda}$, has a one-dimensional eigenspace. It follows from the aforementioned result in \cite{BraRobOAQSM2} that there is a unique (not necessarily pure) ground state for $\delta_{\r, \Lambda}$. We conclude that $\omega_\r$ and $\tilde \omega_\r$ are equal for all local operators, hence equal everywhere by density of $\fA_\tn{loc}$.

\newpage

\appendix

\section{Infinite dimensional manifolds}
For the convenience of the reader we will briefly recapitulate here
some notions from infinite dimensional manifold theory. For more details see
\cite{MilRIDLG,KriMicCSGA,LangIDM,NeebTLTLCG}.

\subsection{Banach and Hilbert manifolds}
\label{sec:banach-hilbert-manifolds}
Assume that $M$ is a Hausdorff topological space.
By a $\sC^\infty$-\emph{chart} or just a \emph{chart} of $M$ \emph{modeled}
on a Banach space $E$ (over the field $\bbK$ of real or complex numbers)
one understands a pair $(U,x)$ consisting of an open subset
$U\subset M$ and a homeomorphism $x:U \to x(U)\subset E$ onto an open
subset of $E$.
Two charts $(U,x)$ and $(V,y)$ modeled on Banach spaces $E_U$ and $E_V$, respectively,
are called 
\emph{compatible} if either $U$ and $V$ are disjoint or if $E_U=E_V$ and the \emph{transition map}
\[ x\circ y|_{U\cap V}^{-1}: y(U\cap V) \to x(U\cap V)\]
is a $\sC^\infty$-diffeomorphism. 
A collection $\sA$ of pairwise compatible
charts of $M$ is called a $\sC^\infty$-\emph{atlas} or just an \emph{atlas} of $M$ if the domains of the
charts contained in $\sA$ cover $M$.
An atlas $\sA$ of $M$ is said to be \emph{maximal}
if every chart compatible with all charts in $\sA$ is an element of $\sA$.
A \emph{Banach manifold} or just a \emph{manifold} is then a Hausdorff topological space endowed
with a maximal $\sC^\infty$-atlas of charts modeled in Banach spaces. 
In case the charts of the  maximal atlas of the manifold $M$ are modeled all in Hilbert spaces, one calls
$M$ a Hilbert manifold.

In a similar fashion one defines \emph{real analytic} and \emph{complex} or \emph{holomorphic}
\emph{Banach manifolds}. These are Hausdorff topological spaces endowed with a maximal atlas of charts
so that all transition maps (and their inverses) are real or complex analytic, respectively.
An atlas of charts with biholomorphic transition maps is referred to as a \emph{holomorphic atlas}.

Given a Banach manifold $M$ and an open subset $O\subset M$ the space $\sC^\infty(O)$ of
\emph{smooth functions} on $O$ consists of all maps $g:O\to \R$ such that for each chart $(U,x)$
with $O\cap U \neq \emptyset$ the composition $g \circ x|_{O\cap U}^{-1}: x(O\cap U) \to \R$ is smooth.  
A map $f:N \to M$ between two  manifolds is called of \emph{class} $\sC^\infty$ or
\emph{smooth} if for each open $O\subset M$ and each element $g\in \sC^\infty (O)$ the pullback
$f^*g= g \circ f|_{f^{-1}(O)}$ is an element of $\sC^\infty (f^{-1}(O))$.
The spaces $\sC^\infty(O)$ give rise to a sheaf on $M$ called the \emph{sheaf of smooth functions}
on $M$. Likewise, one defines the sheaf $\sC^\omega$ of \emph{real analytic functions} and
the sheaf $\cO$ of \emph{holomorphic functions} on a real analytic respectively a holomorphic Banach manifold.

\begin{rem}\label{rem:convenient-expeonential-law}
  The notion of infinite dimensional manifolds as defined above can be extended in a natural way 
  to Fr\'echet spaces or even convenient vector spaces; see \cite{KriMicCSGA} for details.
  The virtue of using convenient vector spaces for defining manifolds lies in the fact that
  for any pair of convenient vector spaces $E,F$ the space $\sC^\infty (E,F)$ of smooth functions between
  them is again convenient \cite[1.7.\ Lemma]{KriMicATIDM} and that an exponential law holds
  for smooth mappings  \cite[3.12.\ Thm.\ \& 3.13.\ Cor.]{KriMicCSGA}. The latter means in particular that
  for $G$  a third convenient vector space the natural map
  \begin{equation}
    \label{eq:smooth-exponential-law}
    {}^{\vee}: \sC^\infty  (E\times F , G) \to \sC^\infty \left( E, \sC^\infty (F,G) \right),
    \quad f \mapsto f^{\vee}= \big( v \mapsto f(v,-)  \big)
  \end{equation}
  is a linear diffeomorphism meaning it is invertible, linear, smooth and has a smooth inverse. 
  Note that Banach and Fr\'echet spaces are convenient vector spaces, so the smooth exponential law applies
  to them.
\end{rem}

A group $G$ equipped with a manifold structure so that both multiplication and inversion are smooth maps
will be called a \emph{Lie group}. In case the Lie group $G$ is modeled on Banach or Hilbert spaces,
one calls it a \emph{Banach} or a \emph{Hilbert Lie group} respectively.

\subsection{Topological fiber bundles}
\label{sec:fiber-bundles}
We define \emph{topological fiber bundles} as in \cite[\S.~2]{SteTFB} and denote them as
quintuples $(E,B,\pi,F,G)$ where $E$ is the total space, $B$ the base space, $\pi:E\to B$ the projection,
$F$ the typical fiber, and $G$ the structure group. The latter  is assumed to be a topological group acting
continuously and effectively on $F$. 
In case the typical fiber $F$ and the structure group $G$ are clear
from context, we usually denote a fiber bundle just by its projection $\pi:E\to B$. Attached to a fiber
bundle  $(E,B,\pi,F,G)$ is a maximal trivializing atlas.
Such a  \emph{trivializing atlas} consists of pairwise compatible \emph{local trivializations} which are 
homeomorphisms $\varphi : \pi^{-1} (O) \to O \times F$ with $O\subset B$  open such that
$\operatorname{pr}_O \circ \varphi = \pi|_{\pi^{-1} (O)}$ and whose domains are assumed to cover the base $B$.  Two local trivializations
$\varphi_i : \pi^{-1} (O_i) \to O_i \times F$ and $\varphi_j : \pi^{-1} (O_j) \to O_j \times F$
are hereby called \emph{compatible} if the map
\[
 \varphi_{i,\bullet} \circ  \varphi_{j,\bullet} ^{-1}: \: O_i \cap O_j \to \sC (F,F), \quad p \mapsto\varphi_i \circ  \varphi_j^{-1}  (p,-) 
\]
factors through a continuous map $g_{ij}: O_i \cap O_j \to G$ called \emph{transition function}.
In other words this means that the diagram
\[
  \begin{tikzcd}
  O_i \cap O_j \arrow{r}{g_{ij}}  \arrow[swap]{rd}{\varphi_{i,\bullet} \circ  \varphi_{j,\bullet}^{-1}} 
  &  G \arrow{d} \\
    & \sC(F,F)
  \end{tikzcd}
\]
commutes, where $G \hookrightarrow \sC (F,F)$ is the canonical continuous injection and
$\sC (F,F)$ carries the topology of pointwise convergence. Note that $ g_{ij}$ is uniquely
determined since $G$ acts effectively on $E$. 

For convenience, we sometimes write $(\varphi,O)$ for
a local trivialization of the form $\varphi : \pi^{-1} (O) \to O \times F$. A local
trivialization $(\varphi,O)$ defined over some open neighborhood $O$ of a point $p\in B$  gives rise
to a homeomorphism $\varphi_p: E_p \to F$  by putting $\varphi_p (e) = \varphi (p,e)$ for all $e\in E_p$,
where as usual $E_p$ denotes the  fiber $\pi^{-1}(p)$ over $p$.
If the fiber bundle has a \emph{global trivialization}, that is, a trivialization of the form  
$\varphi : E \to B \times F$, the fiber bundle is called \emph{trivial}.  
Instead of \emph{fiber bundle} we therefore sometimes say \emph{locally trivial bundle}.
Note that if the group $G$ is contractible, then a fiber bundle with structure group $G$ is
necessarily trivial, see \cite{SteTFB}. More generally,  the family of transition functions
$(g_{ij})_{i,j\in I}$ associated to a trivializing atlas $(\varphi_i,O_i)_{i\in I}$ of the fiber bundle
defines a topological invariant called the \emph{\v{C}ech cohomology} of the bundle since by construction
the family $(g_{ij})_{i,j\in I}$  satisfies the \v{C}ech cocycle conditions
\[
  g_{ij}\cdot g_{jk} = g_{ij} \quad\text{over } O_{ijk} = O_i\cap O_j\cap O_k \text{ for all } i,j,k\in I \ .
\]
The associated \v{C}ech cohomology class in $\check{H}^1(B, G) = \mathrm{colim}_{\mathcal{U}} \check{H}^1(\mathcal{U}, G)$,
where $\mathcal{U}$ runs through the open covers of $B$, does not depend on the
chosen trivializing atlas, hence is a topological invariant indeed. If the underlying bundle is a $G$-principal bundle,
its \v{C}ech cohomology determines the bundle up to isomorphism.
For more details on the \v{C}ech cohomology of (principal) fiber bundles see \cite[Chpt.~4]{BrylinskiLoopSpaces} and
\cite[Sec.~25.8]{HusemollerJoachimJurcoSchottenloher}.

In case $H \hookrightarrow G$ is a continuous injective homomorphism between
topological groups one says that the structure group of a fiber bundle  $(E,B,\pi,F,G)$ can be \emph{reduced} to $H$
if there exists a trivializing atlas such that for  any pair  of local trivializations
$(\varphi_i,O_i)$ and $(\varphi_j,O_j)$ in this atlas the transition function  $g_{ij}: O_i\cap O_j \to G$ factors through a
continuous map $h_{ij}:O_i\cap O_j \to H$.

Under the assumption that the typical fiber $F$ is a Banach space, and that $G$ coincides with the group
$\Aut(F)_{\textup{s}}$ of isometric automorphism  of $F$ endowed with the strong topology,
a fiber bundle of the form $(E,B,\pi,F,G)$ just corresponds to what is usually called
a \emph{Banach vector bundle}  or \emph{locally trivial Banach bundle} with typical fiber $F$.
Note that $\Aut(F)_{\textup{s}}$ is a topological group, cf.~\cite{SchottenloherUnitaryStrongTopology}. The fiberwise homeomorphisms $\varphi_p: E_p \to F$ then endow each $E_p$ with the structure of a Banach space which is independent of the chosen local trivialization $\varphi$
around $p$. Following \cite{SchottenloherUnitaryStrongTopology}, a Banach fiber bundle is called \emph{norm defined}
if its structure group can be reduced to the group $\Aut(F)_{\textup{n}}$ of isometric automorphisms
endowed with the  norm topology. Still under the assumption that $F$ is a Banach space, a vector bundle $(E,B,\pi,F,G)$
is called \emph{banachable} whenever the structure group $G$ is the topological group $\GL (F)_{\textup{n}}$ of topological
linear isomorphisms of the Banach space $F$. Note that when endowed with the norm topology, $\GL (F)$ becomes a topological group, but
in general not when endowed with the strong topology and $F$ is infinite dimensional. Therefore, a banachable vector bundle is always
\emph{norm defined} by definition. 

In case the typical fiber is a Hilbert space $\cH$, the natural structure
group $G$ is the unitary group $\Unitary (\cH)_\textup{s}$ with the strong topology. A corresponding fiber bundle
will then be referred to as a \emph{Hilbert vector bundle}. 
We also say that a locally trivial Hilbert bundle is \emph{norm defined} if the structure group can be reduced to the
unitary group $\Unitary (\cH)_\textup{n}$ with the norm topology.
By Kuiper's theorem, the unitary group of an infinite dimensional separable Hilbert space is contractible,
both in the strong and norm topologies, so a Hilbert fiber bundle with an infinite dimensional separable typical fiber
has to be trivial, even when norm defined \cite[Sec.~4]{SchottenloherUnitaryStrongTopology}.

In case the typical fiber $F$ carries some additional structure like the structure of a
Banach algebra or a $C^*$-algebra and the group $G$ is the topological group of automorphisms $F$
endowed with the strong topology, a fiber bundle of the form
$(E,B,\pi,F,G)$  is called a \emph{Banach algebra fiber bundle} or a
\emph{$C^*$-algebra fiber bundle}, respectively. It is called \emph{norm defined} whenever
the structure group $G$ can be chosen to be the automorphism group endowed with the norm
topology. In other words, this means that there exists a trivializing atlas such that
all transition maps are norm continuous. Note that the fibers of a Banach or $C^*$-algebra fiber bundle carry
in a canonical way the structure of a Banach or $C^*$-algebra, respectively.

By  a \emph{subbundle} of a fiber bundle $(E,B,\pi,F,G)$ we mean a fiber bundle of the form
$(\widetilde{E},B,\pi|_{\widetilde{E}},S,H)$ endowed with a
trivializing atlas $\sT$ so that $S \subset F$ is a subspace, $\widetilde{E}$ is a subspace 
of $E$, the restriction $\pi|_{\widetilde{E}}:\widetilde{E}\to B$ is surjective, the group  $H$ of all elements
of $G$ leaving $S$ invariant is a closed subgroup, and for each element
$(\widetilde{\varphi},O) \in \sT$ there is a local trivialization $(\varphi,O)$ of
$(E,B,\pi,F,G)$ such that $\widetilde{\varphi}$ coincides with the restriction of $\varphi$ to
$\widetilde{E}_O =\pi^{-1}(O)\cap \widetilde{E}$. As a slight generalization one sometimes
allows $H$ to be a topological group for which there exists a distinguished injective morphism of topological groups
$H\to G$. 
Note that we have not required in our definition that $S$ is complementable in $F$.
This differs from some definitions, for example, that of \cite[Ch.\ 3, \S 3]{LangIDM}.
The exception are Hilbert bundles, where our definition coincides with the one of \cite{LangIDM}
since every closed subspace of a Hilbert space has a closed complement.

\subsection{Smooth fiber bundles}
\label{sec:smooth-fiber-bundles}
Next, we consider the smooth case. Assume that $E$, $B$ and $F$ are smooth manifolds and that $G$ is a subgroup of the
diffeomorphism group of $F$. Assume further that $G$ is endowed with a Lie group structure so that the canonical map
$G\to \sC^\infty (F,F)$ is smooth.
By the smooth exponential law from Remark \ref{rem:convenient-expeonential-law} this is equivalent to the map
$G\times F \to F$, $(g,v) \mapsto gv$ being smooth. 
By a \emph{smooth structure} on the fiber bundle $(E,B,\pi,F,G)$ with smooth projection $\pi:E\to B$ 
we understand a maximal collection of smooth local trivializations $\varphi  : \pi^{-1} (O) \to O \times F$
so that their domains $O$ cover $B$ and so that the transition function
$g_{ij}= \varphi_{i,\bullet} \circ \varphi_{j,\bullet}^{-1} : O_i\cap O_j \to G$ is smooth
for any pair of smooth local trivializations $(\varphi_i,O_i)$ and $(\varphi_j,O_j)$. A fiber bundle endowed with a 
smooth structure is called a \emph{smooth fiber bundle}. 

Next assume that $F$ is a Banach space and that $G$ is the group $\Aut (F)_{\textup {n}}$ of isometric automorphisms
endowed with the norm topology. This group inherits a canonical Lie group structure
as a closed subgroup of the general linear group $\GL (\fB(F))$ 
which is a Banach Lie group since  $\GL (\fB(F))$
is an open subset of the Banach space $\fB(F)$ of bounded linear endomorphisms of $F$. 
A  \emph{smooth Banach fiber bundle} now is a smooth fiber bundle with structure group
$\Aut (F)_{\textup {n}}$.

If a Banach or $C^*$-algebra fiber bundle $(E,B,\pi,F,G)$
carries a smooth structure so that its transition maps are smooth into the structure group
$G = \Aut (F)_{\textup{n}}$ endowed with the canonical Lie group structure, we call $(E,B,\pi,F,G)$
a smooth Banach or $C^*$-algebra fiber bundle, respectively.

In case all data are smooth, $S$ is a submanifold of $F$, if $H$ is a Lie subgroup of $G$ or more generally 
if there is a distinguished injective smooth group homomorphism $H\hookrightarrow G$, and finally if the subbundle atlas $\sT$
has smooth transition maps, one calls $(\widetilde{E},B,\pi|_{\widetilde{E}},S,H)$ a \emph{smooth subbundle} 
of $(E,B,\pi,F,G)$.

As an example we briefly discuss the construction of tangent, cotangent and tensor bundles in the
described setting.

\begin{ex}\label{ex:tangent-tensor-bundles}
  Let $E$ be a Banach space. Denote by 
  \[
    \frT^r_s E = \frL ( \underbrace{E^\prime, \ldots ,E^\prime}_{r\text{-times}},  \underbrace{E, \ldots , E}_{s\text{-times}}; \R)
  \]  
  the space of continuous real valued functionals
  $r$-times multilinear in the topological dual $E^\prime = \frL(E;\R)$ and
  $s$-times multilinear in $E$ and call it the space of
  \emph{$r$-fold contravariant} and \emph{$s$-fold covariant} \emph{tensors}
  on $E$. The space $\frT^r_s E$ then becomes a Banach space with norm
  induced by the norm on $E$.  

  Now assume to be given a Banach manifold $M$ modeled on $E$. Let $\sA$ be an
  atlas of $M$ and consider the disjoint unions
  \begin{equation*}
    \begin{split}
      \widetilde{T}M & = \bigsqcup_{(x,U)\in \sA} U \times E  =
      \bigcup_{(x,U)\in \sA} U \times E \times \{x \}\ , \\
      \widetilde{T^*}M & = \bigsqcup_{(x,U)\in \sA} U \times E^\prime
      = \bigcup_{(x,U)\in \sA} U \times E^\prime \times  \{x \}\ , \\
      \widetilde{T^r_s}M & = \bigsqcup_{(x,U)\in \sA} U \times \frT^r_s E
      = \bigcup_{(x,U)\in \sA}   U \times \frT^r_s E \times  \{x \} \ .
    \end{split}
  \end{equation*}
  Given charts $(x,U), (y,V) \in \sA$ call elements $(p,v)_x \in U \times E$
  and $ (q,w)_y\in \times V \times E$ \emph{equivalent}, in signs $(p,v)_x\sim (q,w)_y$, if $p =q$
  and $v = D(x \circ y^{-1})(y(q))(w)$.  Note that hereby we have written $(p,v)_x$ instead of $(p,v)$
  - and likewise for $(q,w)_y$ - to denote that the pair $(p,v)_x$ is actually regarded as lying in
  $U \times E \times \{x \}$.
  The quotient space with respect to the equivalence relation $\sim$
  will be denoted $TM$ and endowed with the unique topology such that for each chart $(x,U)\in \sA$
  the subspace $TU = U\times E \times \{ x\}/\sim$ 
  is open and such that the canonical mapping $U\times E \to TM$, $(p,v)\mapsto [(p,v)_x]$ is a homeomorphism onto $TU$.
  Denote by $(Tx,TU) :TU \to U\times E$ the chart which maps the equivalence class 
  $[(p,v)_x]$ to the pair $(p,v)$. By construction, $Tx$ then is a homeomorphism, and for each other chart $(y,V)\in \sA$
  the transition map
  \begin{equation*}
    \begin{split}
      Tx\circ Ty|_{TU\cap TV}^{-1}  :\: & Ty(TU\cap TV) \to  Tx(TU\cap TV), \\
      & \big( y(q),w\big) \mapsto  \big( x(q),D(x\circ y^{-1})(y(q))(w)\big) 
     \end{split}
  \end{equation*}
  is smooth. This endows $TM$ with the structure of a smooth vector bundle modeled on $E\times E$. 
   
  In a somewhat more technical but analogous way one constructs the tensor bundle $T^r_s M$.
  Given again charts $(x,U), (y,V) \in \sA$ one calls elements $(p,\lambda)_x \in U \times \frT^r_s E $
  and $ (q,\mu)_y\in \times V \times \frT^r_s E $ \emph{equivalent}, in signs $(p,\lambda)_x\sim (q,\mu)_y$, if $p =q$
  and
  \[
    \lambda = \mu \circ \left(
    \underbrace{T_{q}(x \circ y^{-1})^* \times \ldots \times  T_{q}(x \circ y^{-1})^*}_{r\text{-times}}
    \times
    \underbrace{T_{p}(y \circ x^{-1}) \times \ldots \times T_{p}(y \circ x^{-1})}_{s\text{-times}}\right) \ ,
  \] 
  where $T_{p}(y \circ x^{-1})$ stands for the linear map $D(y \circ x^{-1})(x(p))$ and  $T_{q}(x \circ y^{-1})^*$
  is the pullback by  $D(x\circ y^{-1})(y(q))$. Analogously as before $T^r_sM$ is now defined as the quotient space
  $\widetilde{T^r_s}M/\sim$ and given the unique topology such that all for $(x,U)\in\sA$  the set 
  $T^r_sU = U \times \frT^r_s E \times \{x\}/\sim$ is open and the map $U \times \frT^r_s E \to T^r_sM$,
  $(p,\lambda) \mapsto [(p,\lambda)_x]$ is a homeomorphism. The maps 
  \[
    T^r_s x : T^r_s U \to U \times \frT^r_s E, \: [(p,\lambda)_x] \mapsto (p,\lambda)
  \]  
  then form an atlas of $T^r_s M$ turning it into a smooth vector bundle modeled on $E\times \frT^r_s E$.
  It is called the \emph{tensor bundle} of \emph{$r$-fold contravariant} and \emph{$s$-fold covariant tensors} on $M$. 
  In case the modeling Banach space is even a Hilbert space, one can express the typical fiber of the tensor bundle $T^r_s M$ 
  as the completed tensor product $E^{\hat{\otimes}r}\hat{\otimes} {E^\prime}^{\hat{\otimes}s}$,
  which explains the term \emph{tensor bundle}.  
\end{ex}

\bibliographystyle{amsalpha}
\bibliography{AQMref,lmlib}

\end{document}